\documentclass[11pt, letterpaper]{amsart}
\usepackage{amssymb}

\pdfoutput=1

\usepackage[utf8]{inputenc}
\usepackage[T1]{fontenc}
\usepackage[backend=biber,
            isbn=false,
            doi=false,
            maxbibnames=9,
            giveninits=true,
            style=alphabetic,
            citestyle=alphabetic]{biblatex}
\usepackage{url}
\renewbibmacro{in:}{}
\bibliography{skeins.bib}

\usepackage{xcolor}
\definecolor{e-mail}{rgb}{0,.40,.80}
\definecolor{reference}{rgb}{.20,.60,.22}
\definecolor{citation}{rgb}{0,.40,.80}
\usepackage[colorlinks=true,
            linkcolor=reference,
            citecolor=citation,
            urlcolor=e-mail]{hyperref}

\usepackage[all]{xy}
\usepackage[capitalise]{cleveref}
\usepackage{eucal}
\usepackage{mathrsfs}
\usepackage{tikz-cd}
\usepackage{adjustbox}

\usepackage{todonotes,setspace}

\newtheorem{maintheorem}{Theorem}

\newtheorem{mainconjecture}[maintheorem]{Conjecture}
\newtheorem{thm}{Theorem}[section]
\newtheorem{prop}[thm]{Proposition}

\newtheorem{lm}[thm]{Lemma}
\crefname{lm}{Lemma}{Lemmas}
\newtheorem{cor}[thm]{Corollary}

\newtheorem*{p*}{Proposition}

\theoremstyle{definition}
\newtheorem{defn}[thm]{Definition}
\newtheorem{example}[thm]{Example}

\theoremstyle{remark}
\newtheorem{remark}[thm]{Remark}

\newcommand{\defterm}[1]{\textbf{\emph{#1}}}

\newcommand{\cA}{\mathcal A}
\newcommand{\cB}{\mathcal B}
\newcommand{\cC}{\mathcal C}
\newcommand{\cD}{\mathcal D}
\newcommand{\cE}{\mathcal E}
\newcommand{\cF}{\mathcal F}
\newcommand{\cG}{\mathcal G}
\newcommand{\cH}{\mathcal H}

\newcommand{\cK}{\mathcal K}
\newcommand{\cL}{\mathcal L}
\newcommand{\cM}{\mathcal M}

\newcommand{\cO}{\mathcal O}

\newcommand{\cV}{\mathcal V}
\newcommand{\cW}{\mathcal W}

\newcommand{\bD}{\mathbf{D}}

\newcommand{\C}{\mathbf C}
\newcommand{\D}{\mathbb D}

\newcommand{\bL}{\mathbb L}

\newcommand{\Q}{\mathbb Q}
\newcommand{\R}{\mathbf R}

\newcommand{\Z}{\mathbf Z}

\newcommand{\hE}{\widehat{\mathcal{E}}}
\newcommand{\hW}{\widehat{\mathcal{W}}}
\newcommand{\sfB}{\widehat{\mathsf{B}}}
\newcommand{\sfC}{\mathsf{C}}
\newcommand{\sfE}{\mathsf{E}}
\newcommand{\sfhE}{\widehat{\mathsf{E}}}
\newcommand{\sfM}{\mathsf{M}}
\newcommand{\sfDRH}{\mathbf{D}\mathsf{RH}}
\newcommand{\sfRH}{\mathsf{RH}}
\newcommand{\sfhW}{\widehat{\mathsf{W}}}
\newcommand{\oT}{\mathring{\mathrm{T}}}

\newcommand{\rH}{\mathrm{H}}
\newcommand{\rJ}{\mathrm{J}}

\renewcommand{\P}{\mathrm{P}}
\newcommand{\T}{\mathrm{T}}

\newcommand{\Aex}{\mathcal{A}^{2, \mathrm{ex}}}

\newcommand{\aut}{\mathrm{aut}}
\newcommand{\codim}{\operatorname{codim}}

\newcommand{\colim}{\operatorname*{colim}}
\newcommand{\contact}{\mathrm{contact}}
\newcommand{\conv}{\mathrm{conv}}
\newcommand{\Crit}{\mathrm{Crit}}

\newcommand{\DR}{\mathrm{DR}}
\newcommand{\End}{\mathrm{End}}
\newcommand{\Fun}{\mathrm{Fun}}

\DeclareMathOperator{\gr}{gr}

\newcommand{\id}{\mathrm{id}}

\newcommand{\Lag}{\mathrm{Lag}}

\newcommand{\LocSys}{\mathrm{LocSys}}
\newcommand{\locsys}{\mathrm{lisse}}
\newcommand{\Mod}{\mathrm{Mod}}
\newcommand{\mon}{\mathrm{mon}}
\newcommand{\muPerv}{\mathbb{P}\mathrm{erv}}
\newcommand{\muRH}{\mu\mathrm{RH}}
\newcommand{\mush}{\mu\mathrm{sh}}

\newcommand{\ord}{\mathrm{ord}}
\newcommand{\op}{\mathrm{op}}
\newcommand{\Perv}{\mathrm{Perv}}
\newcommand{\Pervaut}{\mathrm{Perv}_\aut}
\newcommand{\Pervdiff}{\mathrm{Perv}_{\mathrm{diff}}}
\newcommand{\pt}{\mathrm{pt}}
\newcommand{\PV}{\mathcal{PV}}
\renewcommand{\Re}{\operatorname{Re}}
\newcommand{\red}{\mathrm{red}}
\newcommand{\reg}{\mathrm{reg}}
\newcommand{\rR}{\mathrm{R}}
\newcommand{\RH}{\mathrm{RH}}
\newcommand{\rh}{\mathrm{rh}}
\newcommand{\Sing}{\operatorname{Sing}}
\newcommand{\Skint}{\mathrm{Sk}^{\mathrm{int}}}
\newcommand{\Skgen}{\mathrm{Sk}^{\mathrm{gen}}}
\newcommand{\SkAlgint}{\mathrm{SkAlg}^{\mathrm{int}}}
\newcommand{\SL}{\mathrm{SL}}

\newcommand{\supp}{\operatorname{supp}}
\newcommand{\TS}{\mathrm{TS}}

\newcommand{\vir}{\mathrm{vir}}
\newcommand{\vol}{\mathrm{vol}}

\newcommand{\Rep}{\mathrm{Rep}}

\newcommand{\ham}{/\!/}
\newcommand{\Hom}{\mathrm{Hom}}
\newcommand{\cHom}{\mathcal{H}om}
\newcommand{\RcHom}{\mathbb{R}\cHom}

\newcommand{\Cbb}{\C[\![\hbar]\!]}
\newcommand{\Cpp}{\C(\!(\hbar)\!)}

\newcommand{\boxast}{\stackrel{\ast}{\boxtimes}}
\newcommand{\boxD}{\stackrel{\cD}{\boxtimes}}
\newcommand{\boxE}{\stackrel{\cE}{\boxtimes}}
\newcommand{\timesast}{\stackrel{\ast}{\times}}

\newcommand{\EigenRange}{\mathbf{G}}

\begin{document}

\title{Deformation quantization and perverse sheaves}
\author{Sam Gunningham}
\address{Department of Mathematical Sciences, Montana State University, Bozeman, MT, USA}
\email{sam.gunningham@montana.edu}
\author{Pavel Safronov}
\address{School of Mathematics, University of Edinburgh, Edinburgh, UK}
\email{p.safronov@ed.ac.uk}
\begin{abstract}
Kashiwara, Polesello, Schapira and D'Agnolo defined canonical deformation quantizations of a holomorphic symplectic manifold and a holomorphic Lagrangian submanifold equipped with an orientation data. The goal of this paper is to use deformation quantization modules to construct a Fukaya-like category of holomorphic Lagrangians, resolving a conjecture of Joyce. Our main result describes the RHom complex between two such deformation quantization modules associated to a pair of Lagrangian submanifolds in terms of the derived geometry of the Lagrangian intersection. Namely, we identify the RHom complex with the DT sheaf associated to the d-critical structure on the Lagrangian intersection. Via the Riemann--Hilbert correspondence this describes the RHom complex in the category of microsheaves of sheaf quantizations of conic holomorphic Lagrangians.
\end{abstract}

\maketitle

\section*{Introduction}

	Suppose we are given a holomorphic symplectic manifold $S$ and a pair of Lagrangian submanifolds $L,M$ each equipped with a choice of square root of the canonical bundle. Associated to this data we have the following two perverse sheaves on the intersection $L\cap M$:
	\begin{itemize}
		\item By \cite{Bussi,BBDJS}, the intersection $L\cap M$ admits the structure of a $d$-critical locus, and thus carries a canonical $\C$-linear perverse sheaf $\phi_{L\cap M}$ (called the DT sheaf), equipped with a canonical automorphism $T_{L\cap M}$ (called the monodromy), obtained by gluing together vanishing cycle sheaves on local charts.
		\item By \cite{PoleselloSchapira, DAgnoloSchapira, KashiwaraSchapira}, there is a canonical $\Cpp$-algebroid $\sfhW_S$ together with a right (respectively, left) module $\sfM_L$ (respectively, $\sfM_M$), such that the derived relative tensor product $\sfM_L \otimes^{\bL}_{\sfhW_S} \sfM_M$ 
        is a $\Cpp$-linear perverse sheaf supported on $L\cap M$. 
	\end{itemize}
	
	Our primary goal of this paper is to compare these two perverse sheaves. To put the two sides on the same footing, we use a version of the Riemann--Hilbert correspondence in the $\hbar$ variable to reinterpret the pair $(\phi_{L\cap M},T_{L\cap M})$
 as a \emph{differential perverse sheaf} $\RH^{-1}(\phi_{L\cap M},T_{L\cap M})$ (informally, a $\Cpp$-linear perverse sheaf equipped with a regular $\hbar$-connection). 
	\begin{maintheorem}[\cref{thm:maincomparison}]
 \label{mainthm:comparison}
		There is a natural differential structure on the perverse sheaf $\sfM_L \otimes^L_{\sfhW_S} \sfM_M$, and an isomorphism of differential perverse sheaves
		\[
		\sfM_L\otimes^\bL_{\sfhW_S} \sfM_M\cong \RH^{-1}(\phi_{L\cap M}, T_{L\cap M}).
		\]
	\end{maintheorem}
In fact, the modules $\sfM_L$ and $\sfM_M$ appearing in \cref{mainthm:comparison} are naturally Verdier self-dual, so that \cref{mainthm:comparison} also provides an isomorphism:
\[
  \sfM_L \otimes^{\bL}_{\sfhW_S} \sfM_M \cong \RcHom_{\sfhW_S}(\sfM_L,\sfM_M)[\dim S] \cong \phi_{L\cap M} \otimes_\C \Cpp
\]
    
  While the existence of some kind of comparison between deformation quantization modules and the DT sheaf has been expected for some time in the community (see \cite[Section 4]{Bussi}), \cref{mainthm:comparison} is the first proof of such a result, and also the first statement that properly accounts for the monodromy and discrepancy of coefficient fields. As we explain below, the extra structure corresponding to the monodromy appears naturally in our approach and is a necessary ingredient in our proof of any such comparison. 

\cref{mainthm:comparison} has a number of consequences and potential applications, which we explore in the rest of this introduction: to skein modules, microlocal sheaves and Fukaya categories.

\subsection*{A local Fukaya category}
One source of motivation for our work is a proposal for a sheaf model of a Fukaya category associated to a holomorphic symplectic manifold $S$:
	\begin{itemize}
		\item Holomorphic Lagrangian submanifolds $L\subset S$ equipped with square roots $K_L^{1/2}$ should define objects.
		\item The graded Hom space between objects associated to oriented Lagrangians $L, M$ is given by the hypercohomology of $\phi_{L\cap M}$.
		\item Composition of Homs is given by a conjectural functoriality of the DT sheaf, see \cite[Conjecture 1.1]{JoyceSafronov}.
	\end{itemize}
    We refer to \cite{BehrendFantechi} and \cite[Remark 6.15]{BBDJS} as well as the more recent \cite[Section 6.3]{KontsevichSoibelman} where it is called the ``local Fukaya category''. We refer to \cite{Mladenov} for the study of Homs in this category between two compact K\"ahler Lagrangians. 
    
    Our results provide a solution to this problem. Namely, we construct a certain sheaf of $\infty$-categories $\sfDRH_S$, enriched over the category $\Rep(\mathbb Z)$ of complexes of $\C$-vector spaces with an automorphism, in which the Hom complexes are identified with the derived global sections of the DT sheaf. 
    The sheaf of $\infty$-categories $\sfDRH_S$ is a derived and enriched version of the sheaf of abelian categories $\sfRH_S$ constructed by D'Agnolo and Kashiwara \cite{DAgnoloKashiwara}. Here, the letters $\sfRH$ stand for ``regular holonomic'', as the category is constructed from regular holonomic deformation quantization modules. The composition structure is inherited from the category of deformation quantization modules (in particular, our work does not address the expected functoriality of the DT sheaf in full generality). 
    
    More precisely, we have the following result.

	\begin{maintheorem}[\cref{thm:RHproperties}]\label{mainthm:category}
		Let $S$ be a holomorphic symplectic manifold. There is a sheaf $\sfDRH_S$ of $\Rep(\Z)$-linear $\infty$-categories on $S$, equipped with a $t$-structure, which satisfies the following properties:
  \begin{itemize}
      \item For each Lagrangian submanifold $L$ of $S$ equipped with orientation data, there is a canonical object $\sfM_L$.
      \item Given a pair of oriented Lagrangian submanifolds $L,M$, there is an isomorphism
				\begin{equation}\label{eq:LagrangianHomequivalence}
                \underline{\cHom}_{\sfDRH_S}(\sfM_M, \sfM_L)[\dim S/2]\cong (\phi_{L\cap M},T_{L\cap M})
                \end{equation}
				of perverse sheaves of $\C$-vector spaces equipped with an automorphism.
  \end{itemize}
	\end{maintheorem}

    We work over $\C$ as we rely on the machinery of microdifferential modules to construct the isomorphism \eqref{eq:LagrangianHomequivalence}. One can similarly construct the $\infty$-category $\sfDRH_S$ over $\Z$ by replacing categories of microdifferential modules \cite{SatoKawaiKashiwara,KashiwaraContact} (defined over $\C$) with categories of perverse microsheaves \cite{Waschkies1,CKNS1} (defined over any ring); we similarly expect the isomorphism \eqref{eq:LagrangianHomequivalence} to hold over any ring.

\subsection*{Skein modules and perverse sheaves}
Another motivation for this work is a conjectural relationship between skein modules and the DT sheaf formulated in \cite{GJS}. Let us recall the statement. Let $M$ be a connected closed oriented 3-manifold with a basepoint $x\in M$ and $G$ a reductive algebraic group equipped with some extra data described in \cite[Section 1.3]{GJS}. Then one can define the skein module $\Skgen_G(M)$ of $M$ for generic quantum parameters, which is a $\C(q^{1/d})$-vector space for some integer $d$. For instance, for $G=\SL_2$ we obtain the Kauffman bracket skein module, which is a $\C(A)$-vector space for $A=q^{1/2}$. The representation variety \[R_G(M) = \Hom(\pi_1(M, x), G)\] carries a natural structure of a d-critical locus. In the case $G=\SL_2$ Abouzaid and Manolescu \cite{AbouzaidManolescu} denote the hypercohomology of the DT sheaf $\phi_{R_G(M)}$ on $R_G(M)$ by $\mathrm{HP}^\bullet_\sharp(M)$, the \emph{framed complexified instanton Floer homology} of $M$. The conjectural relationship between the two objects is as follows.

\begin{mainconjecture}
There is an isomorphism of $\C(q^{1/d})$-vector spaces
\[\Skgen_G(M)\cong \rH^0(R_G(M), \phi_{R_G(M)})\otimes_\C\C(q^{1/d}).\]
\label{conj:skein}
\end{mainconjecture}

We refer to \cite{DKS} for some checks of this conjecture for the Kauffman bracket skein module. The above conjecture is related to \cref{mainthm:comparison} as follows. Choose a Heegaard splitting $M=N_1\cup_\Sigma N_2$ of $M$ as a union of two handlebodies $N_1,N_2$ along the common boundary $\Sigma$. In this case we have the following:
\begin{itemize}
    \item The representation variety $R_G(M)$ becomes an intersection $L_1\cap L_2$ of two complex algebraic Lagrangian submanifolds $L_1, L_2$ (representation varieties of the handlebodies $N_1,N_2$) in a complex symplectic variety $S$ (an open subset of the representation variety of the open surface $\Sigma^* = \Sigma\setminus D$). We refer to \cite[Section 5.3]{GJS} for details.
    \item The generic skein module $\Skgen_G(M)\otimes_{\C(q^{1/d})} \C(\!(\hbar)\!)$ defined with respect to the formal variable $\hbar$ with $q=\exp(\hbar)$ is a relative tensor product
    \[\Skgen_G(M)\otimes_{\C(q^{1/d})} \C(\!(\hbar)\!)\cong \Skint_G(N_1)\otimes_{\SkAlgint_G(\Sigma^*)} \Skint_G(N_2)\]
    of \emph{algebraic} deformation quantization modules $\Skint_G(N_i)$ quantizing $L_i\subset S$ over the deformation quantization algebra $\SkAlgint_G(\Sigma^*)$ providing an algebraic deformation quantization of the symplectic variety $S$.
\end{itemize}
Thus, in order to prove \cref{conj:skein} using \cref{mainthm:comparison}, one must in particular relate the tensor product of the algebraic deformation quantization modules provided by skein theory to the corresponding tensor product of the canonical analytic deformation quantization modules. We plan to address this problem in future work.

\subsection*{Vanishing cycles and the twisted de Rham complex}
Consider the case when $S=\T^\ast X$ is the cotangent bundle of a complex manifold $X$, $L=\Gamma_0$ is the zero-section and $M=\Gamma_{df}$ is the graph of $df$, where $f\colon X\to \C$ is a holomorphic function. This can be considered as a local model for the setup of \cref{mainthm:comparison}, in the sense that every holomorphic Lagrangian intersection is locally isomorphic to such an intersection. In this case, the DT sheaf $\phi_{\Gamma_0 \cap \Gamma_{df}}$ is identified with the sheaf of vanishing cycles $\phi_f$.

The format of the statement of \cref{mainthm:comparison} is inspired by a related result \cite{SabbahSaito} of Sabbah and Saito (see also \cite{Sabbah}) relating the twisted de Rham complex
\[\DR_f=(\Omega^\bullet_U(\!(\hbar)\!), d + \hbar^{-1} df)|_{f^{-1}(0)}\] and the sheaf of vanishing cycles $\phi_f$. (Note that we restrict to the zero set of $f$ to simplify the presentation.) Namely, there is a natural connection on $\DR_f$ in the $\hbar$ direction given by $\partial_\hbar - \hbar^{-2}f$, which equips $\DR_f$ with the structure of a differential perverse sheaf. Sabbah and Saito \cite{SabbahSaito} identify the cohomology sheaves $\cH^k \DR_f$ and $\cH^k \RH^{-1}(\phi_f)$. We show that their isomorphism lifts to an isomorphism of differential perverse sheaves.

\begin{maintheorem}[\cref{cor:DRphicomparison}]
Let $f\colon U\rightarrow \C$ be a holomorphic function on a complex manifold $U$. Then there is an isomorphism of differential perverse sheaves
\[(\DR_f, \hbar\partial_\hbar - \hbar^{-1}f) \cong \RH^{-1}(\phi_f).\]
\end{maintheorem}

A closely related result has recently been shown by Schefers in \cite[Proposition 9.20]{Schefers} using different methods.
 
\subsection*{Contactification and symplectization}
The results and techniques in this paper are rooted in the tight interplay between symplectic and contact geometry in the holomorphic setting. Understanding this interplay also leads to a conceptual explanation for the structure appearing in the holomorphic Fukaya category of \cref{mainthm:category}.

Given a pair $(S,L)$ of a symplectic manifold together with a Lagrangian subvariety (for example, a finite union of Lagrangian submanifolds), a contactification of $(S,L)$ is a pair $(Y, \rho, \Lambda)$ where $\rho:Y \to S$ is a principal $\C$-bundle with connection whose curvature is the symplectic form, $\Lambda$ is a Legendrian subvariety of $Y$ (with respect to the canonical contact structure), and $\rho|_\Lambda$ is a homeomorphism onto $L$, analytic on the smooth part of $\Lambda$.

Given an arbitrary contact manifold $Y$, there is a canonical $\C^\times$-principal bundle $\gamma: \widetilde{Y}\to Y$ called the \emph{symplectization} whose total space carries a canonical symplectic structure. In the case when $\rho:Y\to S$ is the contactification of a symplectic manifold $S$, the symplectization is canonically trivial: $\widetilde{Y} = Y\times \C^\times$. Moreover, if $L$ is a Lagrangian subvariety of $S$, and $(Y,\Lambda)$ a contactification of $(S,L)$, then $\widetilde{\Lambda} = \gamma^{-1}(\Lambda) \cong \Lambda \times \C^\times$ is an exact Lagrangian subvariety in $\widetilde{Y}$.

Note that in our main results we do not assume that our symplectic manifold $S$ is exact. In particular, there may not exist a global contactification $\rho:Y\to S$. Nevertheless, we have the following fundamental result (see \cite[Proposition 2.5.3]{DAgnoloKashiwara}).
\begin{p*}[Unique Contactification Lemma, \cref{prop:uniquecontactification}]
Any pair $(S,L)$ of a symplectic manifold and Lagrangian subvariety admits a contactification $(Y,\rho, \Lambda)$ over a neighborhood of $L$ in $S$. Moreover, this data is unique up to a suitable notion of isomorphism.
 \end{p*}

An important consequence of this result is that, given an arbitrary Lagrangian intersection $L\cap M$ in a symplectic manifold $S$, there is a \emph{canonical} Legendrian intersection $\Lambda_L \cap \Lambda_M$ in a contactification $Y$ of $S$ (indeed, consider the unique contactification of $(S,L\cup M)$). In particular, $(L\cap M) \times \C^\times$ is homeomorphic to the intersection $\widetilde{\Lambda}_L\cap \widetilde{\Lambda}_M$ of $\C^\times$-invariant (exact) Lagrangians in the symplectization $\widetilde{Y}$. As explained in \cref{sect:categories}, the various perverse sheaves associated to the Lagrangian intersection $L\cap M$ in $S$ can be realized naturally as perverse sheaves on $\widetilde{\Lambda}_L \cap \widetilde{\Lambda}_M \cong (L \cap M) \times \C^\times$ which are lisse along the $\C^\times$-direction. 

We are now in a position to recall the construction, due to D'Agnolo--Kashiwara \cite{DAgnoloKashiwara}, of the category $\sfRH_S$ on an arbitrary symplectic manifold $S$. 
Suppose $L \subseteq S$ is a Lagrangian subvariety and let $(Y,\Lambda)$ be a contactification in a neighborhood of $L$ (which exists by the unique contactification lemma \cref{prop:uniquecontactification}). Kashiwara \cite{KashiwaraContact} has defined a certain filtered $\C$-algebroid $\sfE_Y$ on a general contact manifold $Y$, locally modelled on the algebra of microdifferential operators $\cE_X$ on the projectivized cotangent bundle $\P^\ast X$ of a complex manifold $X$. We denote by $\sfRH_{S,L}$ the category of regular holonomic $\sfE_Y$-modules along $\Lambda$ (understood as a sheaf of categories on $L$ by pushing forward along the homeomorphism $\rho|_{\Lambda}$). We then define $\sfRH_S$ by taking the filtered colimit of $\sfRH_{S,L}$ along the diagram of all Lagrangian subvarieties $L$ in $S$. The uniqueness part of the unique contactification lemma tells us that this construction is well-defined. 

To summarize, associated to an arbitrary holomorphic symplectic manifold $S$ and a holomorphic Lagrangian subvariety $L$, we have the following diagram:
\begin{equation}
\xymatrixcolsep{5pc}
\xymatrix{
Y\times \C^\times \cong \widetilde{Y} \ar[r]^-\gamma & Y \ar[r]^-\rho & S \\
\Lambda \times \C^\times \cong \widetilde{\Lambda} \ar@{^{(}->}[u] \ar[r] & \Lambda \ar[r]^-\sim_{\mathrm{homeo}} \ar@{^{(}->}[u] & L \ar@{^{(}->}[u] 
}
\label{eq:symplecization and contactification}
\end{equation}
Here, $Y$ is a contact manifold with Legendrian subvariety $\Lambda$, $\widetilde{Y}$ is exact symplectic manifold with exact Lagrangian subvariety $\widetilde{\Lambda}$, $\rho$ is a principal $\C$-bundle over an open neighborhood of $L$ in $S$, and $\gamma$ is a trivial $\C^\times$-bundle. 

\subsection*{Microlocal sheaves}

For an exact holomorphic symplectic manifold $W$ the authors of \cite{NadlerShende,CKNS1} have introduced a sheaf of stable $\infty$-categories $\mush_W$ on $W$ of \emph{microsheaves} by globalizing the microlocal theory of sheaves of Kashiwara and Schapira \cite{KashiwaraSchapiraSheaves}. Moreover, for a holomorphic contact manifold $Y$ the authors of \cite{CKNS1} have defined a sheaf of abelian categories $\muPerv_Y\subset \gamma_*\mush_{\tilde{Y}}$ of \emph{perverse microsheaves} following \cite{Andronikof2,Waschkies1}. Perverse microsheaves relate to deformation quantization modules via a microlocal Riemann--Hilbert correspondence \cite{Andronikof3,Waschkies2,CKNS2}: namely, there is an equivalence of sheaves of abelian categories
\[
\muRH\colon \muPerv_Y \xrightarrow{\sim} \Mod_{\rh}(\sfE_Y).
\]

Given a holomorphic Legendrian submanifold $\Lambda\subset Y$ equipped with a square root line bundle $K_\Lambda^{1/2}$ there is a perverse microsheaf $\cM_\Lambda\in\muPerv_Y$ supported on $\Lambda$ which provides a \emph{sheaf quantization} of $\Lambda$ in the sense that its microlocalization to $\Lambda$ is a trivial (twisted) local system of rank $1$. Equivalently, $\muRH(\cM_\Lambda)=\sfM_\Lambda$ is a simple $\sfE_Y$-module along $\Lambda$. \cref{mainthm:category} allows us to compute Hom complexes between two such sheaf quantizations in terms of the DT-sheaf associated to the Lagrangian intersection (to our knowledge, this is the first such comparison). 

\begin{maintheorem}[\cref{thm:microsheafHom}]\label{mainthm:microsheaves}
Let $Y$ be a holomorphic contact manifold and $\Lambda_1, \Lambda_2\subset Y$ two holomorphic Legendrian submanifolds equipped with square root line bundles $K_{\Lambda_1}^{1/2}, K_{\Lambda_2}^{1/2}$. Let $\tilde{\Lambda}_1, \tilde{\Lambda}_2$ be the corresponding homogeneous Lagrangian submanifolds of the symplectization $\tilde{Y}\rightarrow Y$. There is an isomorphism of complexes of sheaves of $\C$-vector spaces on $\tilde{Y}$:
\[\cHom_{\mush_{\tilde{Y}}}(\cM_{\Lambda_1}, \cM_{\Lambda_2})[\dim\tilde{Y}/2]\cong \phi_{\tilde{\Lambda}_1\cap \tilde{\Lambda}_2}.\]
\end{maintheorem}

We note that the statement of \cref{mainthm:microsheaves} does not involve deformation quantization modules, and we expect that the same arguments will yield a $\Z$-linear version of the comparison between microsheaf Homs and the DT-sheaf. 

Let us remark also that, while the microsheaf categories considered in \cite{NadlerShende} and \cite{CKNS1} naturally form sheaves over either a contact manifold $Y$ or an exact symplectic manifold (for example, its symplectization $\widetilde{Y}$), the Unique Contactification Lemma (\cref{prop:uniquecontactification}) can also be applied to construct a certain category of microsheaves over an arbitrary holomorphic symplectic manifold $S$. As in the construction of the category $\sfDRH_S$ from \cref{mainthm:category}, near a Lagrangian subvariety $L\subseteq S$ (e.g. a finite union of Lagrangian submanifolds), objects of this category look like microsheaves on the symplectization of the unique contactification. In particular, just as with $\sfDRH_S$, this category will be canonically enriched over the category of local systems on $\C^\times$. This construction makes sense over an arbitrary coefficient ring (for example, $\Z$) and we expect there to be a direct comparison with the hom complexes in this category and the DT-sheaf on the corresponding Lagrangian intersections in $S$. We hope to return to this in future work.



\subsection*{Fukaya categories}

Consider, as before, a holomorphic contact manifold $Y$ with a symplectization $\tilde{Y}\rightarrow Y$. In particular, $\tilde{Y}$ is a real exact symplectic manifold. If we assume that it is Weinstein, then Ganatra, Pardon and Shende \cite{GPS} identify $\mush_{\tilde{Y}}$ with the ind-completion of the wrapped Fukaya category of $\tilde{Y}$. In particular, in this case the left-hand side of \cref{mainthm:microsheaves} may be rewritten in terms of Floer homology. Via the microlocal Riemann--Hilbert correspondence, this relates the sheaf of $\infty$-categories $\sfDRH_S$ on a (non-exact) holomorphic symplectic manifold $S$ constructed in \cref{mainthm:category} to Floer theory in the symplectization $\tilde{Y}$ of a local contactification $Y\rightarrow S$.

The relationship between the sheaf of $\infty$-categories $\sfDRH_S$ and Lagrangian Floer theory in $S$ (as opposed to in a contactification $Y$ or its exact symplectization $\widetilde{Y}$, say) is more subtle, see the discussion in \cite[Section B.2]{DoanRezchikov} and \cite[Section 6.3]{KontsevichSoibelman}. In this direction, let us mention the locality results for holomorphic Floer theory established by Solomon and Verbitsky in \cite{SolomonVerbitsky} as well as a relationship between Lagrangian Floer theory in $S$ and that on the symplectization of a real contactification established by Kuwagaki, Petr and Shende in \cite{KuwagakiPetrShende}. 

\subsection*{Conventions}

\begin{itemize}
    \item For a complex manifold $X$ and a ring $R$ we denote by $\bD(X;R)$ the derived $\infty$-category of sheaves of complexes of $R$-modules.
    \item $\bD_c(X; R)\subset \bD(X; R)$ denotes the subcategory of constructible complexes of $R$-modules on $X$.
    \item $\Perv(X; R)\subset \bD_c(X; R)$ denotes the subcategory of perverse sheaves.
    \item Given a closed immersion $i\colon Y\hookrightarrow X$, perverse sheaves on $Y$ are identified with perverse sheaves on $X$ supported on $Y$ via the functor $i_*$. This identification will often be implicit.
    \item We consider the subset $\EigenRange = \{ \lambda \in \C \mid -1<\Re \lambda \leq 0\}\subset \C$ which defines a splitting of the projection $\C\rightarrow \C/\Z$.
    \item A stack of categories is the same as a sheaf of categories. When there is no confusion, we will denote its global sections by the same letter.
\end{itemize}

\subsection*{Acknowledgments}

We would like to thank Claude Sabbah, Pierre Schapira and Nick Sheridan for useful conversations. S.G. was partially supported by NSF grant DMS-2202363.

\section{Background on symplectic geometry}

In this section we recall a relationship between holomorphic symplectic and holomorphic contact manifolds, their local models as well as Maslov indices of Lagrangians.

\subsection{Symplectic and contact manifolds}

In this paper we will deal with \emph{holomorphic symplectic manifolds} $(S, \omega)$, i.e. complex manifolds $S$ equipped with a closed holomorphic two-form $\omega$ such that
\[\vol_S = \frac{\omega^{\dim S/2}}{(\dim S/2)!}\]
is a volume form. As an example, if $X$ is a complex manifold, its cotangent bundle $\T^* X$ is an exact symplectic manifold with $\omega=d\lambda_X$ for the Liouville one-form $\lambda_X$. We will also be interested in the open subset
\[\oT^* X = \T^* X\setminus \T^*_X X,\]
the cotangent bundle minus the zero section. Given a symplectic manifold $(S, \omega)$, a \defterm{Lagrangian subvariety} is a locally closed subvariety $L\subset S$ of dimension $\dim L = \dim S / 2$ such that $\omega|_{L_{\reg}} = 0$, where $L_{\reg}$ is the smooth locus of $L$. We say $L\subset S$ is a \defterm{Lagrangian submanifold} if $L$ is a smooth Lagrangian subvariety.

\begin{example}
If $Z\subset X$ is a submanifold, its conormal bundle $\T^*_Z X\subset \T^* X$ defines a (homogeneous) Lagrangian submanifold.  If $\alpha$ is a one-form, its graph $\Gamma_\alpha\subset \T^* X$ defines a Lagrangian submanifold if, and only if, $\alpha$ is closed.
\end{example}

Given a symplectomorphism $\phi\colon S_1\rightarrow S_2$, its graph defines a Lagrangian submanifold $\Gamma_\phi\subset S_1\times \overline{S}_2$, where $\overline{S}_2$ refers to $S_2$ with the opposite symplectic structure. For instance, the graph of the identity defines the diagonal Lagrangian $\Delta_S\subset S\times \overline{S}$.

We will use natural symplectomorphisms
\[\T^*X\times \T^*Y\cong \T^*(X\times Y),\qquad \T^*X\times \overline{\T^* Y}\cong \T^*(X\times Y)\]
given by $(x, p, x', p')\mapsto (x, x', p, p')$ and $(x, p, x', p')\mapsto (x, x', p, -p')$.

Besides holomorphic symplectic manifolds we will also encounter holomorphic contact manifolds. Recall that a \defterm{holomorphic contact manifold} $Y$ is a complex manifold equipped with a line bundle $\cO_Y(1)$ and a twisted one-form $\alpha\in\Gamma(Y, \Omega^1_Y\otimes \cO_Y(1))$ satisfying a nondegeneracy condition. Alternatively, a contact structure can be specified by a hyperplane distribution in $\T_Y$ given by the kernel of $\alpha$.

Let $\gamma\colon \tilde{Y}\rightarrow Y$ be the $\C^\times$-torsor associated to the inverse of $\cO_Y(1)$. Then the twisted 1-form $\alpha$ on $Y$ defines a 1-form $\alpha_{\tilde{Y}}$ on $\tilde{Y}$ linear along the fibers. The nondegeneracy condition on $\alpha$ is equivalent to the condition that $(\tilde{Y}, \omega_{\tilde{Y}}=-d\alpha_{\tilde{Y}})$ is an exact symplectic manifold. We call $\tilde{Y}$ the \defterm{symplectization} of $Y$. It carries a natural Euler vector field given by the derivative of the natural $\C^\times$-action with respect to which $\omega_{\tilde{Y}}$ has weight $1$.

A contactomorphism $\psi\colon Y_1\rightarrow Y_2$ is the same as a homogeneous symplectomorphism $\tilde{\psi}\colon \tilde{Y}_1\rightarrow \tilde{Y}_2$. A \defterm{Legendrian subvariety} of a contact manifold $Y$ is a locally closed subvariety $\Lambda\subset Y$ of dimension $\dim\Lambda = (\dim(Y)-1)/2$ such that $\alpha|_{\Lambda_{\reg}} = 0$. A Legendrian subvariety lifts to a $\C^\times$-invariant (alias, \emph{conic}) Lagrangian subvariety $\tilde{\Lambda}\subset \tilde{Y}$ in the symplectization.

\begin{example}
Consider the projectivized cotangent bundle of a complex manifold $X$ which is the quotient
\[\P^* X = \oT^* X / \C^\times.\]
There is a natural contact structure on $\P^* X$ with respect to which $\T^* X\setminus X\rightarrow \P^* X$ is the symplectization. The corresponding one-form $\alpha_{\tilde{Y}} = -\lambda_X$ is minus the Liouville one-form. A submanifold $Z\subset X$ of positive codimension defines a homogeneous Lagrangian submanifold $\oT^*_Z X\subset \oT^* X$ and a Legendrian submanifold $\P^*_Z X\subset \P^* X$.
\end{example}

\begin{example}
Consider the first jet bundle $\rJ^1 X=\T^* X\times\C$ of a complex manifold $X$. There is a natural contact structure on $\rJ^1 X$ given by $\alpha = dt - \lambda_X$, where $t$ is the coordinate along $\C$ and $\lambda_X$ the Liouville one-form. Given a holomorphic function $f\colon X\rightarrow \C$ its one-jet $j^1 f\colon X\rightarrow \rJ^1 X$ defines a Legendrian submanifold. There is a homogeneous symplectic embedding $\tilde{\rJ^1 X}\hookrightarrow \oT^*(X\times \C)$ given by $(x, p, t, \tau)\mapsto (x, t, p\tau, -\tau)$, where $(x, p)$ are Darboux coordinates on $\T^* X$. Under this embedding the conic Lagrangian submanifold $\tilde{j^1 f}\subset \tilde{\rJ^1 X}$ is sent to $\oT^*_{\Gamma_f}(X\times\C)\subset \oT^*(X\times \C)$, where $\Gamma_f\subset X\times \C$ is the submanifold defined by $t = f(x)$.
\end{example}

\subsection{Contactification}

Let us now study the process converse to symplectization.

\begin{defn}
Let $(S, \omega)$ be a symplectic manifold. A \defterm{contactification} of $S$ is a principal $\C$-bundle $\rho\colon Y\rightarrow S$ together with a connection whose curvature is equal to $-\omega$. Let $L\subset S$ be a Lagrangian subvariety. A \defterm{contactification} of $L\subset S$ is a choice of a contactification $\rho\colon Y\rightarrow S$ of $S$ together with a subvariety $\Lambda\subset Y$ (the \defterm{Legendrian lift}), such that $\rho\colon \Lambda\rightarrow L$ is a homeomorphism and an analytic isomorphism over the smooth locus $\Lambda_{\reg}$, on which the connection one-form restricts to zero.
\end{defn}

\begin{remark}
It is not in general possible to find a contactification with $\rho\colon \Lambda\rightarrow L$ an analytic isomorphism everywhere, see \cite[Example 2.5.4]{DAgnoloKashiwara}.
\end{remark}

Given a symplectomorphism $\phi\colon S_1\rightarrow S_2$ and contactifications $Y_1\rightarrow S_1$ and $Y_2\rightarrow S_2$, an isomorphism of contactifications is a diffeomorphism $\psi\colon Y_1\rightarrow Y_2$ which is $\C$-equivariant and which intertwines the connection one-forms $\alpha_i$.

A contactification $\rho\colon Y\rightarrow S$ has a natural contact structure with $\alpha$ the connection one-form on $Y$. A contactification of a Lagrangian subvariety of $S$ is then a Legendrian subvariety of $Y$. It is easy to see that contactifications of a given symplectic manifold $(S, \omega)$ form a $\C$-gerbe $\sfC_S$ whose characteristic class is $[\omega]\in\rH^2(S; \C)$.

Given a contactification $\rho\colon Y\rightarrow S$ of a symplectic manifold $S$, $\rho\colon \overline{Y}\rightarrow \overline{S}$ is a contactification of the opposite symplectic manifold, where $\overline{Y}$ is $Y$ equipped with the opposite contact form $\alpha$ and opposite $\C$-action.

\begin{example}
Suppose $(S, \omega=d\lambda)$ is an exact symplectic manifold. Then $Y=S\times \C$ with $\rho$ the projection to the first factor, $\cO_Y(1)$ trivial line bundle and $\alpha = dt - \lambda$, where $t$ is the coordinate along $\C$, is a contactification of $S$.
\end{example}

We will be particularly interested in the following case of the previous example.

\begin{example}
The projection $\rho\colon \rJ^1 X\rightarrow \T^* X$ is a contactification. Given a closed one-form $\alpha$ on $X$, Legendrian lifts of the Lagrangian submanifold $\Gamma_\alpha\subset \T^* X$ are $j^1 f\subset \rJ^1 X$ for some holomorphic function $f\colon X\rightarrow \C$ with $\alpha = df$.
\label{ex:graphcontactification}
\end{example}

Consider a contactification $\rho\colon Y\rightarrow S$ and its symplectization $\gamma\colon \tilde{Y} = Y\times \C^\times\rightarrow Y$. The $\C$-action on $Y$ lifts to a Hamiltonian $\C$-action on $\tilde{Y}$ with Hamiltonian given by the natural coordinate $\tau$ along $\C^\times$. In this case we may identify $S$ with the Hamiltonian reduction $\tilde{Y}\ham\C$ at $\tau=1$. An isomorphism of contactifications $\psi\colon Y_1\rightarrow Y_2$ covering a symplectomorphism $\phi\colon S_1\rightarrow S_2$ gives rise to a conic symplectomorphism $\tilde{\psi}\colon \tilde{Y}_1\rightarrow \tilde{Y}_2$ compatible the Hamiltonian $\C$-actions.

Given contactifications $Y_1\rightarrow S_1$ and $Y_2\rightarrow S_2$ we denote
\[Y_1\times^\ast Y_2 = (Y_1\times Y_2)/\C\rightarrow S_1\times S_2,\]
where the quotient is taken with respect to the antidiagonal action of $\C$ on $Y_1\times Y_2$. Given Legendrians $\Lambda_1\subset Y_1$ and $\Lambda_2\subset Y_2$ we denote by
\[\Lambda_1\times^\ast \Lambda_2\subset Y_1\times^\ast Y_2\]
the image of $\Lambda_1\times \Lambda_2$ under the projection $Y_1\times Y_2\rightarrow Y_1\times^\ast Y_2$.

\begin{prop}\label{prop: transforms of Legendrians}
Let $S_1, S_2$ be complex symplectic manifolds and $\rho\colon Y_i\rightarrow S_i$ their contactifications.
\begin{itemize}
    \item $Y_1\times^\ast Y_2\rightarrow S_1\times S_2$ is a contactification.
    \item Given a pair of Lagrangians $L_1\subset S_1, L_2\subset S_2$ together with their Legendrian lifts $\Lambda_1\subset Y_1,\Lambda_2\subset Y_2$ the submanifold $\Lambda_1\times^\ast \Lambda_2\subset Y_1\times^\ast Y_2$ provides a contactification of the Lagrangian $L_1\times L_2\subset S_1\times S_2$.
    \item Given a commutative diagram
    \[
    \xymatrix{
    Y_1 \ar^{\psi}[r] \ar[d] & Y_2 \ar[d] \\
    S_1 \ar^{\phi}[r] & S_2
    }
    \]
    where $\phi\colon S_1\rightarrow S_2$ is a symplectomorphism and $\psi\colon Y_1\rightarrow Y_2$ is an isomorphism of contactifications, the image $\Gamma_\psi\subset Y_1\times^\ast \overline{Y}_2$ of the graph of the contactomorphism $\psi$ under $Y_1\times \overline{Y}_2\rightarrow Y_1\times^\ast \overline{Y}_2$ defines a Legendrian lift of the graph $\Gamma_\phi\subset S_1\times \overline{S}_2$ of the symplectomorphism $\phi$.
\end{itemize}
\label{prop:productcontactification}
\end{prop}
\begin{proof}
The one-forms $\alpha_i$ on $Y_i$ satisfy $g^* \alpha_i = \alpha_i$ for any $g\in\C$ as well as $\iota_R \alpha_i = 1$, where $R$ is the vector field generating the $\C$-action. Thus, $\alpha_1 + \alpha_2$ is basic, so it descends along the projection $Y_1\times Y_2\rightarrow (Y_1\times Y_2)/\C$. The diagonal action of $\C$ on $Y_1\times Y_2$ descends to an action on $(Y_1\times Y_2)/\C$ which defines the structure of a principal $\C$-bundle on $(Y_1\times Y_2)/\C\rightarrow S_1\times S_2$. The one-form $\alpha_1+\alpha_2$ on $(Y_1\times Y_2)/\C$ defines a connection one-form whose curvature is $\omega_1 + \omega_2$, so $(Y_1\times Y_2)/\C\rightarrow S_1\times S_2$ is a contactification.

The one-form $\alpha_1+\alpha_2$ vanishes on $\Lambda_1\times \Lambda_2$ and so, by counting dimensions, $\Lambda_1\times \Lambda_2\subset (Y_1\times Y_2)/\C$ is Legendrian.

It is easy to see that the composite $\Gamma_\psi\rightarrow (Y_1\times Y_2)/\C\rightarrow Y_1/\C=S_1$ is a diffeomorphism. The fact that $\psi^* \alpha_2 = \alpha_1$ implies that $\alpha_1 - \alpha_2$ vanishes on $\Gamma_\psi\subset (Y_1\times Y_2)/\C$.
\end{proof}

\begin{remark}
We may identify the symplectization of $Y_1\times^\ast Y_2$ with the Hamiltonian reduction $(\tilde{Y}_1\times \tilde{Y}_2)\ham\C$ at $\tau_1=\tau_2$. The graph $\Gamma_{\tilde{\psi}}\subset \tilde{Y}_1\times \overline{\tilde{Y}_2}$ of the conic symplectomorphism $\tilde{\psi}\colon \tilde{Y}_1\rightarrow \tilde{Y}_2$ is $\C$-equivariant and we may identify $\tilde{\Gamma_\psi}\subset (\tilde{Y}_1\times \tilde{Y}_2)\ham \C$ with $\Gamma_{\tilde{\psi}}/\C$.
\end{remark}

\begin{example}
Given a contactification $\rho\colon Y\rightarrow S$, the contactification of the diagonal Lagrangian $\Delta_S\subset S\times \overline{S}$ is given by the Legendrian $\Delta_Y\subset Y\times^\ast \overline{Y}$ given by the image of the diagonal of $Y$.
\end{example}

\begin{example}
Let $X_1, X_2$ be complex manifolds. Then
\[\rJ^1 X_1\times^\ast \rJ^1 X_2\longrightarrow \rJ^1(X_1\times X_2)\]
given by $(x, p, t), (x', p', t')\mapsto (x, x', p, p', t + t')$ is a contactomorphism. Given holomorphic functions $f_1\colon X_1\rightarrow \C$ and $f_2\colon X_2\rightarrow \C$, under this isomorphism the Legendrian
\[j^1f_1\times^\ast j^1 f_2\subset \rJ^1 X_1\times^\ast \rJ^1 X_2\]
is identified with the Legendrian $j^1(f_1\boxplus f_2)\subset \rJ^1(X_1\times X_2)$, where $f_1\boxplus f_2\colon X_1\times X_2\rightarrow \C$ is the function $(x, x')\mapsto f_1(x)+f_2(x')$. Similarly,
\[\rJ^1 X_1\times^\ast \overline{\rJ^1 X_2}\longrightarrow \rJ^1(X_1\times X_2)\]
given by $(x, p, t), (x', p', t')\mapsto (x, x', p, -p', t-t')$ is a contactomorphism.
\label{ex:productjets}
\end{example}

We will now show that near a Lagrangian subvariety $L\subset S$ there is a canonical contactification.

\begin{prop}
Let $L\subset S$ be a Lagrangian subvariety. Then there is a canonical section of $\sfC_\omega|_L$. Explicitly, there is a unique contactification $\rho\colon Y\rightarrow S$ in a neighborhood of $L$ together with a lift of $L$ to a Legendrian $\Lambda\subset Y$ such that $\rho\colon \Lambda\rightarrow L$ is a homeomorphism and, moreover, is an analytic isomorphism over the smooth locus of $L$.
\label{prop:uniquecontactification}
\end{prop}
\begin{proof}
The existence of such a contactification is shown in \cite[Proposition 2.5.3]{DAgnoloKashiwara}. Let us now show that any two such contactifications are uniquely isomorphic. By \cite{Giesecke} we may choose a triangulation of $S$ with $L\subset S$ as a simplicial subcomplex. Therefore, there is a neighborhood $U$ of $L$ which deformation retracts to $L$.

Choose a cover $\{U_i\}$ of $U$. Suppose $\{\theta^1_i, f^1_i, \varphi^1_{ij}\}$ and $\{\theta^2_i, f^2_i, \varphi^2_{ij}\}$ are two descent data on $\{U_i\}$ defining two contactifications $Y^1, Y^2$ of $S$. Here $\theta^n_i\in\Omega^1(U_i)$, $f^n_i$ is a continuous function on $L\cap U_i$ holomorphic on the smooth locus and $\varphi^n_{ij}$ is a holomorphic function on $U_{ij}$ which together satisfy
\begin{align*}
d\theta^n_i &= \omega,\qquad df^n_i|_{L_{\reg}} = \theta^n_i|_{L_{\reg}},\qquad d\varphi^n_{ij} = \theta^n_i - \theta^n_j,\\
\varphi^n_{ij}|_L &= f^n_i - f^n_j,\qquad \varphi^n_{ij}+\varphi^n_{jk}+\varphi^n_{ki} = 0.
\end{align*}

Consider their difference:
\[\theta_i = \theta^1_i - \theta^2_i,\qquad f_i = f^1_i - f^2_i,\qquad \varphi_{ij} = \varphi^1_{ij} - \varphi^2_{ij}.\]
Then
\begin{align*}
d \theta_i &= 0,\qquad df_i|_{L_{\reg}} = \theta_i|_{L_{\reg}},\qquad d\varphi_{ij} = \theta_i - \theta_j, \\ 
\varphi_{ij}|_L &= f_i - f_j,\qquad \varphi_{ij}+\varphi_{jk}+\varphi_{ki} = 0.
\end{align*}

Possibly refining the cover, we choose holomorphic functions $g_i$ on $U_i$ such that $dg_i = \theta_i$. Then $g_i - g_j - \varphi_{ij}$ define locally constant functions on $U_{ij}$ which determine a cohomology class in $\rH^1(U; \C)$. The restriction of these functions to $L$ is $(g_i - f_i) - (g_j - f_j)$. Note that $d(g_i - f_i)|_{L_{\reg}} = 0$, i.e. $g_i - f_i$ is locally constant on $L_{\reg}\cap U_i$ and hence, by density, on $L\cap U_i$ (see \cite[Sublemma 2.5.2]{DAgnoloKashiwara}). Thus, under the restriction map $\rH^1(U; \C)\rightarrow \rH^1(L; \C)$ the class $[g_i-g_j-\varphi_{ij}]$ restricts to zero. Since $U$ deformation retracts to $L$, this implies that there are locally constant functions $h_i$ on $U_i$ such that
\[\varphi_{ij} = (g_i - h_i) - (g_j - h_j),\qquad (g_i - h_i)|_L = f_i.\]
Thus, the functions $g_i-h_i$ define an isomorphism of the contactifications 
of $S$ which preserves the Legendrian lift of $L$.
\end{proof}

\subsection{Orientation torsors}

Given a triple of Lagrangian subspaces of a symplectic vector space, Wall and Kashiwara have associated a certain quadratic form whose index (in the real case) is the \emph{Maslov index} of the triple of Lagrangians. In this section we will present a variant of this construction.

Throughout this section we fix a field $k$. Let $W$ be a $k$-vector space equipped with a quadratic form $q$. There is an induced quadratic form on the determinant line $\det(W)$. We denote by $P_q$ the $\Z/2$-torsor of elements $\omega\in\det(W)$ of norm $1$. In other words, $P_q$ is the torsor of orientations of the quadratic vector space $(W, q)$. If $W$ is a quadratic vector bundle over a manifold, the definition of $P_q$ extends to define a $\Z/2$-torsor over the base manifold.

Given a pair $(W_1, q_1)$, $(W_2, q_2)$ of quadratic vector spaces, we have a natural isomorphism
\begin{equation}
P_{q_1}\otimes P_{q_2}\cong P_{q_1\oplus q_2}
\label{eq:Pqaddition}
\end{equation}
given by $\omega_1\otimes \omega_2\mapsto \omega_1\wedge \omega_2$, where we denote by $q_1\oplus q_2\colon W_1\oplus W_2\rightarrow k$ is the orthogonal direct sum of quadratic forms.

Now let $V$ be a symplectic $k$-vector space and $L_1,L_2\subset V$ are two Lagrangian subspaces which are transverse. The transversality condition may be equivalently stated as the fact that the map
\[f_{L_1L_2}\colon L_1\longrightarrow L_2^*\]
given by $f_{L_1L_2}(v) = (w\mapsto \omega(v, w))$ is an isomorphism.

\begin{defn}
Let $L_1, L_2, \dots, L_n\subset V$ be a sequence of Lagrangian subspaces, where any two subsequent subspaces are transverse.
\begin{itemize}
    \item If $n$ is even, we define the isomorphism $C(L_1, \dots, L_n)\colon L_1\rightarrow L_n^*$ to be given by the composite
    \[
    C(L_1, \dots, L_n)\colon L_1\xrightarrow{f_{L_1L_2}} L_2^*\xrightarrow{f_{L_3L_2}^{-1}} L_3\rightarrow\dots \xrightarrow{f_{L_{n-1}L_n}} L_n^*
    \]
    \item If $n$ is odd, we define the isomorphism $C(L_1, \dots, L_n)\colon L_1\rightarrow L_n$ to be given by the composite
    \[
    C(L_1, \dots, L_n)\colon L_1\xrightarrow{f_{L_1L_2}} L_2^*\xrightarrow{f_{L_3L_2}^{-1}} L_3\rightarrow\dots \xrightarrow{f_{L_{n-1}L_n}^{-1}} L_n
    \]
\end{itemize}
\label{def:sequenceLagrangianisomorphism}
\end{defn}

\begin{example}
Consider pairwise transverse Lagrangian subspaces $L_1, L_2, L_3\subset V$. The isomorphism
\[C(L_1, L_2, L_3, L_1)\colon L_1\longrightarrow L_1^*\]
is self-dual and comes from a quadratic form
\[q(L_1, L_2, L_3)\colon L_1\rightarrow k\]
which we call the \defterm{Maslov quadratic form}. For $k=\R$ its signature is the Maslov index $\tau(L_1, L_2, L_3)$ of the triple of Lagrangians as defined in \cite[Chapter A.3]{KashiwaraSchapiraSheaves}.
\label{ex:Maslovquadraticform}
\end{example}

\begin{example}
Consider a finite-dimensional vector space $L$ and the corresponding symplectic vector space $V = L\oplus L^*$ with the symplectic form
\[\omega(x+\eta, y+\xi) = \eta(y) - \xi(x).\]
It carries two transverse Lagrangian subspaces $L, L^*\subset V$. Consider a linear map $A\colon L\rightarrow L^*$ and the corresponding half-dimensional subspace $M\subset V$ given by
\[M = \{(l, A(l))\mid l\in L\}.\]
The subspace $M\subset V$ is Lagrangian precisely if $A$ is self-dual, i.e. $A$ comes from a quadratic form $q\colon L\rightarrow k$.

The subspace $M\subset V$ is transverse to $L\subset V$ (equivalently, to $L^*\subset V$) precisely if $A$ is an isomorphism, i.e. $q$ is nondegenerate.

In this case $C(L, M, L^*, L)$ is given by the composite
\[L\xrightarrow{f_{LM}} M^* \xleftarrow{f_{L^*M}} L^*\xrightarrow{f_{L^*L}} L^*.\]
It is easy to work out it to be $-A$, so that the Maslov quadratic form $q(L, M, L^*)$ is equal to $-q$.
\label{ex:Maslovexample}
\end{example}

We have the following obvious properties of the isomorphisms $C(L_1, \dots, L_n)$.

\begin{prop}$ $
\begin{itemize}
    \item Suppose $n$ is odd. Then
    \[C(L_n, L_{n+1}, \dots, L_{n+m})\circ C(L_1, \dots, L_n) = C(L_1, \dots, L_n, \dots, L_{n+m}).\]
    \item Consider a sequence $L_1, \dots, L_{n-1}, L_n, L_{n-1}, \dots, L_m$ of Lagrangian subspaces, where the subsequent ones are transverse to each other. Then
    \[C(L_1, \dots, L_{n-1}, L_n, L_{n-1}, \dots, L_m) = C(L_1, \dots, L_{n-1}, \widehat{L_n}, \widehat{L_{n-1}}, \dots, L_m).\]
\end{itemize}
\label{prop:Maslovcomposition}
\end{prop}

\begin{defn}
Let $V$ be a symplectic space and $L\subset V$ a Lagrangian subspace. An \defterm{orientation data} on $L$ is the choice of the square root line $\det(L)^{1/2}$. Let $L\subset S$ be a Lagrangian submanifold of a symplectic manifold. An \defterm{orientation data} on $L$ is the choice of the square root line bundle $K_L^{1/2}$ of the canonical bundle of $L$.
\end{defn}

\begin{defn}
Let $L_1, \dots, L_n\subset V$ be a sequence of Lagrangian subspaces, where any two subsequent subspaces are transverse. Moreover, suppose $L_1$ and $L_n$ are equipped with an orientation data. An \defterm{orientation} of the sequence $L_1, \dots, L_n$, where $n$ is odd, is the choice of an isomorphism
\[\det(L_1)^{1/2}\longrightarrow \det(L_n)^{1/2}\]
whose square is
\[\det(C(L_1, \dots, L_n))\colon \det(L_1)\longrightarrow \det(L_n).\]
In the case of $n$ even the definition is analogous. The choices of an orientation of the sequence define a $\Z/2$-torsor $P_{L_1, \dots, L_n}$.
\label{def:orientationsequenceLagrangians}
\end{defn}

\begin{remark}
The notion of an orientation of the sequence $L_1, L_2, L_3, L_1$ is independent of the choice of the orientation data on $L_1$. In this case we have
\[P_{L_1, L_2, L_3, L_1}\cong P_{q(L_1, L_2, L_3)}\]
for the Maslov quadratic form $q(L_1, L_2, L_3)\colon L_1\rightarrow k$.
\end{remark}

These definitions extend in an obvious way to symplectic bundles and sequences of Lagrangian subbundles.

\subsection{Charts on symplectic manifolds}

In this section we describe convenient local models of symplectic manifolds. Let us first recall the notion of a polarization.

\begin{defn}
A \defterm{polarization} of $S$ is a holomorphic Lagrangian fibration $\pi\colon S\rightarrow E$.
\end{defn}

\begin{remark}
The Darboux theorem ensures that polarizations of symplectic manifolds exist locally.
\end{remark}

\begin{defn}
A polarization $\pi\colon S\rightarrow E$ is \defterm{transverse} to a Lagrangian $L\subset S$ if the composite $L\subset S\rightarrow E$ is a local diffeomorphism.
\end{defn}

When a polarization $\pi$ is transverse to a Lagrangian $L$ we obtain a projection map $\pi^{S\rightarrow L}\colon S\rightarrow L$ locally on $L$.

\begin{example}
Consider a symplectic manifold $S$ equipped with a polarization $\pi$ and two Lagrangian submanifolds $L, M\subset S$ transverse to $\pi$. At any point $x\in L\cap M$ we have a sequence $T_{L, x}, T_{\pi, x}, T_{M, x}$ of Lagrangian subspaces of $T_{S, x}$, where the subsequent subspaces are transverse. Thus, from \cref{def:sequenceLagrangianisomorphism} we obtain an isomorphism
\[C(\T_{L, x}, \T_{\pi, x}, \T_{M, x})\colon \T_{L, x}\longrightarrow \T_{M, x}.\]
Unpacking this isomorphism, it is given by the composite of the isomorphisms
\[\T_{L, x}\xrightarrow{\pi} \T_{E, \pi(x)}\xleftarrow{\pi} \T_{M, x}.\]
\end{example}

\begin{example}
If $S = \T^* E$ is a cotangent bundle, there is a canonical \emph{vertical polarization} $\T^* E\rightarrow E$ given by the projection. It is transverse to the zero section $E\subset \T^* E$.
\end{example}

The following three statements provide a local description of symplectic manifolds equipped with polarizations.

\begin{prop}
Let $S$ be a symplectic manifold, $\rho\colon Y\rightarrow S$ a contactification $\pi\colon S\rightarrow E$ a polarization, $L\subset S$ a Lagrangian submanifold transverse to $\pi$ and $\Lambda_L\subset Y$ is its Legendrian lift. For every point $y\in \Lambda_L$ there is an open neighborhood $U\subset S$ of $\rho(y)$, a commutative diagram
\[
\xymatrix{
\rho^{-1}(U) \ar^{i_{\pi}}[r] \ar^{\rho}[d] & \rJ^1 L \ar[d] \\
U \ar^{\overline{i}_{\pi}}[r] & \T^* L
}
\]
with $i_{\pi}$ a contact embedding and $\overline{i}_{\pi}$ a symplectic embedding which identifies $L\subset S$ with the zero section of $\T^* L$ and $\Lambda_L\subset Y$ with $j^1 0\subset \rJ^1 L$. Moreover, the maps $i_\pi$ and $\overline{i}_{\pi}$ are uniquely determined by these properties.
\label{prop:polarizationLagrangian}
\end{prop}
\begin{proof}
For the statement on the level of symplectic manifolds, see \cite[Theorem 7.1]{WeinsteinLagrangian}. By \cref{ex:graphcontactification} $j^10\subset \rJ^1 L$ is a contactification of the zero section $L\subset \T^* L$ and by \cref{prop:uniquecontactification} contactifications of Lagrangians are unique. This implies the statement about the local form of the contactification.
\end{proof}

\begin{remark}
Given a symplectic manifold $S$ with a polarization $\pi$ transverse to a Lagrangian submanifold $L\subset S$ we may apply \cref{prop:polarizationLagrangian} to both $S$ and $\overline{S}$ to get symplectic embeddings $U\hookrightarrow \T^* L$ and $\overline{U}\hookrightarrow \T^* L$. The two differ by negation along the fibers of the cotangent bundle.
\end{remark}

In the case of two Lagrangians we have the following statement.

\begin{prop}
Let $S$ be a symplectic manifold, $\rho\colon Y\rightarrow S$ a contactification, $\pi\colon S\rightarrow E$ a polarization, $L,M \subset S$ Lagrangian submanifolds transverse to $\pi$ and $\Lambda_L,\Lambda_M\subset Y$ their Legendrian lifts. For every point $y\in \Lambda_L\cap \Lambda_M$ there is an open neighborhood $U\subset S$ of $\rho(y)$, a commutative diagram
\[
\xymatrix{
\rho^{-1}(U) \ar^{i_{\pi}}[r] \ar^{\rho}[d] & \rJ^1 L \ar[d] \\
U \ar^{\overline{i}_{\pi}}[r] & \T^* L
}
\]
where $i_{\pi}$ is a contact embedding and $\overline{i}_{\pi}$ is a symplectic embedding and a locally defined holomorphic function $f\colon L\rightarrow \C$ vanishing on $(L\cap M)^{\red}$, such that:
\begin{itemize}
    \item Under the embedding $\overline{i}_{\pi}\colon U\hookrightarrow \T^*L$ the polarization $\pi$ identifies with the vertical polarization $\T^* L\rightarrow L$.
    \item The Legendrian $\Lambda_L\subset Y$ covering $L\subset S$ identifies with $j^10\subset \rJ^1 L$ covering the zero section $L\subset \T^* L$.
    \item The Legendrian $\Lambda_M\subset Y$ covering $M\subset S$ identifies with $j^1 f\subset \rJ^1 L$ covering the graph $\Gamma_{df}\subset \T^* L$ of $df$.
\end{itemize}
Moreover, the maps $i_{\pi}$, $\overline{i}_{\pi}$ as well as the function $f$ are uniquely determined by these conditions.
\label{prop:polarizationtwoLagrangians}
\end{prop}
\begin{proof}
Using \cref{prop:polarizationLagrangian} we find a neighborhood $U$ of $\rho(y)$ and a contact embedding $\rho^{-1}(U)\hookrightarrow \rJ^1 L$ covering a symplectic embedding $U\subset \T^* L$ under which $\Lambda_L\subset Y$ covering $L\subset S$ identifies with $j^1 0\subset \rJ^1 L$ covering the zero section $L\subset \T^* L$.

Since $M$ is transverse to $\pi$, the composite $M\subset S\rightarrow E$ is a local diffeomorphism. Therefore, under $U\subset \T^*L$ the submanifold $M\subset S$ identifies with the graph of a locally defined one-form $\alpha$ on $L$ which is closed since $M$ is Lagrangian. Lifts of $\Gamma_\alpha\subset \T^* L$ to the contactification $\rJ^1 L\rightarrow \T^* L$ are given by choosing a function $f\colon M\rightarrow \C$ with $\alpha = df$, so that $\Lambda_M\subset Y$ is identified with $j^1 f\subset \rJ^1 L$. The function $f$ vanishes at $y\in\Lambda_L\cap \Lambda_M$ and it is locally constant on $(L\cap M)^{\red}$. Therefore, shrinking the neighborhood $U$, we may assume $f$ vanishes on all of $(L\cap M)^{\red}$.
\end{proof}

\begin{remark}
It is not in general possible to arrange $f$ to be zero on $L\cap M$ when this intersection is non-reduced, see \cite[Example 2.13]{JoycedCrit}.
\end{remark}

Finally, we will need to describe the change of local models under a change of polarizations. Recall the symplectomorphism $\T^*L\times \overline{\T^* M}\cong \T^*(L\times M)$ and the contactomorphism $\rJ^1L\times^\ast \overline{\rJ^1M}\cong \rJ^1(L\times M)$ from \cref{ex:productjets}.

\begin{prop}
Let $S$ be a symplectic manifold, $\rho\colon Y\rightarrow S$ a contactification, $\pi_i\colon S\rightarrow E_i$ for $i=1, 2$ transverse polarizations, $L,M \subset S$ Lagrangian submanifolds transverse to both $\pi_1$ and $\pi_2$ and $\Lambda_L,\Lambda_M\subset Y$ their Legendrian lifts. For every point $y\in \Lambda_L\cap \Lambda_M$ there is an open neighborhood $U\subset S$ of $\rho(y)$, commutative diagrams
\begin{align*}
&\xymatrix{
\rho^{-1}(U) \ar^{i_{\pi_1}}[r] \ar^{\rho}[d] & \rJ^1 L \ar[d] \\
U \ar^{\overline{i}_{\pi_1}}[r] & \T^* L
}\qquad
\xymatrix{
\rho^{-1}(U) \ar^{i_{\pi_2}}[r] \ar^{\rho}[d] & \rJ^1 M \ar[d] \\
U \ar^{\overline{i}_{\pi_2}}[r] & \T^* M
}\\
&\xymatrix@C=1.5cm{
\rho^{-1}(U)\times^\ast \overline{\rho^{-1}(U)} \ar^-{i_{\pi_1}\times i_{\pi_2}}[r] \ar^{\rho\times\rho}[d] & \rJ^1 (L\times M) \ar[d] \\
U\times \overline{U} \ar^{\overline{i}_{\pi_1}\times \overline{i}_{\pi_2}}[r] & \T^*(L\times M)
}
\end{align*}
where $i_{\pi_1},i_{\pi_2}$ are contact embeddings and $\overline{i}_{\pi_1},\overline{i}_{\pi_2}$ are symplectic embeddings and locally defined functions $h\colon L\times M\rightarrow \C$, $f\colon L\rightarrow \C$, $g\colon M\rightarrow \C$ vanishing on $(L\cap M)^{\red}$, such that:
\begin{itemize}
    \item Under the embedding $U\hookrightarrow \T^* L$ the polarization $\pi_1$ identifies with the vertical polarization $\T^* L\rightarrow L$.
    \item Under the contact embedding $\rho^{-1}(U)\hookrightarrow \rJ^1 L$ the Legendrian $\Lambda_L\subset Y$ covering $L\subset S$ identifies with $j^10\subset \rJ^1 L$ covering the zero section $L\subset \T^* L$. Under the same embedding the Legendrian $\Lambda_M\subset Y$ covering $M\subset S$ identifies with $j^1 f\subset \rJ^1 L$ covering $\Gamma_{df}\subset \T^* L$.
    \item Under the embedding $U\hookrightarrow \T^* M$ the polarization $\pi_2$ identifies with the vertical polarization $\T^* M\rightarrow M$.
    \item Under the contact embedding $\rho^{-1}(U)\hookrightarrow \rJ^1 M$ the Legendrian $\Lambda_M\subset Y$ covering $M\subset S$ identifies with $j^10\subset \rJ^1 M$ covering the zero section $M\subset \T^* M$. Under the same embedding the Legendrian $\Lambda_L\subset Y$ covering $L\subset S$ identifies with $j^1 g\subset \rJ^1 M$ covering $\Gamma_{dg}\subset \T^* M$.
    \item Under the embedding $U\times \overline{U}\hookrightarrow \T^* (L\times M)$ the polarization $\pi_1\times \pi_2$ identifies with the vertical polarization $\T^* (L\times M)\rightarrow L\times M$.
    \item Under the contact embedding $\rho^{-1}(U)\times^\ast \overline{\rho^{-1}(U)}\hookrightarrow \rJ^1 (L\times M)$ the Legendrian $\Lambda_L\times^\ast \Lambda_M\subset Y\times^\ast \overline{Y}$ covering $L\times M\subset S\times \overline{S}$ identifies with $j^10\subset \rJ^1 (L\times M)$ covering the zero section $L\times M\subset \T^* (L\times M)$. Under the same embedding the Legendrian $\Delta_Y\subset Y\times^\ast \overline{Y}$ covering the diagonal $\Delta_S\subset S\times \overline{S}$ identifies with $j^1 h\subset \rJ^1 (L\times M)$ covering $\Gamma_{dh}\subset \T^* (L\times M)$.
    \item Consider the local diffeomorphism $\pi^{S\rightarrow L}_1\colon M\rightarrow L$. Under the locally defined embedding
    \[\Xi= (\id\times (\pi^{S\rightarrow L}_1)^{-1})\colon L\hookrightarrow L\times M\]
    we have $\Xi^* h = f$.
    \item Consider the local diffeomorphism $\pi^{S\rightarrow M}_2\colon L\rightarrow M$. Under the locally defined embedding
    \[\Upsilon= ((\pi^{S\rightarrow M}_2)^{-1}\times \id)\colon M\hookrightarrow L\times M\]
    we have $\Upsilon^* h = -g$.
\end{itemize}
\label{prop:twopolarizationstwoLagrangians}
\end{prop}
\begin{proof}
Applying \cref{prop:polarizationtwoLagrangians} to $(Y\rightarrow S, \pi_1, L, M)$ we get a commutative diagram
\[
\xymatrix{
\rho^{-1}(U) \ar^{i_{\pi_1}}[r] \ar[d] & \rJ^1 L \ar[d] \\
U \ar^{\overline{i}_{\pi_1}}[r] & \T^* L
}
\]
together with a locally defined function $f\colon L\rightarrow \C$, such that under the embeddings the polarization $\pi_1$ goes to the vertical polarization, $L$ goes to the zero section, $\Lambda_L$ goes to $j^10$, $M$ goes to $\Gamma_{df}\subset \T^* L$ and $\Lambda_M$ goes to $j^1 f\subset \rJ^1 L$.

Applying \cref{prop:polarizationtwoLagrangians} to $(Y\rightarrow S, \pi_2, M, L)$ (and possibly shrinking $U$) we get a commutative diagram
\[
\xymatrix{
\rho^{-1}(U) \ar^{i_{\pi_2}}[r] \ar[d] & \rJ^1 M \ar[d] \\
U \ar^{\overline{i}_{\pi_2}}[r] & \T^* M
}
\]
together with a locally defined function $g\colon M\rightarrow \C$, such that under the embeddings the polarization $\pi_2$ goes to the vertical polarization, $M$ goes to the zero section, $\Lambda_M$ goes to $j^10$, $L$ goes to $\Gamma_{dg}\subset \T^* M$ and $\Lambda_L$ goes to $j^1 g\subset \rJ^1 M$.

Consider the product $S\times \overline{S}$. By \cref{prop:productcontactification} it has a contactification $Y\times^\ast \overline{Y}$. We have a Lagrangian submanifold $L\times M\subset S\times \overline{S}$ with a Legendrian lift $\Lambda_L\times \Lambda_M$. The diagonal $\Delta_S\subset S\times \overline{S}$ also defines a Lagrangian submanifold with a Legendrian lift $\Delta_Y\subset Y\times^\ast \overline{Y}$. We have a polarization $\pi_1\times \pi_2$ on $S\times \overline{S}$ which is transverse to both $L\times M$ and $\Delta$. Therefore, we may apply \cref{prop:polarizationtwoLagrangians} to $(Y\times^\ast \overline{Y}\rightarrow S\times \overline{S}, \pi_1\times \pi_2, L\times M, \Delta_S)$ (and again by shrinking $U$) we get a commutative diagram
\[
\xymatrix{
\rho^{-1}(U)\times^\ast \overline{\rho^{-1}(U)}\ar[r] \ar[d] & \rJ^1 (L\times M) \ar[d] \\
U\times \overline{U} \ar[r] & \T^* (L\times M)
}
\]
together with a locally defined function $h\colon L\times M\rightarrow \C$, such that under the embeddings the polarization $\pi_1\times\pi_2$ goes to the vertical polarization, $L\times M$ goes to the zero section, $\Lambda_L\times \Lambda_M$ goes to $j^10$, $\Delta_S$ goes to $\Gamma_{dh}\subset \T^*(L\times M)$ and $\Delta_Y$ goes to $j^1 h\subset \rJ^1 (L\times M)$. By uniqueness of the contact embeddings, we see that the top map is $i_1\times i_2\colon \rho^{-1}(U)\times^\ast \overline{\rho^{-1}(U)}\rightarrow \rJ^1 L\times^\ast \overline{\rJ^1 M}\cong \rJ^1(L\times M)$ and the bottom map is $\overline{i}_1\times \overline{i}_2\colon U\times \overline{U}\rightarrow \T^*L\times \overline{\T^*M}\cong \T^*(L\times M)$.

By construction we have a commutative diagram
\[
\xymatrix{
M\ar^{\pi_1\times \id}[rr] \ar[d] & & L\times M  \ar^{j^1(f\boxplus 0)}[d] \\
S \ar^-{\pi_1\times \pi_2}[r] & L\times M \ar^-{j^1 h}[r] & \rJ^1(L\times M)
}
\]
Its commutativity implies that $(\pi_1^{S\rightarrow L}\times \id)^* h = (\pi_1^{S\rightarrow L})^* f$ and hence $\Xi^* h = f$. The claim about $\Upsilon$ is proven similarly.
\end{proof}

In the setup of the previous proposition we get a local symplectomorphism $\T^* L\cong \T^* M$ and a local contactomorphism $\rJ^1 L\cong \rJ^1 M$ covering it. Their graphs may be described as follows:
\begin{itemize}
    \item The correspondence
    \[
    \xymatrix{
    & \Gamma_{dh} \ar_{p_1}[dl] \ar^{p_2^a}[dr] & \\
    \T^* L && \T^* M
    }
    \]
    where $p_1\colon \T^*(L\times M)\rightarrow \T^* L$ is the projection on the first factor and $p_2^a\colon \T^*(L\times M)\rightarrow \T^* M$ is the projection on the second factor post-composed with negation along the fibers is the graph of the symplectomorphism between $U\hookrightarrow \T^* L$ and $U\hookrightarrow \T^* M$.
    \item Under the identification $\rJ^1 L\times^\ast\overline{\rJ^1 M}\cong \rJ^1(L\times M)$ from \cref{ex:graphcontactification} the graph $\Gamma_\psi$ of the local contactomorphism $\rJ^1 L\cong \rJ^1 M$ corresponds to the Legendrian $j^1 h\subset \rJ^1(L\times M)$.
    \item The correspondence
    \[
    \xymatrix{
    & \T^*_Z (L\times \C\times M\times \C) \ar_{p_1}[dl] \ar^{p_2^a}[dr] & \\
    \T^*(L\times \C) && \T^*(M\times \C)
    }
    \]
    is the graph of the local homogeneous symplectomorphism between $\T^*(L\times \C)$ and $\T^*(M\times \C)$, where $Z = \{t_L - t_M = h\}\subset L\times \C\times M\times \C$.
\end{itemize}

Consider the map $\Xi\colon L\rightarrow L\times M$ from \cref{prop:twopolarizationstwoLagrangians}. By construction the critical loci of $f$ on $L$ and $h$ on $L\times M$ are both $L\cap M$. So, by \cite[Proposition 2.25]{JoycedCrit} there is a canonical quadratic form $q_\Xi$ on the normal bundle $N_\Xi|_{L\cap M}$ of the embedding $\Xi\colon L\rightarrow L\times M$. It has the following alternative description. For any $x\in L\cap M$ we have a sequence $\T_{M, x}, \T_{\pi_1, x}, \T_{\pi_2, x}, \T_{M, x}$ of Lagrangian subspaces of $\T_{S, x}$, where the subsequent subspaces are transverse. Therefore, we have the Maslov quadratic form $q(\T_M, \T_{\pi_1}, \T_{\pi_2})$ on $\T_M|_{L\cap M}$.

\begin{prop}
Under the natural isomorphism $N_\Xi\cong \T_M$ provided by the projection $L\times M\rightarrow M$ the quadratic form $q_\Xi$ on $N_\Xi|_{L\cap M}$ is equal to the Maslov quadratic form $q(\T_M, \T_{\pi_1}, \T_{\pi_2})$ on $\T_M|_{L\cap M}$. In particular, the quadratic form $q_\Xi$ on $\T_M|_{L\cap M}$ induces a quadratic form on $(\det \T^*_M)^{\otimes 2}|_{L\cap M}$ given by
\[\omega^{\otimes 2}\mapsto \frac{(\pi_{12}^{S\times S\rightarrow M\times M})^*(\omega\wedge \omega)}{\vol_S}|_{L\cap M},\]
where $\pi_{12}^{S\times S\rightarrow M\times M}$ is the composite $S\hookrightarrow S\times S\xrightarrow{\pi_1\times \pi_2} M\times M$.
\label{prop:quadraticformcriticalembedding}
\end{prop}
\begin{proof}
Consider $x\in L\cap M$. Choose local coordinates $\{l_a\}$ on $L$ and $\{m_i\}$ on $M$. We have a local diffeomorphism $L\times M\rightarrow \T^* L$ near $(x, x)$ given by $(l, m)\mapsto (l, \frac{\partial h}{\partial l_a})$. Similarly, we have a local diffeomorphism $L\times M\rightarrow \T^* M$ near $(x, x)$ given by $(l, m)\mapsto (m, -\frac{\partial h}{\partial m_i})$. The induced maps on tangent spaces are given by the matrices
\[
\begin{pmatrix}
\id & 0 \\
\frac{\partial^2 h}{\partial l_a\partial l_b} & \frac{\partial^2 h}{\partial l_a\partial m_i}
\end{pmatrix}\colon \T_{L\times M, (x, x)}\longrightarrow \T_{\T^* L, (x, 0)}
\]
and
\[
\begin{pmatrix}
0 & \id \\
-\frac{\partial^2 h}{\partial l_a\partial m_i} & -\frac{\partial^2 h}{\partial m_i\partial m_j}
\end{pmatrix}\colon \T_{L\times M, (x, x)}\longrightarrow \T_{\T^* M, (x, 0)}.
\]

By assumption the matrix $\frac{\partial^2 h}{\partial l_a\partial m_i}(x)$ is invertible. Let us denote by $A_{ia}$ its inverse. Thus, the composite $\T_{\T^* L, (x, 0)}\rightarrow \T_{L\times M, (x, x)}\rightarrow \T_{\T^* M, (x, 0)}$ is given by
\[
\begin{pmatrix}
0 & \id \\
-\frac{\partial^2 h}{\partial l_a\partial m_i} & -\frac{\partial^2 h}{\partial m_i\partial m_j}
\end{pmatrix}
\begin{pmatrix}
\id & 0 \\
-\sum_b\frac{\partial^2 h}{\partial l_a\partial l_b} A_{bi} & A_{ai}
\end{pmatrix}
\]
which is equal to
\[
\begin{pmatrix}
-\sum_b\frac{\partial^2 h}{\partial l_a\partial l_b} A_{bi} & A_{ai} \\
-\frac{\partial^2 h}{\partial l_a\partial m_i} + \sum_{b,j}\frac{\partial^2 h}{\partial m_i\partial m_j}\frac{\partial^2 h}{\partial l_a\partial l_b} A_{bj} & -\sum_j\frac{\partial^2 h}{\partial m_i\partial m_j} A_{aj}
\end{pmatrix}
\]
Therefore, the subspace $\T_{\pi_1, x}\subset \T_{\T^* M, (x, 0)}\cong \T_{M, x}\oplus \T^*_{M, x}$ is given by
\[\left\{\left(\sum_a A_{ai} v^a, -\sum_{a,j}\frac{\partial^2 h}{\partial m_i\partial m_j} A_{aj} v^a\right)\mid v^a\in \T_{M, x}\right\}.\]

Comparing with \cref{ex:Maslovexample} we deduce that the Maslov quadratic form $q(\T_{M, x}, \T_{\pi_1, x}, \T_{\pi_2, x})$ is given by $\frac{\partial^2 h}{\partial m_i\partial m_j}$.

By construction the embedding $\Xi\colon L\hookrightarrow L\times M$ is given by the equation $\frac{\partial h}{\partial m_i} = 0$. Therefore, the induced quadratic form $q_\Xi$ on the normal bundle is given by $\frac{\partial^2 h}{\partial m_i\partial m_j}$ which, as explained above, coincides with the Maslov quadratic form $q(\T_{M, x}, \T_{\pi_1, x}, \T_{\pi_2, x})$.
\end{proof}

\section{Monodromic sheaves}

\subsection{Monodromic \texorpdfstring{$D$}{D}-modules}
\label{sect:monodromicDmodules}

We begin with the following well-known result.

\begin{prop}\label{prop:monodromyeigenvalueblocks}
Let $\cC$ be an idempotent complete $\C$-linear category and $M\in\cC$ an object with $\End_\cC(M)$ finite-dimensional. Given an endomorphism $\theta\colon M\rightarrow M$ there is a decomposition
\[M\cong \bigoplus_{\lambda\in\C} M^{\lambda},\]
where only finitely many $M^\lambda$ are nonzero, and where $\theta$ acts on $M^\lambda$ with generalized eigenvalue $\lambda$.
\end{prop}

In this paper we apply the above result to the category of regular holonomic $D$-modules and perverse sheaves on a complex manifold.

Let $X$ be a complex manifold and suppose $\gamma:Y\to X$ is either a principal $\C^\times$-bundle or a holomorphic vector bundle. In either case, we denote by $\theta \in \cD_Y$ the vector field induced by the $\C^\times$ action. In this section we summarize some results about monodromic $D$-modules in the analytic setting; see, e.g. \cite[Section 4.1]{SaitoTS} for further details.

\begin{defn}
    A coherent $\cD_Y$-module $\cM$ is called \defterm{monodromic} if $\gamma_\ast \cM$ is generated as a $\gamma_\ast \cD_Y$-module by local sections $u$ such that there exists $b(\theta) \in \C[\theta]$ such that $b(\theta)u=0$. 
\end{defn}

For a complex manifold $Y$ denote by $\Mod_{\rh}(\cD_Y)$ the stack (that is, sheaf of categories) of regular holonomic $D$-modules on $Y$. We denote by $\Mod_{\rh,\mon}(\cD_Y)$ the substack of $\gamma_\ast \Mod_{\rh}(\cD_Y)$ consisting of monodromic modules. 

Let $\cD_{[Y]} \subseteq \gamma_\ast \cD_Y$ denote the subsheaf consisting of operators which are polynomial in the fibers of $\gamma$. Equivalently, $\cD_{[Y]}$ is the subsheaf of $\gamma_\ast \cD_Y$ consisting of sections $P$ such that there exists $b(\theta) \in \C[\theta]$ with $[b(\theta),P]=0$. Similarly, for a $\cD_Y$-module $\cM$, we denote by $[\cM]$ the subsheaf of $\gamma_\ast \cM$ consisting of local sections $u$ such that there exists $b(\theta) \in \C[\theta]$ such that $b(\theta)u=0$. Then $[\cM]$ is naturally a $\cD_{[Y]}$-module and there is a natural map 
\[
\cD_Y \otimes_{\gamma^{-1} \cD_{[Y]}} [\cM] \to \cM
\]
which is an isomorphism precisely when $\cM$ is monodromic. Note that 
\[
[\cM] = \bigoplus_{\lambda \in \C} \cM^\lambda
\]
where
\[\cM^\lambda = \bigcup_{l\geq 0}\ker((\theta-\lambda\cdot \id)^l\colon \gamma_* \cM\rightarrow \gamma_* \cM)\]
is the generalized eigenspace with eigenvalue $\lambda\in\C$. 

Now let us specialize to the case when $Y=X\times \C^\times$ and $\gamma\colon Y\to X$ is the projection map. In that case, we have $\cD_{[Y]} = \cD_X[\theta]\langle t,t^{-1}\rangle$, where $t$ is the coordinate on the $\C^\times$-factor and $\theta = t\partial_t$. Note that if $\cM$ is a $\cD_Y$-module then acting by $t$ induces an isomorphism $\cM^\lambda \to \cM^{\lambda + 1}$. Let
\[[\cM]_{\EigenRange} = \bigoplus_{\lambda\in \EigenRange} \cM^\lambda,\]
where we recall that $\EigenRange = \{ \lambda \in \C \mid -1<\Re \lambda \leq 0\}$. 

Let us denote by $\Mod_{\rh,\EigenRange}(\cD_X[\theta])$ the full subcategory of $\cD_X[\theta]$-modules for which the action of the $\theta$ is locally finite with eigenvalues contained in $\EigenRange$ and which are regular holonomic as $\cD_X$-modules.
\begin{prop}
    The assignment $\cM \mapsto [\cM]_\EigenRange$ defines an equivalence of categories between 
    $\Mod_{\rh,\mon}(\cD_Y)$ and $\Mod_{\rh,\EigenRange}(\cD_X[\theta])$.
\end{prop}

It will be convenient to introduce another perspective on the category of monodromic $D$-modules. Consider the category $\Mod_{\rh,\aut}(\cD_X)$ of regular holonomic $\cD_X$-modules $\cM$ equipped with an automorphism $T\colon \cM\rightarrow \cM$. There is a functor
\[\Mod_{\rh,\EigenRange}(\cD_X[\theta])\longrightarrow \Mod_{\rh,\aut}(\cD_X)\]
given by setting $T=\exp(-2\pi i\theta)$. It also identifies $\Mod_{\rh,\aut}(\cD_X)$ with the full subcategory of regular holonomic $\cD_X[\theta]$-modules, where the eigenvalues of $\theta$ are contained in $\EigenRange$.

The resulting equivalence
\[\Mod_{\rh,\mon}(\cD_Y)\cong \Mod_{\rh,\aut}(\cD_X)\]
may be understood as follows. Consider the inclusion $i\colon X\times\{1\}\hookrightarrow Y$. Then the composite
\[\Mod_{\rh,\mon}(\cD_Y)\cong \Mod_{\rh,\aut}(\cD_X)\longrightarrow \Mod_{\rh}(\cD_X)\]
is given by the pullback $i^*$. The automorphism $T$ is given by the monodromy along the $\C^\times$ direction.

\subsection{Monodromic perverse sheaves}
\label{sect:monodromicperverse}

Let us now describe analogous objects on the other side of the Riemann--Hilbert correspondence. Denote by $\Pervaut(X)$ the stack of $\C$-linear perverse sheaves $\cF$ on a complex manifold $X$ equipped with an automorphism $T\colon \cF\rightarrow \cF$. As in the case of $D$-modules, this stack has the following two equivalent presentations:
\begin{itemize}
    \item The stack of perverse sheaves $\cF$ on $X$ equipped with an endomorphism $\theta\colon \cF\rightarrow \cF$ whose eigenvalues lie in $\EigenRange$.
    \item The stack of perverse sheaves on $X\times \C^\times$, which are locally constant along the $\C^\times$ orbits.
\end{itemize}

Given a regular holonomic left $\cD_X$-module $\cL$ and a regular holonomic right $\cD_X$-module $\cM$, the complex of sheaves $\cM\otimes^\bL_{\cD_X} \cL$ is a perverse sheaf on $X$. This statement generalizes to monodromic $D$-modules as follows. Let $Y=X\times \C^\times$ with $\gamma\colon Y\rightarrow X$ the projection.

\begin{prop}
Let $\cL$ (respectively, $\cM$) be a regular holonomic monodromic left (respectively, right) $\cD_Y$-module, so that $[\cL]_\EigenRange$ (respectively, $[\cM]_{\EigenRange}$) is a left (respectively, right) $\cD_X[\theta]$-module with $\theta$-eigenvalues contained in $\EigenRange$.
Then:
\begin{itemize}
    \item The perverse sheaf
    \[\cM\otimes^\bL_{\cD_Y} \cL\]
    on $Y$ is locally constant along the $\C^\times$ orbits.
    \item There is an isomorphism of perverse sheaves
    \begin{equation}[\cM]_\EigenRange \otimes_{\cD_X}^\bL [\cL]_\EigenRange \cong (\cM\otimes^\bL_{\cD_Y} \cL)|_{X\times \{1\}}[-1].
    \label{eq:monodromic}
    \end{equation}
    \item Consider the endomorphism $\theta$ on the left-hand side induced by the $\theta$-action on $[\cL]_\EigenRange$ and $[\cM]_\EigenRange$ and let $T = \exp(2\pi i \theta)$. Then the action of $T$ on the left side of \eqref{eq:monodromic} corresponds to the monodromy on along the $\C^\times$ direction the right.
\end{itemize}
\label{prop:monodromicRH}
\end{prop}

We will also consider perverse sheaves of $\C[\![\hbar]\!]$-modules as in \cite[Section 5]{DAgnoloGuillermouSchapira}. We say a perverse sheaf $\cL\in\Perv(X; \C[\![\hbar]\!])$ is \defterm{torsion-free} if the map $\hbar\colon \cL\rightarrow \cL$ is a monomorphism in the abelian category $\Perv(X; \C[\![\hbar]\!])$.

We have a natural functor
\[\bD_c(X; \C[\![\hbar]\!])\longrightarrow \bD_c(X; \C[\hbar]/\hbar^n)\]
given by $\cL\mapsto \cL\otimes^{\bL}_{\C[\![\hbar]\!]} \C[\hbar]/\hbar^n$ which we for brevity denote by $\cL\mapsto \cL/\hbar^n$. If $\cL$ is a torsion-free perverse sheaf, then $\cL/\hbar^n$ can be understood as the quotient in the abelian category of perverse sheaves.

\begin{defn}
A \defterm{differential perverse sheaf} is a pair $(\cM,D)$, where $\cM$ is a $\C(\!(\hbar)\!)$-linear perverse sheaf and $D$ a $\C$-linear endomorphism of the underlying complex of sheaves of $\C$-vector spaces satisfying 
\[
D( f\eta) = \hbar\frac{df}{d\hbar}\eta + fD(\eta)
\]
for any local section $\eta$ of $\cM$, and $f\in \C(\!(\hbar)\!)$. We denote by $\Pervdiff(X)$ the stack of differential perverse sheaves.
\end{defn}

\begin{defn}
Let $\cM\in\Pervdiff(X)$ be a differential perverse sheaf. A \defterm{lattice} $\cL\subset \cM$ is a torsion-free $\C[\![\hbar]\!]$-perverse sheaf $\cL$ together with an isomorphism $\cL[\hbar^{-1}]\cong \cM$ which is preserved by the action of $D$ with the following property: all eigenvalues of the action of $D$ on the $\C$-linear perverse sheaf $\cL/\hbar$ lie in $\EigenRange$. If $\cM\in\Pervdiff(X)$ has a lattice, we say that it has \defterm{regular singularities}.
\end{defn}

We will need the following Yoneda lemma for monodromic perverse sheaves. Consider a $\C^\times$-torsor $\gamma\colon \tilde{Y}\rightarrow Y$ and let
\[\Perv_{\C^\times}(\tilde{Y})\subset \Perv(\tilde{Y})\]
be the full substack of perverse sheaves, which are locally constant along the fibers of $\gamma$, and
\[\LocSys(\tilde{Y})\subset \Perv(\tilde{Y})\]
the full substack of local systems. There is a natural action of the symmetric monoidal stack $\LocSys(Y)$ on $\LocSys(\tilde{Y})$ (via $\gamma^*$) and $\Perv(Y)$. For $\cF\in\Perv_{\C^\times}(\tilde{Y})$ the derived pushforward $\gamma_*\cF$ has cohomology only in perverse degrees $-1$ and $0$.

Consider the functor
\[\gamma_*\Perv_{\C^\times}(\tilde{Y})\longrightarrow \Fun_{\LocSys(Y)}(\gamma_*\LocSys(\tilde{Y}), \Perv(Y))\]
of stacks on $Y$ given by sending $\cF\in\Perv_{\C^\times}(\tilde{Y})$ to the $\LocSys(Y)$-linear functor $V\in\LocSys(\tilde{Y})\mapsto {}^p\cH^{-1}(\gamma_*(V\otimes \cF))$. The following lemma implies that this functor is fully-faithful.

\begin{lm}\label{lm:automorphismYoneda}
 For every $\cF_1, \cF_2\in \gamma_*\Perv_{\C^\times}(\tilde{Y})$ the natural morphism
\[\RcHom(\cF_1, \cF_2)\longrightarrow \RcHom({}^p\cH^{-1}(\gamma_*((-)\otimes \cF_1)), {}^p\cH^{-1}(\gamma_*((-)\otimes \cF_2)))\]
is an isomorphism of complexes of sheaves on $Y$.
\end{lm}
\begin{proof}
The question is local on $Y$, so we may assume that $\tilde{Y}=Y\times\C^\times\rightarrow Y$ is a trivial $\C^\times$-torsor and $Y$ is contractible. By \cref{prop:monodromyeigenvalueblocks} both $\gamma_*\Perv_{\C^\times}(\tilde{Y})\cong\Pervaut(Y)$ and $\gamma_*\LocSys(\tilde{Y})$ split into blocks according to the generalized eigenvalue of the monodromy operator. This allows us to reduce the statement to a similar statement for the substack of perverse sheaves $\Pervaut^0(Y)\subset \Pervaut(Y)$ consisting of pairs $(\cF,T)\in\Pervaut(Y)$, where the automorphism $T\colon \cF\rightarrow \cF$ is unipotent, i.e. $(T-\id)^N = 0\colon \cF\rightarrow \cF$ for some $N$. Moreover, since $Y$ is assumed to be contractible, $\gamma_*\LocSys(\tilde{Y})$ is given by the stack of $\C[T^{\pm 1}]$-modules. For a local system $V$ of $\C[T^{\pm 1}]$-modules and $\cF\in\Pervaut^0(Y)$ we have
\[{}^p\cH^{-1}(\gamma_*(V\otimes \cF))\cong (V\otimes \cF)^T,\]
where on the right we consider $T$-invariants. This sheaf vanishes if the $T$-action on $V$ is not unipotent, so we will further restrict to such modules.

So, for $\cF_1,\cF_2\in\Pervaut^0(Y)$ we have to show that the natural morphism
\[\RcHom_{\C[T^{\pm 1}]}(\cF_1, \cF_2)\longrightarrow \RcHom_{\C}(((-)\otimes \cF_1)^T, ((-)\otimes \cF_2)^T)\]
is an isomorphism. Fix a natural number $N$, so that $(T-1)^N$ annihilates both $\cF_1$ and $\cF_2$. In this case the left-hand side becomes isomorphic to
\[\RcHom_{\C[T^{\pm 1}]/(T-1)^N}(\cF_1, \cF_2)\]
and the right-hand side becomes isomorphic to
\[\RcHom_{\C[T^{\pm 1}]/(T-1)^N}((\C[T^{\pm 1}]/(T-1)^N\otimes \cF_1)^T, (\C[T^{\pm 1}]/(T-1)^N\otimes \cF_2)^T),\]
which is again isomorphic to $\RcHom_{\C[T^{\pm 1}]/(T-1)^N}(\cF_1, \cF_2)$ by an explicit computation of the invariants.
\end{proof}

\subsection{The Riemann--Hilbert functor}
Our next goal is to define an inverse Riemann--Hilbert functor
\[
\RH^{-1}\colon \Pervaut(X) \to \Pervdiff(X)
\]
relating differential perverse sheaves and perverse sheaves equipped with an automorphism.

\begin{defn}
Given a sheaf $\cF\in\bD(X;\C)$, we denote by $\cF[\![\hbar]\!]$ the inverse limit
\[
\cF[\![\hbar]\!] = \lim_m\left(\cF \otimes_\C \C[\hbar]/\hbar^m\right)\in\bD(X;\C[\![\hbar]\!]).
\]

We write 
\[
\cF(\!(\hbar)\!) = \cF[\![\hbar]\!] \otimes_{\C[\![\hbar]\!]} \C(\!(\hbar)\!)\in\bD(X;\C(\!(\hbar)\!)).
\]
\end{defn}

\begin{remark}
As tensor products preserve colimits, there is a natural isomorphism
\[
\cF(\!(\hbar)\!) = \colim_n\left(\hbar^{-n} \cF[\![\hbar]\!]\right) \cong \colim_n \lim_m\left(\cF \otimes_\C \hbar^{-n} \C[\hbar]/\hbar^m\right).
\]
\end{remark}

We have the following two technical lemmas.
\begin{lm}
Suppose $\cF \in \bD_c(X;\C)$ is a constructible complex. Then the natural maps
\[
\cF \otimes_\C \C[\![\hbar]\!] \to \cF[\![\hbar]\!]
\]
and
\[
\cF \otimes_\C \C(\!(\hbar)\!) \to \cF(\!(\hbar)\!)
\]
are isomorphisms.
\label{lm:constructibletensorcomplete}
\end{lm}
\begin{proof}
As $\cF$ is a constructible complex, for any point $x\in X$, there is a basis of open neighborhoods $U_x$ of $x$ such that the map $\cF_x \to \cF(U_x)$ is a quasi-isomorphism. In particular, taking the stalk commutes with the inverse limit. Thus, we are reduced to checking that the natural maps
\[
\cF_x \otimes_\C \C[\![\hbar]\!] \to \lim_m \cF_x \otimes_\C \C[\hbar]/\hbar^m
\]
are quasi-isomorphisms of complexes. The projective diagram of complexes is naturally split, so the inverse limit commutes with cohomology. The result then follows from the statement the corresponding statement for finite-dimensional vector spaces, which is well-known. The second statement follows immediately from the first. 
\end{proof}

\begin{lm}
Suppose $\cF\in\bD_c(X;\C[\![\hbar]\!])$ is a constructible complex. Then the natural morphism
\[\cF\longrightarrow \lim_n \cF/\hbar^n\]
is an isomorphism.
\label{lm:constructiblecomplete}
\end{lm}
\begin{proof}
By \cite[Proposition 7.1.7]{KashiwaraSchapira} a constructible complex $\cF$ is cohomologically complete, i.e.
\[\RcHom_{\C_X[\![\hbar]\!]}(\C_X(\!(\hbar)\!), \cF) = 0.\]
The fact that for a cohomologically complete $\C_X[\![\hbar]\!]$-module the map
\[\cF\longrightarrow \lim_n\cF/\hbar^n\]
is an isomorphism is the content of \cite[Theorem 7.8]{PSY} in the case $X=\pt$ and the same proof works for sheaves.
\end{proof}

We are now ready to define the functor $\RH^{-1}$. Given a perverse sheaf $\cF$ equipped with an automorphism $T$ we can write $T=\exp(-2\pi i M)$ for a unique endomorphism $M\colon \cF\rightarrow \cF$ whose eigenvalues lie in $\EigenRange\subset\C$. Then we define $\RH^{-1}(\cF)$ to be the pair $(\cF(\!(\hbar)\!),D)$, where
\[D(f\otimes \eta) = f\otimes (M\eta) + \hbar \frac{df}{d\hbar}\otimes \eta\]
for any local section $\eta$ of $\cF$ and $f\in\C(\!(\hbar)\!)$.

Conversely, if $(\cM, D)\in\Pervdiff(X)$ is a differential perverse sheaf with a lattice $\cL\subset \cM$, consider $\cF=\cL/\hbar\in\Perv(X;\C)$. By assumption it is preserved by $D$ and we define the monodromy automorphism $T\colon \cF\rightarrow \cF$ to be $T=\exp(-2\pi i D)$. In this way we get an object $(\cF, T)\in\Pervaut(X)$.

\begin{prop}\label{prop:Riemann Hilbert}
The functor $\RH^{-1}\colon\Pervaut(X)\rightarrow \Pervdiff(X)$ is fully faithful. Moreover, if $\cM\in\Pervdiff(U)$ is a differential perverse sheaf with a lattice $\cL\subset \cM$, then $\cM\cong \RH^{-1}(\cL/\hbar)$. In particular, the essential image of $\RH^{-1}$ consists of differential perverse sheaves with regular singularities.
\end{prop}
\begin{proof}
Given a morphism $\cF_1\rightarrow \cF_2$ in $\Pervaut(X)$, a morphism $\RH^{-1}(\cF_1)\rightarrow \RH^{-1}(\cF_2)$ of $\C(\!(\hbar)\!)$-linear perverse sheaves is uniquely specified by its restriction $\cF_1\rightarrow \cF_2(\!(\hbar)\!)$ to $\cF_1\hbar^0\subset \cF_1(\!(\hbar)\!)$. Let $\cF_1^\lambda\subset \cF_1$ be a generalized eigenspace of $D$ with eigenvalue $m\in \EigenRange$, where $\exp(-2\pi i m)=\lambda$. The map $\cF_1\rightarrow \cF_2(\!(\hbar)\!)$ will therefore send $\cF_1^\lambda$ to the $m$-generalized eigenspace of $\cF_2(\!(\hbar)\!)$ which is precisely $\cF_2^\lambda \hbar^0\subset \cF_2(\!(\hbar)\!)$. The resulting map $\cF_1^\lambda\rightarrow \cF_2^\lambda$ is compatible with the nilpotent part of $D$ precisely when it is compatible with the unipotent part of $T$. This proves fully faithfulness.

Now suppose $\cE\in\Pervdiff(X)$ has a lattice $\cL\subset \cM$ and consider the perverse sheaf $\cF=\cL/\hbar$ equipped with an automorphism. Consider the extension
\[0\longrightarrow \cL/\hbar\xrightarrow{\hbar}\cL/\hbar^2\longrightarrow \cL/\hbar \longrightarrow 0.\]
The perverse sheaves $\cF=\cL/\hbar$ and $\cL/\hbar^2$ split into generalized eigenspaces for $D$. But since no two eigenvalues on $\cF$ differ by $1$, the above exact sequence is canonically split, i.e. we can identify $\cL/\hbar^2\cong \cF\otimes_\C \C[\hbar]/\hbar^2$. Proceeding by induction we identify
\[\cL/\hbar^n\cong \cF\otimes_\C \C[\hbar]/\hbar^n\]
compatibly with the maps obtained by modding out by a power of $\hbar$. Therefore, by \cref{lm:constructibletensorcomplete}
\[\lim_n \cL/\hbar^n\cong \cF[\![\hbar]\!].\]
By \cref{lm:constructiblecomplete} the natural morphism
\[\cL\longrightarrow \lim_n \cL/\hbar^n\]
is an isomorphism, so this identifies $\cL\cong\cF[\![\hbar]\!]$. Therefore, $\cE\cong \cF(\!(\hbar)\!)=\RH^{-1}(\cF)$. Under this isomorphism the action of $D$ on $\cE$ coincides with the action of $D$ coming from the monodromy operator $T$ on $\cF$.
\end{proof}

\section{Microdifferential systems}\label{sect:microdifferential}

Throughout this section $X$ denotes a complex manifold.

\subsection{Microdifferential operators} \label{sect:microdiffops}

In this section we recall some algebras of microdifferential operators that we will use. We refer to \cite{KashiwaraDmodules,KashiwaraIntroduction} for more details. Recall the punctured cotangent bundle $\oT^*X=\T^* X\setminus \T^*_X X$ and the projectivized cotangent bundle $\P^* X = \oT^* X / \C^\times$ with $\gamma\colon \oT^* X\rightarrow \P^* X$ the natural projection. We denote the projections $\T^* X\rightarrow X$ and $\P^* X\rightarrow X$ by the same letter $\pi$.

Consider the sheaf $\hE_X$ on $\P^* X$ of \defterm{formal microdifferential operators} introduced in \cite{SatoKawaiKashiwara}. It has a filtration
\[\hE_X = \bigcup_{m\in\Z} \hE_X(m)\]
by order of the microdifferential operator. Denote by \[\sigma_m\colon \hE_X(m)\longrightarrow \cO_{\T^* X}(m)\]
the \defterm{principal symbol}, where $\cO_{\T^* X}(m)\subset \gamma_*\cO_{\oT^* X}$ consists of functions on $\oT^* X$ homogeneous of order $m$. The principal symbol establishes an isomorphism
\[\gr \hE_X\cong \bigoplus_{m\in\Z}\cO_{\T^* X}(m).\]

\begin{remark}
In \cref{sect:categories} we will also encounter the sheaf $\cE_X$ of non-formal microdifferential operators, so that $\hE_X$ is obtained by completion from $\cE_X$.
\end{remark}

By construction there is a natural embedding
\[\pi^{-1}\cD_X\rightarrow \hE_X\]
of differential into microdifferential operators.

\begin{defn}
A coherent $\hE_X$-module is \defterm{holonomic} if its support is a Legendrian subvariety of $\P^* X$.
\end{defn}

Let $\Lambda\subset \P^* X$ a Legendrian subvariety and $\tilde{\Lambda}\subset \oT^* X$ the corresponding conic Lagrangian subvariety. Let $\hE_{\Lambda/X}\subset \hE_X|_\Lambda$ be the subalgebra generated by
\[I_\Lambda = \{P\in \hE_X(1)|_\Lambda\mid \sigma_1(P)|_\Lambda=0\}.\]

\begin{defn}\label{def:Erh}
A holonomic $\hE_X$-module $\cM$ is \defterm{regular holonomic} if locally there is a coherent $\hE_X(0)$-submodule $\cM_0\subset \cM$ and a Legendrian subvariety $\Lambda\subset \P^*X$ containing $\supp \cM$ such that $I_\Lambda \cM_0\subset \cM_0$ and $\hE_X\cM_0 = \cM$.
\end{defn}

\begin{remark}
A $\cD_X$-module $\cM$ is (regular) holonomic if, and only if, the $\hE_X$-module $\hE_X\otimes_{\pi^{-1}\cD_X} \pi^{-1} \cM$ is (regular) holonomic.
\end{remark}

\subsection{Specialization and microlocalization}

Let us now recall the theory of $V$-filtration for microdifferential operators. Let $Z\subset X$ be a submanifold with $J_Z\subset \cO_X|_Z$ the vanishing ideal. Define the Kashiwara--Malgrange filtration of $\cD_X|_Z$ by
\[V^k_Z \cD_X = \{P\in\cD_X|_Z\mid P J_Z^j\subset J_Z^{j+k},\ \forall j,j+k\geq 0\}.\]
As shown in \cite{KashiwaraVanishing}, there is a natural embedding
\[\gr_Z \cD_X\subset \tau_* \cD_{\T_Z X},\]
where $\tau\colon \T_Z X\rightarrow Z$ is the projection, whose image $\cD_{[\T_Z X]}$ consists of differential operators polynomial along the fibers of $\tau$. Let $\theta\in V^0_Z\cD_X$ be any operator lifting the Euler vector field in $\gr_Z \cD_X=\cD_{[\T_Z X]}$ which acts on $J_Z/J_Z^2$ by the identity.

If $\cM\in\Mod_{\rh}(\cD_X)$ is a regular holonomic $\cD_X$-module and $u\in\cM$ is a local nonvanishing section, then there is a polynomial $b_u(s)\in\C[s]$ such that $b_u(\theta)u\in V^1_Z(\cD_X) u$. The \defterm{order} $\ord_Z(u)$ of $u\in\cM$ is the set of zeros of its $b$-function. The \defterm{canonical $V$-filtration} of $\cM$ is given by the subspaces
\[V^k_Z\cM = \{u\in\cM|_Z\mid \Re\ord_Z(u)\geq k-1\},\qquad k\in\Z.\]
Denote by $\gr_Z\cM$ the associated graded module with respect to the canonical $V$-filtration $V^k_Z\cM$. We refer to \cite{KashiwaraVanishing,MonteiroFernandes} for details on the following functors.

\begin{defn}
Let $Z\subset X$ be a submanifold.
\begin{itemize}
\item The \defterm{specialization functor}
\[\nu^{dR}_Z\colon \Mod_{\rh}(\cD_X)\longrightarrow \tau_*\Mod_{\rh, \mon}(\cD_{\T_Z X})\]
is
\[\nu^{dR}_Z(\cM) = \cD_{\T_Z X}\otimes_{\tau^{-1} \gr_Z \cD_X} \tau^{-1}\gr_Z \cM,\]
where $\Mod_{\rh, \mon}(\cD_{\T_Z X})\subset \Mod_{\rh}(\cD_{\T_Z X})$ is the subcategory of monodromic regular holonomic $D$-modules.
\item The \defterm{microlocalization functor}
\[\mu^{dR}_Z\colon \Mod_{\rh}(\cD_X)\longrightarrow \pi_*\Mod_{\rh, \mon}(\cD_{\T^*_Z X})\]
is the composite of the specialization functor $\nu^{dR}_Z$ with the formal Fourier transform $\Mod_{\rh,\mon}(\cD_{T^\ast_Z X}) \to \Mod_{\rh,\mon}(\cD_{\T_Z X})$, where $\pi\colon \T^* X\rightarrow X$ is the projection. 
\end{itemize}
\end{defn}

There is a microlocal analog of the above construction as follows. Let $\Lambda\subset \P^* X$ be a Legendrian submanifold and $\tilde{\Lambda}\subset \oT^* X$ the corresponding conic Lagrangian submanifold. Consider the filtration on $\hE_X|_\Lambda$
defined by
\[\hE_{\Lambda/X}(m) = \hE_{\Lambda/X}\hE_X(m)|_\Lambda.\]
Consider the line bundle \[K_{\tilde{\Lambda}|X}=K_{\tilde{\Lambda}}\otimes \pi^* K_X^{-1}\]
on $\tilde{\Lambda}$. We denote by $\cD_{\widetilde{\Lambda}}^{K^{1/2}_{\tilde{\Lambda}|X}}$ the sheaf of twisted differential operators with respect to a square root of the line bundle $K_{\tilde{\Lambda}|X}$. 
By \cite[Lemma 1.5.1]{KashiwaraKawaiHolonomic} there is a natural embedding
\[L\colon\gr_\Lambda \hE_X\hookrightarrow \gamma_* \cD_{\tilde{\Lambda}}^{K^{1/2}_{\tilde{\Lambda}|X}},\]
where $\gamma\colon \tilde{\Lambda}\rightarrow \Lambda$, whose image $\cD_{[\tilde{\Lambda}]}^{K^{1/2}_{\tilde{\Lambda}|X}}$ consists of differential operators polynomial along the fibers of $\gamma$.

Consider the Euler vector field on $\T^* X$. Since $\tilde{\Lambda}\subset \T^* X$ is homogeneous, the Lie derivative along the Euler vector field defines an element $\theta'\in\cD_{\tilde{\Lambda}}^{K^{1/2}_{\tilde{\Lambda}|X}}$. We denote by $\theta'\in \hE_X(1)|_\Lambda$ any preimage (such a lift is characterized by \cite[Equations (1.5.5) and (1.5.6)]{KashiwaraKawaiHolonomic}). If $\cM$ is a regular holonomic $\hE_X$-module and $u\in\cM$ is a local nonvanishing section, then by \cite[Theorem 4.1.1]{KashiwaraKawaiMicrolocalization} there is a polynomial $b_u(s)\in\C[s]$ such that $b_u(\theta')u\in \hE_{\Lambda/X}(-1)u$. The \defterm{order} $\ord_\Lambda(u) \subseteq \C$ of $u\in\cM$ is the set of zeros of its $b$-function. For a fixed $\alpha\in\R$ the \defterm{canonical $V$-filtration} of $\cM$ is given by the subspaces
\[V^k_\Lambda\cM = \{u\in\cM|_\Lambda\mid \Re\ord_\Lambda(u)\leq \alpha-k\},\qquad k\in\Z.\]
Denote by $\gr_\Lambda\cM$ the associated graded module with respect to the canonical $V$-filtration $V^k_\Lambda\cM$.

\begin{defn}
Let $\Lambda\subset \P^* X$ be a Legendrian submanifold and $\tilde{\Lambda}\subset \oT^* X$ the corresponding conic Lagrangian. The \defterm{microlocalization functor}
\[\mu^{dR}_\Lambda\colon \Mod_{\rh}(\hE_X)\longrightarrow \Mod_{\rh, \mon}\left(\cD^{K_{\tilde{\Lambda}|X}^{1/2}}_{\tilde{\Lambda}}\right)\]
is
\[\mu^{dR}_\Lambda(\cM) = \cD^{K_{\tilde{\Lambda}|X}^{1/2}}_{\tilde{\Lambda}}\otimes_{\gamma^{-1} \gr_\Lambda \hE_X} \gamma^{-1}\gr_\Lambda \cM.\]
\end{defn}

Suppose $\cM$ is a regular holonomic $\cD_X$-module, $Z\subset X$ a submanifold and $\Lambda=\P^*_ZX\subset \P^*X$ its projectivized conormal bundle. If we denote by $K_{Z|X} = K_Z\otimes K_X^{-1}$ the relative canonical bundle, then there is a natural isomorphism
\[K_{\tilde{\Lambda}|X}\cong \pi^* K_{Z|X}^{\otimes 2}.\]

Consider the microlocalization
\[\tilde{\cM} = \hE_X\otimes_{\pi^{-1}\cD_X} \pi^{-1}\cM.\]

Fix $\alpha = \codim Z/2+1$. Then we may consider the $V$-filtration on $\cM$ along $Z$ or its microlocal version on $\tilde{\cM}$ along $\Lambda$. They are related as follows (see \cite[Remark 1.2.4]{MonteiroFernandes}): the natural map $\pi^{-1}\cM\rightarrow \tilde{\cM}$ induces an isomorphism
\begin{equation}
\pi^{-1}\gr^k_Z \cM \cong K_{Z|X}^{-1}\otimes_{\pi^{-1}\cO_X} \gr^k_\Lambda \tilde{\cM}
\label{eq:microlocalVcomparison}
\end{equation}
for any $k\leq 0$.

\subsection{Half twists}

Let $X$ be a complex manifold. The adjoint of a differential operator defines a map
\[\ast\colon \cD_X\longrightarrow K_X\otimes \cD_X\otimes K_X^{-1}.\]
One may define the sheaf of half-twisted differential operators $\cD^{\sqrt{v}}_X$, so that if $K_X^{1/2}$ is a square root of the canonical bundle, one has
\[\cD^{\sqrt{v}}_X = K^{1/2}_X\otimes_{\cO_X} \cD_X\otimes_{\cO_X} K^{-1/2}_X.\]
The sheaf $\cD^{\sqrt{v}}_X$ carries an antiinvolution $\ast$ given by passing to the adjoint differential operator. A vector field $\xi\in T_X$ defines a section of $\cD^{\sqrt{v}}_X$ via Lie derivative. By construction $\xi^* = -\xi$.

\begin{remark}
The sheaf $\cD^{\sqrt{v}}_X$ is independent of the choice of the square root of the canonical bundle.
\end{remark}

Similarly, let
\[\hE^{\sqrt{v}}_X = \pi^{-1} K^{1/2}_X\otimes_{\pi^{-1} \cO_X} \hE_X\otimes_{\pi^{-1}\cO_X} \pi^{-1} K^{-1/2}_X\]
be the sheaf of half-twisted microdifferential operators. The antiinvolution $\ast$ on $\cD^{\sqrt{v}}_X$ extends to one on $\hE^{\sqrt{v}}_X$.

The previous constructions extend to the half-twisted setting in a straightforward way. For instance, for a Legendrian submanifold $\Lambda\subset \P^* X$ there is a subalgebra $\hE^{\sqrt{v}}_{\Lambda/X}\subset \hE^{\sqrt{v}}_X|_\Lambda$, an embedding
\[L\colon \gr_\Lambda \hE^{\sqrt{v}}_X\hookrightarrow \gamma_* \cD_{\tilde{\Lambda}}^{\sqrt{v}}\]
whose image $\cD_{[\tilde{\Lambda}]}^{\sqrt{v}}$ consists of twisted differential operators polynomial along the fibers of $\gamma$ and a microlocalization functor
\[\mu^{dR}_\Lambda\colon \Mod_{\rh}(\hE^{\sqrt{v}}_X)\longrightarrow \gamma_*\Mod_{\rh, \mon}\left(\cD^{\sqrt{v}}_{\tilde{\Lambda}}\right).\]

Any left $\hE_X^{\sqrt{v}}$-module $\cM$ may be considered as a right $\hE_X^{\sqrt{v}}$-module via the formula $uP := P^\ast u$ for $P\in \hE_X^{\sqrt{v}}$, $u\in \cM$. We will freely pass between left and right modules using the antiinvolution $\ast$.

\subsection{Microdifferential modules}

Let $\Lambda\subset \P^* X$ be a Legendrian submanifold and $\tilde{\Lambda}\subset \oT^* X$ the corresponding conic Lagrangian submanifold.

\begin{defn}
Let $\cM$ be a coherent $\hE_X$-module supported on $\Lambda$. It is \defterm{regular along $\Lambda$} if there is locally a coherent $\hE_{\Lambda/X}$-submodule $\cM_0\subset \cM|_\Lambda$ generating $\cM|_\Lambda$ over $\hE_X|_\Lambda$.
\end{defn}

If $\cM$ is regular along a Legendrian $\Lambda\subset \P^* X$, then it is a regular holonomic $\hE_X$-module and $\mu_\Lambda^{dR}(\cM)$ is a vector bundle of finite rank on $\tilde{\Lambda}$.

\begin{defn}
An $\hE_X$-module $\cM$ is \defterm{simple along $\Lambda$} if it is regular along $\Lambda$ and $\mu_\Lambda^{dR}(\cM)$ is a line bundle.
\end{defn}

The notion of regular and simple $\hE^{\sqrt{v}}_X$-modules is defined similarly. If $\cM$ is an $\hE^{\sqrt{v}}_X$-module simple along $\Lambda$, then for any choice of the square root $K^{1/2}_{\tilde{\Lambda}}$ we have a rank $1$ local system
\[\cHom_{\cD^{\sqrt{v}}_{\tilde{\Lambda}}}\left(\mu_\Lambda^{dR}(\cM), K^{1/2}_{\tilde{\Lambda}}\right)\]
which is a subsheaf of $K^{1/2}_{\tilde{\Lambda}}$. For any simple generator $u\in\cM$ it allows us to define the \defterm{symbol}
\[\sigma_\Lambda(u)\in K^{1/2}_{\tilde{\Lambda}} / \C^\times_{\tilde{\Lambda}},\]
i.e. a section of $K^{1/2}_{\tilde{\Lambda}}$ defined uniquely up to a locally constant scale. It satisfies the following properties:
\begin{itemize}
    \item If $u$ has order $\lambda\in\C$, then the Euler vector field $\theta'$ acts on $\sigma_\Lambda(u)$ by $\lambda$.
    \item If $P\in\hE^{\sqrt{v}}_X(m)$ is a section such that $\sigma_m(P)|_\Lambda$ never vanishes, then $Pu$ is also a simple generator and
    \[\sigma_\Lambda(Pu) = \sigma_m(P)\sigma_\Lambda(u)\]
    and $\ord(Pu) = \ord(u) + m$.
    \item The \defterm{order} of simple $\hE^{\sqrt{v}}_X$-module is $\ord(\cM)\in\C/\Z$ defined to be the equivalence class of $\ord(u)$ for a simple generator $u$.
\end{itemize}

We have the following important classification result of microdifferential modules regular along a Legendrian. Let
\[\Mod_{\locsys}\left(\cD^{\sqrt{v}}_{\tilde{\Lambda}}\right)\subset \Mod_{\rh}\left(\cD^{\sqrt{v}}_{\tilde{\Lambda}}\right)\]
be the full substack of twisted $D$-modules on $\tilde{\Lambda}$ which are local systems, i.e. whose underlying $\cO$-module is a finite rank vector bundle. The following statement combines \cite[Theorem 10.3.1]{KashiwaraIntroduction} and \cite[Theorem 7.9]{KashiwaraKawaiMicrolocal}.

\begin{thm}\label{thm:regularclassificationlocal}
Let $\Lambda\subset \P^* X$ be a Legendrian submanifold. Microlocalization defines an equivalence of stacks
\[\mu_\Lambda^{dR}\colon \Mod_{\Lambda, \rh}(\hE^{\sqrt{v}}_X)\xrightarrow{\sim} \gamma_* \Mod_{\locsys}\left(\cD^{\sqrt{v}}_{\tilde{\Lambda}}\right)\]
between the stack of $\hE^{\sqrt{v}}_X$-modules regular along $\Lambda$ and twisted local systems on $\tilde{\Lambda}$. Moreover, if $\cM$ is a regular holonomic $\hE^{\sqrt{v}}_X$-module and $\cL$ a $\hE^{\sqrt{v}}_X$-module regular along $\Lambda$, then microlocalization defines an isomorphism
\[\RcHom_{\hE^{\sqrt{v}}_X}(\cM, \cL)\xrightarrow{\sim} \gamma_*\RcHom_{\cD^{\sqrt{v}}_{\tilde{\Lambda}}}(\mu_\Lambda^{dR}(\cM), \mu_\Lambda^{dR}(\cL)).\]
\end{thm}

\begin{remark}
Locally on $\Lambda$ twisted $D$-modules on $\tilde{\Lambda}$ whose underlying $\cO$-module is a line bundle are classified by an element $\C/\Z$ given by the eigenvalues of the Euler operator. So, simple $\hE^{\sqrt{v}}_X$-modules along $\Lambda$ are locally classified by the order.
\end{remark}

\subsection{Quantized contact transformations}\label{sect:QCT}

To define the stack of microdifferential modules for a general contact manifold $Y$, Kashiwara in \cite{KashiwaraContact} uses the theory of quantized contact transformations developed in \cite{SatoKawaiKashiwara} that we now recall.

\begin{defn}
Suppose $V_X\subset \P^* X$ and $V_Y\subset \P^* Y$ are open subsets together with a contactomorphism $\psi\colon V_X\xrightarrow{\sim} V_Y$. A \defterm{quantized contact transformation over $\psi$} is an isomorphism
\[\Psi\colon \hE_Y|_{V_Y}\xrightarrow{\sim}\psi_*\hE_X|_{V_X}\]
of sheaves of filtered $\C$-algebras which coincides with $\psi$ after passing to the associated graded algebras.
\end{defn}

By \cite[Theorem 8.11]{KashiwaraDmodules} quantized contact transformations exist locally. We will also need a half-twisted version as follows.

\begin{defn}
Suppose $V_X\subset \P^* X$ and $V_Y\subset \P^* Y$ are open subsets together with a contactomorphism $\psi\colon V_X\xrightarrow{\sim} V_Y$. A \defterm{$\ast$-preserving quantized contact transformation over $\psi$} is an isomorphism
\[\Psi\colon \hE^{\sqrt{v}}_Y|_{V_Y}\xrightarrow{\sim}\psi_*\hE^{\sqrt{v}}_X|_{V_X}\]
of sheaves of filtered $\C$-algebras which intertwines the antiinvolutions $\ast$ and which coincides with $\psi$ after passing to the associated graded algebras.
\end{defn}

The construction of the microlocalization functor $\mu^{dR}_\Lambda$ involves the filtration of microdifferential operators by order and the involution $\ast$. Therefore, microlocalization (and, a posteriori, the notion of symbols and orders of simple $\hE^{\sqrt{v}}_X$-modules) commutes with $\ast$-preserving quantized contact transformations.

\begin{prop}
Consider a $\ast$-preserving quantized contact transformation $\Psi$ over $\psi\colon V_X\xrightarrow{\sim} V_Y$ as in the above definition. Suppose $\Lambda_X\subset V_X$ and $\Lambda_Y\subset V_Y$ are two Legendrians such that $\Lambda_Y=\psi(\Lambda_X)$. Then a $\ast$-preserving quantized contact transformation $\Psi\colon \hE^{\sqrt{v}}_Y|_{V_Y} \xrightarrow{\sim} \psi_*\hE^{\sqrt{v}}_X|_{V_X}$ restricts to an isomorphism
\[\Psi\colon \hE^{\sqrt{v}}_{\Lambda_Y/Y}|_{V_Y}\xrightarrow{\sim}\psi_*\hE^{\sqrt{v}}_{\Lambda_X/X}|_{V_X}.\]
Moreover, the diagram
\[
\xymatrix{
\gr_\Lambda \hE^{\sqrt{v}}_{\Lambda_Y/Y}\ar[r] \ar^{\Psi}[d] & \gamma_* \cD_{\tilde{\Lambda_Y}}^{\sqrt{v}} \ar^{\psi}[d] \\
\gr_\Lambda \hE^{\sqrt{v}}_{\Lambda_X/X}\ar[r] & \gamma_* \cD_{\tilde{\Lambda_X}}^{\sqrt{v}}
}
\]
is commutative. In particular, there is a commutative diagram
\[
\xymatrix{
\Mod_{\rh}(\hE^{\sqrt{v}}_Y)|_{V_Y}\ar^-{\mu_{\Lambda_Y}^{dR}}[r] \ar^{\Psi}[d] & \gamma_*\Mod_{\rh, \mon}\left(\cD^{\sqrt{v}}_{\tilde{\Lambda_Y}}\right) \ar^{\psi}[d] \\
\psi_* \Mod_{\rh}(\hE^{\sqrt{v}}_X)|_{V_X}\ar^-{\mu_{\Lambda_X}^{dR}}[r] & \gamma_*\Mod_{\rh, \mon}\left(\cD^{\sqrt{v}}_{\tilde{\Lambda_X}}\right)
}
\]
\label{prop:QCTmicrolocalization}
\end{prop}

It will be convenient to rephrase quantized contact transformations as follows. Let $U_X\subset \T^* X$ and $U_Y\subset \T^* Y$ be the preimages of $V_X,V_Y$ and $\tilde{\psi}\colon U_X\xrightarrow{\sim} U_Y$ the homogeneous symplectic transformation lifting $\psi$. Let
\[p_1\colon \T^*(X\times Y)\rightarrow \T^* X\]
be the projection on the first factor and
\[p_2^a\colon \T^*(X\times Y)\rightarrow \T^* Y\]
be the projection on the second factor post-composed with the multiplication by $-1$ along the cotangent fibers.

We have the symplectic volume form $\vol_{\T^* Y}$ on $\T^* Y$. So, there is a canonical choice of the square root line bundle $K^{1/2}_{\T^* Y}$ with section $\sqrt{\vol_{\T^* Y}}$.

\begin{prop}
The data of a quantized contact transformation
\[\Psi\colon \hE^{\sqrt{v}}_Y|_{V_Y}\xrightarrow{\sim}\psi_*\hE^{\sqrt{v}}_X|_{V_X}\]
over $\psi$ is equivalent to the data of a $\hE^{\sqrt{v}}_{X\times Y}$-module $\cM$ defined on an open subset $U\subset \T^*(X\times Y)$ which is simple along a Legendrian $\Lambda$ with a simple generator $K$ satisfying the following properties:
\begin{itemize}
    \item $p_1(U)\subset U_X$ and $p_1\colon \tilde{\Lambda} \rightarrow U_X$ is an isomorphism.
    \item $p_2^a(U)\subset U_Y$ and $p^a_2\colon \tilde{\Lambda}\rightarrow U_Y$ is an isomorphism.
\end{itemize}
The relationship between the pair $(\cM, K)$ and the QCT $\Psi$ is given by
\[\Psi(Q) K = Q^* K.\]
If $\Psi$ is a $\ast$-preserving quantized contact transformation, then
\[\sigma_\Lambda(K) = \sqrt{\vol_{\T^* Y}},\qquad \ord(K) = (\dim Y)/2.\]
\label{prop:QCTbimodule}
\end{prop}
\begin{proof}
The correspondence between quantized contact transformations and pairs $(\cM, K)$ is established in \cite[Chapter 8.2]{KashiwaraDmodules}.

Now assume $\Psi$ is a $\ast$-preserving quantized contact transformation. Consider $Q\in\hE^{\sqrt{v}}_Y(m+1)$. By definition we have
\[(\Psi(Q) - Q^*)K = 0.\]
To show that $\sigma_\Lambda(K) = \sqrt{\vol_{\T^* Y}}$, using the notation of \cite[Chapter 8.6]{KashiwaraDmodules} we have to prove that
\[(H_{\sigma_{m+1}(\Psi(Q))} + \sigma_m(\Psi(Q)) - H_{\sigma_{m+1}(Q^*)} - \sigma_m(Q^*)) \sqrt{\vol_{\T^* Y}} = 0.\]
We have $\sigma_m(Q^*) = (-1)^m \sigma_m(Q) = p_2^a \sigma_m(Q)$ and $\sigma_m(\Psi(Q)) = \psi(\sigma_m(Q))$. Moreover, the function
\[\psi(\sigma_m(Q)) - (-1)^m\sigma_m(Q)\]
vanishes along the Lagrangian $\tilde{\Lambda}$.

The vector field $H_{\sigma_{m+1}(\Psi(Q))} - H_{\sigma_{m+1}(Q^*)} = H_{\psi(\sigma_{m+1}(Q)} + (-1)^m H_{\sigma_{m+1}(Q)}$ is tangent to the Lagrangian $\tilde{\Lambda}$. Identifying $\Lambda$ with $U_X$ using $p_1$ this vector field is identified with $H_{\psi(\sigma_{m+1}(Q))}$. Under the same identification the section $\sqrt{\vol_{\T^* Y}}$ of $K^{1/2}_{\tilde{\Lambda}}$ goes to the section $\sqrt{\vol_{\T^* X}}$ of $K^{1/2}_{U_X}$. By Liouville's theorem $H_{\sigma_{m+1}(\Psi(Q))} \sqrt{\vol_{\T^* X}} = 0$ which implies the result.
\end{proof}

\subsection{Globalization}

Recall that algebras form a 2-category as follows:
\begin{itemize}
    \item Its objects are algebras.
    \item 1-morphisms from $A$ to $B$ are given by homomorphisms $f\colon A\rightarrow B$.
    \item 2-morphisms from $f\colon A\rightarrow B$ to $g\colon A\rightarrow B$ are given by invertible elements $b\in B$ such that $g(a) = b f(a) b^{-1}$.
\end{itemize}

By an \defterm{algebroid} over a complex manifold we will mean a sheaf (i.e. a stack) of algebras viewed as objects of the above 2-category. For an algebroid it makes sense to consider the sheaf (i.e. a stack) of modules over it.

Let $Y$ be a contact manifold with a symplectization $\gamma\colon \tilde{Y}\rightarrow Y$. The papers \cite{KashiwaraContact,PoleselloSchapira} have defined the algebroid $\sfhE_Y$ over $Y$ uniquely determined by the following properties:
\begin{itemize}
    \item If $U\subset Y$ is an open subset together with a contact embedding $U\hookrightarrow \P^* X$ into a projectivized cotangent bundle, then
    \[\sfhE_Y|_U\cong \hE^{\sqrt{v}}_X|_U.\]
    \item Given an intersection of two such charts, the isomorphism of the algebras of half-twisted microdifferential operators is given by a $\ast$-preserving quantized contact transformation.
\end{itemize}

Consider a Legendrian submanifold $\Lambda\subset Y$. As the notion of regularity of $\hE^{\sqrt{v}}_X$-modules is local and invariant under quantized contact transformations, it extends to $\sfhE_Y$-modules. Moreover, using \cref{prop:QCTmicrolocalization} we obtain a global version of the microlocalization functor
\[\mu^{dR}_\Lambda\colon \Mod_{\rh}(\sfhE_Y)\longrightarrow \gamma_*\Mod_{\rh}\left(\cD^{\sqrt{v}}_{\tilde{\Lambda}}\right).\]

Using this functor one obtains the following classification result generalizing \cref{thm:regularclassificationlocal} (see \cite[Corollary 6.4]{DAgnoloSchapira}).

\begin{thm}
Let $\Lambda\subset Y$ be a Legendrian submanifold. Microlocalization defines an equivalence of stacks
\[\mu_\Lambda^{dR}\colon \Mod_{\Lambda, \rh}(\sfhE_Y)\xrightarrow{\sim} 
\gamma_*\Mod_{\locsys}\left(\cD^{\sqrt{v}}_{\tilde{\Lambda}}\right)\]
between the stack of $\sfhE_Y$-modules regular along $\Lambda$ and the stack of twisted local systems on $\tilde{\Lambda}$. Moreover, if $\cM$ is a regular holonomic $\sfhE_Y$-module and $\cL$ is an $\sfhE_Y$-module regular along $\Lambda$, then microlocalization defines an isomorphism
\[\RcHom_{\sfhE_Y}(\cM, \cL)\xrightarrow{\sim} \gamma_* \RcHom_{\cD^{\sqrt{v}}_{\tilde{\Lambda}}}(\mu_\Lambda^{dR}(\cM), \mu_\Lambda^{dR}(\cL)).\]
\label{thm:simpleclassification}
\end{thm}

The antiinvolution $\ast$ on $\hE^{\sqrt{v}}_X$ globalizes to an isomorphism of algebroids $\sfhE_Y \cong \sfhE_Y^{\op}$ for any contact manifold $Y$. Thus, any left $\sfhE_Y$-module $\sfM$ can be considered as a right $\sfhE_Y$-module and vice versa.

\section{DQ modules}
\label{sect:DQmodules}

In this section we define one of the main objects of the paper: the stack of deformation quantization modules on a holomorphic symplectic manifold.

\subsection{WKB-differential operators}

Let $X$ be a complex manifold. Consider the sheaf of algebras $\hE_{X\times\C}$ on $\rJ^1 X\subset \P^*(X\times\C)$. The natural $\C$-action on $\rJ^1 X$ lifts to a Hamiltonian $\C$-action on $\hE_{X\times \C}$ with the quantum moment map given by $\partial_t$. We consider the sheaf
\[\hW_X\subset \rho_*(\hE_{X\times \C})\]
of $\C(\!(\hbar)\!)$-algebras of \defterm{formal WKB-differential operators} introduced in \cite{PoleselloSchapira}, where $\rho\colon \rJ^1 X\rightarrow \T^* X$ and $\hbar = \partial_t^{-1}$. Explicitly, $\hW_X$ is defined as $\rho_\ast \hW_{X,\widehat{t}}$, where $\hW_{X,\widehat{t}}$ is the subsheaf of $\hE_{X\times \C}$ consisting of operators $P$ such that $[P,\partial_t]=0$.

\begin{remark}
One may interpret $\hW_X$ with the quantum Hamiltonian reduction of $\hE_{X\times\C}$ at the moment map value $\partial_t = \hbar^{-1}$ for $\hbar\in\C(\!(\hbar)\!)$ a formal variable.
\end{remark}

\begin{example}
There is an isomorphism of sheaves of $\C(\!(\hbar)\!)$-algebras
\[\hW_X|_X\cong \cD_X(\!(\hbar)\!)\]
on $X$.
\end{example}

\begin{remark}
The filtration of $\hE_{X\times\C}$ by order induces a filtration on $\hW_X$. In particular, $\hW_X(0) \subseteq \hW_X$ is a $\Cbb$-lattice. We have an isomorphism
\[\hW_X(0) / \hbar\cong \cO_{\T^* X}\]
of sheaves of algebras, so $\hW_X(0)$ defines a deformation quantization of $\T^* X$.
\end{remark}

\begin{lm}
The $\hW_X$-module $\rho_*(\hE_{X\times\C})$ is flat.
\label{lm:Eflat}
\end{lm}
\begin{proof}
To prove the claim it is enough to show that the $\hW_X(0)$-module $\rho_*(\hE_{X\times\C}(0))$ is flat as $\rho_*(\hE_{X\times\C}) = \rho_*(\hE_{X\times\C}(0))[\hbar^{-1}]$.

We have that $\rho_*(\hE_{X\times\C}(0))$ is $\hbar$-complete and has no $\hbar$-torsion. Therefore, by \cite[Corollary 1.5.7]{KashiwaraSchapira} it is cohomologically $\hbar$-complete. By \cite[Theorem 1.6.6]{KashiwaraSchapira} in this case it is enough to prove flatness modulo $\hbar$, i.e. it is enough to prove that $\rho_*\cO_{\rJ^1 X}$ is flat as a $\cO_{\T^*X}$-module. Since the map $\rJ^1 X\rightarrow \T^* X$ is open, it is flat by \cite[Theorem V.2.13]{BanicaStanasila}.
\end{proof}

The half-twisted algebra
\[\hW^{\sqrt{v}}_X\subset \rho_*(\hE^{\sqrt{v}}_{X\times\C})\]
is defined similarly, where we note that $\hbar^* = -\hbar$. In this case the antiinvolution $\ast$ on $\hE^{\sqrt{v}}_{X\times\C}$ restricts to one on $\hW^{\sqrt{v}}_X$.

Let us now explain how to glue these sheaves of algebras. Suppose $U_X\subset \T^* X$ and $U_Y\subset \T^*Y$ are open subsets together with a symplectomorphism $\phi\colon U_X\xrightarrow{\sim} U_Y$. Suppose the symplectomorphism $\phi$ lifts to an isomorphism of contactifications $\psi\colon V_X=\rho^{-1}(U_X)\xrightarrow{\sim} V_Y=\rho^{-1}(U_Y)$, where $\rho\colon \rJ^1 X\rightarrow \T^* X$ and $\rho\colon \rJ^1 Y\rightarrow \T^* Y$ are the standard contactifications. Recall that in this case $\psi$ is compatible with the Hamiltonian $\C$-actions. The following is \cite[Lemma 8.5]{PoleselloSchapira}.

\begin{prop}
Consider an isomorphism of contactifications $\psi\colon V_X=\rho^{-1}(U_X)\xrightarrow{\sim} V_Y=\rho^{-1}(U_Y)$ as above. Then there is a $\ast$-preserving quantized contact transformation
\[\Psi\colon \hE^{\sqrt{v}}_{Y\times\C}|_{V_Y}\xrightarrow{\sim}\psi_*\hE^{\sqrt{v}}_{X\times\C}|_{V_X}\]
compatible with the Hamiltonian $\C$-actions, i.e. $\Psi(\partial_t) = \partial_t$ and $\Psi$ is $\C$-equivariant.
\label{prop:QST}
\end{prop}

Consider the setting of \cref{prop:QST}. Then we may restrict the $\ast$-preserving quantized contact transformation $\Psi$ to a \defterm{$\ast$-preserving quantized symplectic transformation}
\[\Phi\colon \hW^{\sqrt{v}}_Y|_{U_Y}\xrightarrow{\sim}\phi_*\hW^{\sqrt{v}}_X|_{U_X}.\]

Let $\cM$ be the $\hE^{\sqrt{v}}_{X\times \C\times Y\times \C}$-module simple along a Legendrian $\Lambda$ with a generator $K$ corresponding to the quantized contact transformation $\Psi$ via \cref{prop:QCTbimodule}. We have
\[\sigma_\Lambda(K) = \sqrt{\tau^{\dim Y}\omega_{\T^* Y} \wedge d\tau\wedge dt}.\]
Consider the Hamiltonian $\C$-action on $\hE^{\sqrt{v}}_{X\times\C\times Y\times \C}$ given by $(t_1, t_2)\mapsto (t_1 + \delta t, t_2 + \delta t)$ with Hamiltonian $\partial_{t_1} + \partial_{t_2}$. We may identify the Hamiltonian reduction as
\[(\hE^{\sqrt{v}}_{X\times \C\times Y\times \C})\ham \C\cong \hE^{\sqrt{v}}_{X\times Y\times \C},\]
where $t$ corresponds to $(t_1-t_2)$ and $\partial_t$ corresponds to $\partial_{t_1}=-\partial_{t_2}$. By \cref{prop:QST} the module $\cM$ is strongly $\C$-equivariant, so after Hamiltonian reduction we obtain a $\hE^{\sqrt{v}}_{X\times Y\times \C}$-module
\[\cM^\partial = \{m\in \cM\mid (\partial_{t_1}+\partial_{t_2})m = 0\}.\]

\begin{prop}
The $\hE^{\sqrt{v}}_{X\times Y\times \C}$-module $\cM^\partial$ is simple along the Legendrian $\Lambda'=\Gamma_\psi\subset (V_X\times \overline{V_Y})/\C\hookrightarrow \rJ^1(X\times Y)$. Moreover,
\[\sigma_{\Lambda'}(K) = \sqrt{\tau^{\dim Y}\omega_{\T^* Y}\wedge d\tau},\qquad \ord(K) = (\dim Y+1)/2.\]
\label{prop:QSTbimodule}
\end{prop}
\begin{proof}
We may identify
\[(\cM, K)\cong (\hE^{\sqrt{v}}_{X\times \C}|_{V_X}, 1)\]
using the projection $p_1\colon \T^*(X\times\C\times Y\times \C)\rightarrow \T^*(X\times \C)$. In a similar way, we may identify
\[(\cM^\partial, K)\cong (\hW^{\sqrt{v}}_X|_{U_X}, 1).\]
To simplify the notation, we will omit the restrictions to open subsets $V_X$, $U_X$ and so on. Consider the $\hE^{\sqrt{v}}_{X\times Y\times \C}(0)$-submodule $\cM^\partial(n)\subset \cM^\partial$ given by
\[\cM^\partial(n) = \hE^{\sqrt{v}}_{X\times Y\times \C}(n) K.\]

Note that we have $\hW^{\sqrt{v}}_{X\times Y}(n) K=\hW^{\sqrt{v}}_X(n)$. But $t\in\hE^{\sqrt{v}}_{X\times Y\times\C}(0)$ acts on $\cM^\partial\cong \hW^{\sqrt{v}}_X$ by $\hbar^2\partial_\hbar$, so $\cM^\partial(n) = \hW^{\sqrt{v}}_{X\times Y}(n) K = \hW^{\sqrt{v}}_X(n)$.

Similarly, one defines the $\hE^{\sqrt{v}}_{X\times \C\times Y\times \C}(0)$-submodules $\cM(n)\subset \cM$ which can be identified with $\hE^{\sqrt{v}}_{X\times \C}(n)$. Note that such an identification implies an isomorphism $\cM(n)^\C\cong \cM^\partial(n)$. To show that $\cM^\partial$ is simple along $\Lambda'$, we have to show that the $\hE^{\sqrt{v}}_{X\times Y\times \C}(0)/\hE^{\sqrt{v}}_{X\times Y\times \C}(-1)$-module structure on $\cM^\partial(0)/\cM^\partial(-1)$ extends to a $\gr^0_{\Lambda'}\hE^{\sqrt{v}}_{X\times Y\times\C}$-module structure and the corresponding $\cO_{\Lambda'}$-module is locally free of rank 1.

Since the homogeneous symplectic transformation $\tilde{\psi}$ preserves $\tau$, we have $(\tau_1 + \tau_2)|_{\tilde{\Lambda}'} = 0$. In particular, $\partial_{t_1} + \partial_{t_2}\in\hE^{\sqrt{v}}_{\Lambda/X\times\C\times Y\times \C}$. So, the Hamiltonian $\C$-action on $\hE^{\sqrt{v}}_{X\times\C\times Y\times \C}$ restricts to one on $\hE^{\sqrt{v}}_{\Lambda/X\times\C\times Y\times \C}$. Passing to the associated graded, we get the natural Hamiltonian $\C$-action on $\cD^{\sqrt{v}}_{[\tilde{\Lambda}]}$ induced by the $\C$-action on $\Lambda$. As $\cD^{\sqrt{v}}_{[\tilde{\Lambda}]}\ham\C\cong \cD^{\sqrt{v}}_{[\tilde{\Lambda'}]}$, we get an isomorphism
\[\hE^{\sqrt{v}}_{\Lambda/X\times\C\times Y\times\C}\ham\C\cong \hE^{\sqrt{v}}_{\Lambda'/X\times Y\times\C},\]
so that the Hamiltonian reduction of the homomorphism
\[L\colon \hE^{\sqrt{v}}_{\Lambda/X\times\C\times Y\times\C}\longrightarrow \cD^{\sqrt{v}}_{[\tilde{\Lambda}]}\]
is the homomorphism
\[L\colon \hE^{\sqrt{v}}_{\Lambda'/X\times Y\times\C}\longrightarrow \cD^{\sqrt{v}}_{[\tilde{\Lambda'}]}.\]

We have a pair of strongly $\C$-equivariant $\cD^{\sqrt{v}}_{\tilde{\Lambda}}(0)$-modules $\cM(0)/\cM(-1)$ and $K^{1/2}_{\tilde{\Lambda}}$ whose Hamiltonian reduction gives $\cD^{\sqrt{v}}_{\tilde{\Lambda'}}(0)$-modules $(\cM(0)/\cM(-1))^\C$ and $(K^{1/2}_{\tilde{\Lambda}})^\C\cong K^{1/2}_{\tilde{\Lambda'}}$.

We may identify
\[\cM(0)/\cM(-1)\cong\hE^{\sqrt{v}}_{X\times\C}(0) / \hE^{\sqrt{v}}_{X\times\C}(-1)\cong \cO_{\rJ^1 X}\]
and
\[\cM^\partial(0)/\cM^\partial(-1)\cong \hW^{\sqrt{v}}_{X}(0) / \hW^{\sqrt{v}}_X(-1)\cong \cO_{\T^* X}.\]
In particular, the natural map $\cM^\partial(0)/\cM^\partial(-1)\rightarrow (\cM(0)/\cM(-1))^\C$ is an isomorphism. This implies that $\cM$ is simple along the $\Lambda'$ and
\[\sigma_{\Lambda'}(K) = \sigma_{\Lambda}(K) / \sqrt{dt}\]
which, in conjunction with the computation of the symbol $\sigma_\Lambda(K)$ in \cref{prop:QCTbimodule}, implies the result.
\end{proof}

\subsection{Globalization}

Let $S$ be a symplectic manifold. The authors of \cite{PoleselloSchapira} have constructed a $\C(\!(\hbar)\!)$-algebroid $\sfhW_S$ on $S$ with the following properties:
\begin{itemize}
    \item If $U\subset S$ is an open subset together with a symplectic embedding $U\hookrightarrow \T^* X$ into a cotangent bundle, then
    \[\sfhW_S|_U\cong \hW^{\sqrt{v}}_X|_U.\]
    \item Given an intersection of two such charts, the isomorphism of the algebras of half-twisted WKB-differential operators is given by a $\ast$-preserving quantized symplectic transformation (see the paragraph following \cref{prop:QST}).
    \item If $Y\rightarrow S$ is a contactification, then there is a natural embedding
    \[\sfhW_S\subset \rho_* \sfhE_Y.\]
\end{itemize}

The antiinvolution $\ast$ on $\hW^{\sqrt{v}}_X$ globalizes to a $\C$-linear isomorphism of algebroids $\sfhW_S\cong \sfhW^{\op}_S$ which sends $\hbar\mapsto -\hbar$. As for $\sfhE_Y$-modules, we freely pass between left and right $\sfhW_S$-modules using this antiinvolution.

\begin{defn}
A $\sfhW_S$-module is \defterm{holonomic} if its support is a Lagrangian subvariety of $S$.
\end{defn}

For a Lagrangian subvariety $L\subset S$ we denote by $\sfhW_{L/S}\subset \sfhW_S|_L$ the $\C[\![\hbar]\!]$-subalgebroid generated by
\[I_L = \{P\in\sfhW_S(1)|_L\mid \sigma_1(P)|_L = 0\}.\]

\begin{defn}
A holonomic $\sfhW_S$-module $\cM$ is \defterm{regular holonomic} if locally there is a coherent $\sfhW_S(0)$-module $\cM_0\subset \cM$ and a Lagrangian subvariety $L\subset S$ containing $\supp\cM$ such that $I_L\cM_0\subset \cM_0$ and $\sfhW_S\cM_0 = \cM$.
\end{defn}

\begin{example}
Let $L\subset S$ be a Lagrangian subvariety and $\Lambda\subset Y$ be its contactification with $\rho\colon Y\rightarrow S$ the projection map. If $\cM$ is a $\sfhE_Y$-module supported on $\Lambda$, then $\cM$, regarded as a $\sfhW_S$-module via the embedding $\sfhW_S\subset \rho_* \sfhE_Y$, is supported on $L$ and is, therefore, holonomic. This gives a forgetful functor
\[\rho_*\Mod_{hol}(\sfhE_Y)\longrightarrow \Mod_{hol}(\sfhW_S)\]
from holonomic $\sfhE_Y$-modules to holonomic $\sfhW_S$-modules. Since $\sfhW_{L/S}\subset \rho_*\sfhE_{\Lambda/Y}$, this forgetful functor restricts to a functor
\[\rho_*\Mod_{\rh}(\sfhE_Y)\longrightarrow \Mod_{\rh}(\sfhW_S).\]
\label{ex:forgetEW}
\end{example}


\subsection{Tensor product of DQ modules}
As in the previous section, we consider a holomorphic symplectic manifold $S$ and $\rho:Y\to S$ a contactification. Let $L$ and $M$ be lagrangian submanifolds of $S$ and $X=L\cap M$. 

Consider two regular holonomic $\rho_\ast\sfhE_Y$-modules $\cL$ and $\cM$, supported on $L$ and $M$ respectively. In this section, we will study the complex of $\C(\!(\hbar)\!)$-linear sheaves $\cL\otimes^\bL_{\sfhW_S}\cM$ on $X$. Our first goal is to endow it with extra structure.

We also make the following assumptions (which always hold locally on $S$): $\sfhE_Y$ and $\sfhW_S$ are represented by sheaves of algebras (as opposed to just algebroids) and there is an element $\hbar\in \sfhE_Y(-1)$ with principal symbol $\tau^{-1}$ satisfying $\hbar^*=-\hbar$. Then $\sfhW_S$ is the subalgebra of $\rho_*\sfhE_Y$ of elements commuting with $\hbar$. 

Consider
\[\sfB_Y = (\rho_*\sfhE_Y)\otimes_{\sfhW_S} (\rho_*\sfhE_Y),\]
where the right action of $\sfhW_S$ on the first copy of $\rho_*\sfhE_Y$ is via $P\in\sfhW_S,Q\in\rho_*\sfhE_Y\mapsto P^* Q$. It carries a natural bimodule structure as follows: $\rho_*\sfhE_Y\otimes \rho_*\sfhE_Y$ acts in an obvious way on the right and $\hbar$ acts on $Q_1\otimes Q_2\in \sfB_Y$ as $Q_1\otimes \hbar Q_2 = -\hbar Q_1\otimes Q_2$.

\begin{prop}
$\sfB_Y$ has a canonical structure of a $(\C(\!(\hbar)\!)\langle D\rangle, \rho_*\sfhE_Y\otimes \rho_*\sfhE_Y)$-bimodule, where $D\hbar = \hbar (D+1)$.
\label{prop:differentialenhancement}
\end{prop}
\begin{proof}
Let $\partial_t = \hbar^{-1}\in\sfhE_Y(1)$. Locally we may choose an element $t\in \sfhE_Y(0)$ satisfying
\[[\partial_t, t] = 1.\]
If we let $D = \partial_t t$, we have
\[[D, \hbar] = \hbar,\qquad [D^*, \hbar] = -\hbar.\]
This allows us to enhance $\sfB_Y$ to a $(\C(\!(\hbar)\!)\langle D\rangle, \rho_*\sfhE_Y\otimes \rho_*\sfhE_Y)$-bimodule, where $D=\partial_t t$ acts on $Q_1\otimes Q_2\in\sfB_Y$ as $-D^*Q_1 \otimes Q_2 + Q_1\otimes D Q_2$.

Let us show that this enhancement is independent of the choice of $t$. If $t'$ is another element of $\sfhE_Y(0)$ satisfying $[\partial_t, t'] = 1$, then $[\partial_t, t - t'] = 0$, i.e. $[\hbar, t-t'] = 0$. Thus, there is an element $R\in\sfhW_S(0)$ such that
\[t' = t + R.\]
Then
\[\partial_t t' = \partial_t t + \partial_t R.\]
We then have
\begin{align*}
-&(\partial_t t')^* Q_1\otimes Q_2 + Q_1\otimes \partial_t t' Q_2 \\
= -&(\partial_t t)^*Q_1\otimes Q_2 + Q_1\otimes \partial_t t Q_2 - (\partial_t R)^* Q_1\otimes Q_2 + Q_1\otimes \partial_t R Q_2.    
\end{align*}
But $\partial_t R\in\sfhW_S$ and hence the last two terms cancel.
\end{proof}

Using the previous proposition we obtain the following result.

\begin{thm}
For any regular holonomic $\rho_*\sfhE_Y$-modules $\cL,\cM$ the sheaf $\cL\otimes^\bL_{\sfhW_S} \cM$ has a canonical structure of a differential perverse sheaf.
\label{thm:tensorproductdifferentialperverse}
\end{thm}
\begin{proof}
Consider the duality functor $\D(\cL) = \RcHom_{\sfhW_S}(\cL, \sfhW_S)[\dim S/2]$. By \cite[Lemma 7.2.2]{KashiwaraSchapira} $\D(\cL)$ is a holonomic $\sfhW_S$-module concentrated in degree 0 and we have
\[\RcHom_{\sfhW_S}(\D(\cL), \cM)[\dim S/2]\cong \cL\otimes^\bL_{\sfhW_S} \cM.\]
It is then shown in \cite[Theorem 7.2.3]{KashiwaraSchapira} that $\RcHom_{\sfhW_S}(\D(\cL), \cM)[\dim S/2]$ is a perverse sheaf of $\C(\!(\hbar)\!)$-vector spaces.

$\rho_*\sfhE_Y$ is flat as a $\sfhW_S$-module by \cref{lm:Eflat}. Therefore,
\[\cL\otimes^\bL_{\sfhW_S} \cM \cong \sfB_Y\otimes^\bL_{\rho_* \sfhE_Y\otimes \rho_* \sfhE_Y} (\cL\otimes \cM)\]
which has a differential structure by \cref{prop:differentialenhancement}.
\end{proof}

Our next goal is to show that the differential structure has regular singularities.

\begin{thm}
Consider two regular holonomic $\rho_*\sfhE_Y$-modules and assume $\cL$ is simple along a Lagrangian submanifold $L\subset S$. Then
\[\cL\otimes^\bL_{\sfhW_S} \cM\]
is a differential perverse sheaf with regular singularities.
\label{thm:WHomdifferentialperverse}
\end{thm}
\begin{proof}
Let $\rho\colon Y\rightarrow S$ be a contactification defined in a neighborhood of the union of supports of $\cL,\cM$ and $\Lambda\subset Y$ a Legendrian lift of $L\subset S$.

To define the differential structure we may choose $t\in\sfhE_Y(0)$ whose principal symbol $\sigma_0(t)\in\cO_Y$ vanishes along $\Lambda$. With this choice $D=\partial_t t=\hbar\partial_\hbar\in\sfhE_{\Lambda/Y}$. By assumption $\cL$ is an $\sfhE_Y$-module simple along $\Lambda$. Thus, there is an $\sfhE_{\Lambda/Y}$-lattice $\cL_0\subset \cL$ generating $\cL$ over $\sfhE_Y$ and $D$ acts on
\[\overline{\cL} = \cL_0 / \sfhE_{\Lambda/Y}(-1) \cL_0\]
via multiplication by $\lambda\in\C$.

Let $\cM_0\subset \cM|_\Lambda$ be a $\sfhE_{\Lambda/Y}$-lattice such that no two eigenvalues of $D$ on
\[\overline{\cM} = \cM_0 / \sfhE_{\Lambda/Y}(-1)\cM_0\]
differ by an integer.

By \cite[Lemma 7.1.3(v)]{KashiwaraSchapira}
\[\cL_0\otimes^\bL_{\sfhW_{L/S}} \cM_0\]
is a $\C[\![\hbar]\!]$-lattice in $\cL\otimes^\bL_{\sfhW_S} \cM$. Since $D\in\sfhE_{\Lambda/Y}$, this lattice is preserved by $D$. Moreover, no two eigenvalues of $D$ on the $\C$-linear perverse sheaf
\[(\cL_0\otimes^\bL_{\sfhW_{L/S}} \cM_0)/\hbar\cong \overline{\cL}_0\otimes_{\sfhW_{L/S}/\hbar} \overline{\cM}_0\]
differ by an integer.
\end{proof}

We will now give a description of the above differential perverse sheaf using microlocalization. Consider the setting of \cref{thm:WHomdifferentialperverse} and suppose $L\subset S$ is lifted to a Legendrian $\Lambda\subset Y$ in a contactification $\rho\colon Y\rightarrow S$. In this case we may identify $\tilde{\Lambda}\cong L\times \C^\times$. Then $\mu^{dR}_\Lambda \cL,\mu^{dR}_\Lambda\cM$ are twisted $D$-modules on $\tilde{\Lambda}$, monodromic with respect to the $\C^\times$-action. Therefore,
\[\mu^{dR}_\Lambda \cL\otimes^\bL_{\cD^{\sqrt{v}}_{\tilde{\Lambda}}} \mu^{dR}_\Lambda \cM\]
is a perverse sheaf on $\tilde{\Lambda}=L\times \C^\times$, which is locally constant along the $\C^\times$ orbits. Its restriction to $L\times\{1\}$ defines a perverse sheaf
\[(\mu^{dR}_\Lambda \cL\otimes^\bL_{\cD^{\sqrt{v}}_{\tilde{\Lambda}}} \mu^{dR}_\Lambda \cM)|_L[-1]\]
equipped with a monodromy operator $T$.

\begin{thm}
Consider the setting of \cref{thm:WHomdifferentialperverse}. There is an isomorphism of differential perverse sheaves
\[
\cL\otimes^\bL_{\sfhW_S} \cM \cong \RH^{-1} \left((\mu^{dR}_\Lambda \cL\otimes^\bL_{\cD^{\sqrt{v}}_{\tilde{\Lambda}}} \mu^{dR}_\Lambda \cM)|_L[-1]\right)
\]
on $L$.
\label{thm:Wtensormicrolocal}
\end{thm}
\begin{proof}
As in the proof of \cref{thm:WHomdifferentialperverse} we choose $\sfhE_{\Lambda/Y}$-submodules $\cL_0\subset \cL$ and $\cM_0\subset \cM|_\Lambda$ generating these modules over $\sfhE_Y|_\Lambda$ as well as an element $D=\partial_t t\in\sfhE_{\Lambda/Y}$.

We may identify
\[\gr^0_\Lambda \sfhE_{\Lambda/Y}\cong \cD^{\sqrt{v}}_L[\theta],\]
where $\theta$ is the Euler operator along $\Lambda$ which is the image of $D$ under the projection
\[\sfhE_{\Lambda/Y}\longrightarrow \gr^0_\Lambda \sfhE_{\Lambda/Y}.\]

Consider
\[\overline{\cL}_0 = \cL_0/\hbar,\qquad \overline{\cM}_0 = \cM_0/\hbar.\]
As $\hbar \sfhE_{\Lambda/Y}(0) = \sfhE_{\Lambda/Y}(-1)$, these are exactly the associated graded sheaves with respect to the microlocal $V$-filtration
\[\overline{\cL}_0\cong \gr^0_\Lambda \cL,\qquad \overline{\cM}_0\cong \gr^0_\Lambda \cM.\]

The complex of sheaves
\[\cL_0\otimes^\bL_{\sfhW_{L/S}} \cM_0\]
defines a $\C[\![\hbar]\!]$-lattice in $\cL\otimes_{\sfhW_S} \cM$ preserved by $D$. Therefore, by \cref{prop:Riemann Hilbert} the claim will follow once we establish an isomorphism
\[\gr^0_\Lambda \cL\otimes^\bL_{\cD^{\sqrt{v}}_L}\gr^0_\Lambda \cM\cong (\mu^{dR}_\Lambda\cL\otimes^\bL_{\cD^{\sqrt{v}}_{L\times\C^\times}} \mu^{dR}_\Lambda \cM)|_L[-1]\]
of perverse sheaves on $L$ equipped with monodromy automorphisms.

We have
\[\mu^{dR}_\Lambda \cL = \cD^{\sqrt{v}}_{L\times\C^\times}\otimes_{\gamma^{-1} \cD^{\sqrt{v}}_L[\theta]}\gamma^{-1}\gr^0_\Lambda \cL\]
and similarly for $\cM$ and so the claim follows from \cref{prop:monodromicRH}.
\end{proof}

We end this section with a lemma which will be needed later.
The local model of a contactification $\rho\colon Y\rightarrow S$ is given by $Y=\rJ^1 X\rightarrow S=\T^* X$ in which case we have
\[\sfB_Y = (\rho_* \hE^{\sqrt{v}}_{X\times\C}) \otimes_{\hW^{\sqrt{v}}_X} (\rho_* \hE^{\sqrt{v}}_{X\times\C}).\]

The involution $i\colon X\times\C\rightarrow X\times\C$ given by $(x, t)\mapsto (x, -t)$ induces an involution
\begin{equation}
(-)'\colon \hE^{\sqrt{v}}_{X\times\C}\xrightarrow{\sim} i_*\hE^{\sqrt{v}}_{X\times\C}
\label{eq:barinvolution}
\end{equation}
which sends $\partial_t\mapsto -\partial_t$ and which intertwines the antiinvolution $\ast$. Therefore, we get an involution
\[(-)'\colon \sfB_Y\longrightarrow \sfB_Y\]
given by $Q_1\otimes Q_2\mapsto Q_2'\otimes Q_1'$ of $(\C(\!(\hbar)\!)\langle D\rangle, \rho_*\sfhE_Y\otimes \rho_*\sfhE_Y)$-bimodules.

For a regular holonomic $\hE^{\sqrt{v}}_{X\times \C}$-module $\cM$ denote by $\cM'$ the same module with the $\hE^{\sqrt{v}}_{X\times \C}$-action twisted by the involution $(-)'$. Then we have the following observation.

\begin{lm}
The above involution of $\sfB_Y$ induces an isomorphism
\begin{equation}
\cM_1\otimes^\bL_{\hW^{\sqrt{v}}_X} \cM_2\cong \cM_2'\otimes^\bL_{\hW^{\sqrt{v}}_X} \cM_1'
\label{eq:barflip}
\end{equation}
of differential perverse sheaves.
\label{lm:barflip}
\end{lm}

\section{Vanishing cycles and DQ modules}

In this section we investigate a relationship between the following four objects: vanishing cycles, $\cE$-modules, twisted de Rham complex and $\cW$-modules.

\subsection{Vanishing cycles}

Let $X$ be a complex manifold equipped with a holomorphic function $f\colon X\rightarrow \C$. Recall the perverse sheaf $\phi_f\in\Perv(f^{-1}(0);\C)$ of vanishing cycles of the constant perverse sheaf $\C_X[\dim X]\in\Perv(X)$ \cite{Dimca}, which carries a monodromy automorphism $T$. In this section we give a microlocal description of this sheaf.

Consider $X\times\C$ with $t$ the coordinate on $\C$ and let
\[\Gamma_f = \{(x, t)\in X\times \C\mid f(x) - t = 0\}\subset X\times\C\]
be the graph of $f$.

Consider the $\cD_{X\times\C}$-module
\[\cB_f = i^{\cD}_*\cO_X,\]
where $i\colon X\hookrightarrow X\times \C$ is the embedding of the graph $\Gamma_f$ given by $x\mapsto (x, f(x))$ and $i^{\cD}_*$ is the $D$-module direct image. Its microlocalization $\hE_{X\times\C}\otimes_{\pi^{-1}\cD_{X\times\C}} \pi^{-1}\cB_f$ is supported on the open subset $\rJ^1 X\subset \P^*(X\times\C)$ and we denote by
\[\cC_f = \rho_*\left(\hE_{X\times\C}\otimes_{\pi^{-1}\cD_{X\times\C}} \pi^{-1}\cB_f|_{\rJ^1 X}\right)\]
the corresponding $\rho_*\hE_{X\times\C}$-module.

We may identify $\cB_f = \cO_X[\partial_t]$ as an $\cO_X[\partial_t]$-module with generator $\delta(t-f)$. The action of vector fields $\xi$ on $X$ and $t$ are given by
\begin{align*}
\xi(g \partial_t^i \delta(t-f)) &= (\xi g) \partial_t^i\delta(t-f) - (\xi f) g\partial_t^{i+1}\delta(t-f) \\
t(g\partial_t^i \delta(t-f)) &= fg\partial_t^i \delta(t-f) - ig \partial_t^{i-1} \delta(t-f).
\end{align*}

\begin{example}
Consider the Legendrian $\Lambda = j^1 f\subset \rJ^1 X$. Then $\cC_f$ is a simple $\rho_*\hE_{X\times\C}$-module along $\Lambda$. As the Euler operator we may take
\[\theta' = -(t-f) \partial_t - 1/2\in\hE_{\Lambda/\tilde{X}}.\]
Then $(\theta' - 1/2)\delta(t-f) = 0$, so $\delta(t-f)\in\cC_f$ is a simple generator of order $1/2$.
\label{ex:Cforder}
\end{example}

We may identify
\[\cC_f = \hE_{X\times\C} / \hE_{X\times\C}\left(t-f, \frac{\partial f}{\partial x_i}\partial_t + \partial_{x_i}\right).\]

Let
\[\cC^\lambda_f = \hE_{X\times\C} / \hE_{X\times\C}\left(t-f + (\lambda-1/2)\partial^{-1}_t, \frac{\partial f}{\partial x_i}\partial_t + \partial_{x_i}\right).\]
We denote its generator by $\partial_t^{\lambda-1/2}\delta(t-f)$ and the computation similar to \cref{ex:Cforder} shows that
\[\ord(\partial_t^{\lambda-1/2}\delta(t-f)) = \lambda.\]
We have an obvious isomorphism $(\cC^\lambda_f)'\cong \cC^\lambda_{-f}$ which preserves the generators, where $(-)'$ is the involution from \eqref{eq:barinvolution}.

Let us now introduce a few variants of these modules. We consider $\C$ with a canonical coordinate $t$. It allows us to trivialize $K_\C$ using the section $dt$ and hence we obtain a square root line bundle $K_\C^{1/2}$.
\begin{itemize}
    \item Let $X$ be a complex manifold. Let
    \[\cC^{\lambda, \rR}_f = \pi^{-1} K_{X\times\C}\otimes_{\pi^{-1} \cO_{X\times\C}} \cC^\lambda_f\]
    be the corresponding right $\rho_*\hE_{X\times \C}$-module.
    \item Let $X$ be a complex manifold together with a choice of the square root $K^{1/2}_X$. Let
    \[\cC^{\lambda, \sqrt{X}}_f = \pi^{-1} K^{1/2}_{X\times\C}\otimes_{\pi^{-1} \cO_{X\times\C}} \cC^\lambda_f\]
    be the corresponding $\rho_*\hE^{\sqrt{v}}_{X\times\C}$-module.
\end{itemize}

\begin{remark}
We use the notation $(-)^{\sqrt{v}}$ for $\cE, \cW, \cD$ and $(-)^{\sqrt{X}}$ for the modules to indicate that the half-twisted version of the algebras does not depend on the choice of a square root line bundle, while the half-twisted version of the modules does depend on this choice.
\end{remark}

\begin{example}
Choose a nonvanishing section $\sqrt{\omega}$ of $K_X^{1/2}$ and consider the simple generator
\[\sqrt{\omega} \partial_t^{\lambda-1/2} \delta(t-f)\in\cC^{\lambda, \sqrt{X}}_f.\]
Then
\[\sigma_{\Lambda}(\sqrt{\omega} \partial_t^{\lambda-1/2} \delta(t-f)) = \tau^{\lambda-1/2} \sqrt{d\tau\wedge \omega},\]
see \cite[Lemma 8.23]{KashiwaraDmodules}. Its order is
\[\ord(\sqrt{\omega} \partial_t^{\lambda-1/2} \delta(t-f)) = \lambda.\]
\end{example}


Consider the $V$-filtration along the hypersurface $\Gamma_0: t=0$ in $X\times\C$ and let
\[\phi_f^\cD = \gr^0_V \cB_f\in\Mod_{\rh,\EigenRange}(\cD_X[\theta])\]
be the associated graded $D$-module which carries an action of the Euler vector field $\theta=t\partial_t$ with eigenvalues contained in $\EigenRange$.

\begin{example}
$\phi_0^\cD$ is isomorphic to $\cO_X$ equipped with the endomorphism $\theta=0$.
\end{example}

The following statement is due to Kashiwara and Malgrange, see \cite[Th\'eor\`eme 6.2.6]{Malgrange} and \cite[Theorem 2]{KashiwaraVanishing}.

\begin{prop}
There is an isomorphism of perverse sheaves
\[\RcHom_{\cD_X}(\cO_X, \phi_f^\cD)[\dim X]\cong \phi_f.\]
Under this isomorphism the monodromy operator on $\phi_f$ is given by $T=\exp(-2\pi i\partial_tt)$.
\label{prop:KashiwaraMalgrange}
\end{prop}

We will also work with the sheaves of vanishing cycles microlocally. For this consider the microlocal $V$-filtration along the Legendrian $\Lambda=\P^*_{\Gamma_0}(X\times\C)\subset \P^*(X\times\C)$, where we may identify $\Lambda\cong X$. Using the trivialization of $K_{\Gamma_0|X\times\C}$ by $dt$ the associated graded $\gr^0_\Lambda \cC_f$ becomes a $\gr^0_\Lambda\hE_{\tilde{X}}=\cD_X[\theta]$-module and $\mu^{dR}_\Lambda \cC_f$ a $\cD_{\tilde{\Lambda}}$-module monodromic along $\C^\times$. Recall the correspondence between monodromic $D$-modules on $X\times \C^\times$ and $\cD_X[\theta]$-modules from \cref{sect:monodromicDmodules}.

\begin{prop}
For any $\lambda\in\C$ there is an isomorphism of $\cD_X[\theta]$-modules
\[\gr^0_\Lambda \cC^\lambda_f\cong (\phi_f^\cD, \theta + \lambda-1/2).\]
Moreover,
\[\mu^{dR}_\Lambda \cC_f = \cD_{\tilde{\Lambda}}\otimes_{\gamma^{-1} \cD_X[\theta]} \gamma^{-1} \gr^0_\Lambda\cC_f\]
is the corresponding monodromic $D$-module on $X\times \C^\times$.
\label{prop:microlocalmicrofunctions}
\end{prop}
\begin{proof}
The first claim for $\lambda=1/2$ follows from the comparison of the usual and microlocal $V$-filtration given by \eqref{eq:microlocalVcomparison}. The automorphism $t\mapsto t + (\lambda-1/2)\partial_t^{-1}$ of $\hE_{X\times \C}|_{\rJ^1 X}$ sends $\cC_f$ to $\cC^\lambda_f$ and $\theta\mapsto \theta + (\lambda-1/2)$ which proves the first claim for a general $\lambda$.

The second claim is merely the definition of the microlocalization functor.
\end{proof}

Using this microlocal perspective we can prove a relationship between vanishing cycles and DQ modules. Recall from \cref{thm:WHomdifferentialperverse} that
\[\cC^\rR_0\otimes^\bL_{\hW_X} \cC_f\]
is a differential perverse sheaf.

\begin{thm}\label{thm:Vanishing cycles via tensor product}
For any $\lambda,\mu\in\C$ there is an isomorphism of differential perverse sheaves
\[M_f\colon \cC^{\lambda,\rR}_0\otimes^\bL_{\hW_X} \cC^\mu_f\cong \RH^{-1}(\phi_f\otimes\C_{\mu+\lambda}),\]
where $\C_{\mu+\lambda}$ is the trivial one-dimensional vector space equipped with the automorphism $\exp(2\pi i(\mu+\lambda))$.
\label{thm:DQlocal}
\end{thm}
\begin{proof}
Consider the zero section $L = X\subset \T^* X$ and its contactification $\Lambda = j^10\subset \rJ^1 X$. $\cC_0$ as a $\hW_X$-module is simple along $L$. Using 
\cref{thm:Wtensormicrolocal} we get an isomorphism
\[\cC^{\lambda,\rR}_0\otimes^\bL_{\hW_X} \cC^\mu_f\cong \RH^{-1}\left((\mu^{dR}_\Lambda \cC^{\lambda,\rR}_0\otimes^\bL_{\cD_{\tilde{\Lambda}}} \mu^{dR}_\Lambda \cC^\mu_f)|_L[-1]\right)\]
of differential perverse sheaves.

Here $\mu^{dR}_L \cC_f$ is a twisted left $D$-module on $X\times \C^\times$ monodromic along $\C^\times$. By \cref{prop:microlocalmicrofunctions} and \cref{prop:monodromicRH} we identify
\[\mu^{dR}_L \cC^{\lambda,\rR}_0\otimes^\bL_{\cD_{X\times \C^\times}}\mu^{dR}_L \cC^\mu_f\cong K_X\otimes^\bL_{\cD_X} \phi_f^\cD\]
with the monodromy automorphism $T=\exp(-2\pi i \partial_t t + 2\pi i (\mu+\lambda))$ on the right. Finally, using \cref{prop:KashiwaraMalgrange} we identify
\[K_X\otimes^{\bL}_{\cD_X}\phi_f^\cD\cong \RcHom_{\cD_X}(\cO_X, \phi^\cD_f)[\dim X]\cong \phi_f.\]
\end{proof}

\begin{example}
Consider \cref{thm:DQlocal} for $X=\pt$. Namely, choose $\lambda,\mu\in\C$ and consider the differential perverse sheaf $\cC^{\lambda, \rR}_0\otimes^\bL_{\C(\!(\hbar)\!)} \cC^\mu_0$. From the definitions we obtain an isomorphism
\[\cC^{\lambda, \rR}_0\otimes^\bL_{\C(\!(\hbar)\!)} \cC^\mu_0\cong \C(\!(\hbar)\!)\]
where the differential structure on the right-hand side is given by $D(1) = \lambda+\mu+1$. Thus, there is an isomorphism
\[\cC^{\lambda, \rR}_0\otimes^\bL_{\C(\!(\hbar)\!)} \cC^\mu_0\cong \RH^{-1}(\C_{\mu+\lambda}).\]
\label{ex:tensortwist}
\end{example}

\subsection{Aside: twisted de Rham complex}

Consider the \defterm{twisted de Rham complex}
\[\DR_f = (\Omega^\bullet_X (\!(\hbar)\!), d + \hbar^{-1}df)|_{f^{-1}(0)}[\dim X].\]
In this section we explain a relationship between DQ modules and the twisted de Rham complex.

Recall that
\[\hW_X|_X\cong \cD_X(\!(\hbar)\!).\]
Thus, $\cC_f|_X$ is a certain $\cD_X(\!(\hbar)\!)$-module. To describe it, let us first analyze $\cB_f$. We may identify
\[\cB_f\cong \cO_X[\partial_t]^f,\]
where the superscript indicates the $\cD_{X\times \C}$-action as follows: the $\cO_X$ and the $\C[\partial_t]$-action is the obvious one; a vector field $\xi\in\cD_X$ acts by
\[\xi(g\partial_t^i) = (\xi g)\partial_t^i + (\xi f) g\partial_t^{i+1}\]
and $t\in\cD_\C$ acts by
\[t(g\partial_t^i) = -fg\partial_t^i - i g \partial_t^{i-1}.\]
The restriction $\cB_f|_X$ to the hypersurface $t=0$ is then
\[\cB_f|_X\cong \cO_{X, f^{-1}(0)}[\partial_t]^f\]

As before, let $\hbar = \partial_t^{-1}$ and consider the $\cD_X(\!(\hbar)\!)$-module $\cO_X(\!(\hbar)\!)^f$ with the $\cD_X(\!(\hbar)\!)$-action as above. Rewriting the $t$-action in terms of the variable $\hbar$ we get
\[t(g\hbar^i) = -fg \hbar^i + ig \hbar^{i+1},\]
so $t$ acts by $\hbar^2\partial_\hbar -f$.

\begin{thm}
For any $\lambda\in\C$ there is an isomorphism of differential perverse sheaves
\[(\DR_f, \hbar \partial_\hbar - \hbar^{-1} f + 1/2-\lambda)\cong \cC^\rR_0 \otimes^\bL_{\hW_X}\cC^\lambda_f.\]
\label{thm:twistedDRDQ}
\end{thm}
\begin{proof}
For simplicity we give a proof for $\lambda=1/2$ to simplify some shifts.

We may identify $\cC^\rR_0|_X\cong K_X(\!(\hbar)\!)$ with the obvious right $\cD_X(\!(\hbar)\!)$-action. By definition
\[\cC_f = \hE_{X\times \C}\otimes_{\cD_{X\times \C}} \cB_f,\]
where we omit obvious inverse image functors. Therefore, its restriction to $t=0$ is
\[\cC_f|_X\cong \hE_{X\times \C}|_X\otimes_{\cD_{X\times \C}|_X} \cB_f|_X\cong \hE_{\C, 0}\otimes_{\cD_{\C, 0}} \cB_f|_X.\]

Consider the de Rham resolution
\[\Omega^\bullet_X\otimes_{\cO_X}\cD_X(\!(\hbar)\!)[\dim X]\]
of the right $\cD_X(\!(\hbar)\!)$-module $K_X(\!(\hbar)\!)$. Using this resolution we obtain an isomorphism
\begin{align*}
\cC^\rR_0\otimes^\bL_{\hW_X}\cC_f &\cong K_X(\!(\hbar)\!)\otimes^\bL_{\cD_X(\!(\hbar)\!)} \cC_f|_X \\
&\cong \cC_f|_X\otimes_{\cO_X} \Omega^\bullet_X[\dim X],
\end{align*}
where the last complex of sheaves is equipped with the de Rham differential.

The $\cD_{\C, 0}$-action on $\cO_{X, f^{-1}(0)}(\!(\hbar)\!)^f$ extends to a $\hE_{\C, 0}$-action, so the natural inclusion
\[\cO_{X, f^{-1}(0)}[\hbar^{-1}]^f\hookrightarrow \cO_{X, f^{-1}(0)}(\!(\hbar)\!)^f\]
extends to a morphism
\[\cC_f|_X=\hE_{\C, 0}\otimes_{\cD_{\C, 0}}\cO_{X, f^{-1}(0)}[\hbar^{-1}]^f\longrightarrow \cO_{X, f^{-1}(0)}(\!(\hbar)\!)^f.\]
Under this morphism the action of $D=\partial_t t=\hbar^{-1} t$ on the left goes to the action of $D=\hbar^{-1}(\hbar^2\partial_\hbar - f) = \hbar\partial_\hbar - \hbar^{-1} f$ on the right.

By \cite[Proposition 2]{SabbahSaito} the induced morphism
\[K_X(\!(\hbar)\!)\otimes^\bL_{\cD_X(\!(\hbar)\!)}\cC_f|_X\longrightarrow \DR_f\]
is a quasi-isomorphism.
\end{proof}

\begin{example}
Let $W$ be a complex vector space equipped with a quadratic form $q\colon W\rightarrow \C$ and consider $\lambda\in\C$. The differential perverse sheaf $\cC^\rR_0\otimes^\bL_{\hW_W} \cC^\lambda_q$ is supported on the critical locus of $q$, i.e. at the origin. Therefore, $(\cC^\rR_0\otimes^\bL_{\hW_W} \cC^\lambda_q)|_0$ is a $\C(\!(\hbar)\!)$-vector space with an action of the derivation $D$. Let us compute it.

We have an isomorphism
\begin{equation}
P_q\otimes_{\Z/2}\C(\!(\hbar)\!)\longrightarrow \DR_f|_0
\label{eq:DRqtrivialization}
\end{equation}
given by $\omega\otimes 1\mapsto \omega$. To compute the derivation $D$, choose an identification $W\cong \C^n$ under which $q\mapsto \sum_i x_i^2$. Take $\omega = dx_1\wedge \dots\wedge dx_n$. Then
\begin{align*}
(\hbar\partial_\hbar - \hbar^{-1} q + 1/2 - \lambda)\omega &= (1/2-\lambda)\omega-\hbar^{-1} q \omega \\
&= \left(d + 2\hbar^{-1}\sum_i x_i dx_i\right) \alpha + \frac{n + 1 - 2\lambda}{2}\omega,
\end{align*}
where
\[\alpha = \sum_{i=1}^n (-1)^i x_i dx_1\wedge \dots \widehat{dx_i}\wedge\dots dx_n.\]
So, \eqref{eq:DRqtrivialization} is an isomorphism of differential vector spaces if we equip $\C(\!(\hbar)\!)$ with the differential structure with
\[D(1) = (\dim W + 1)/2 - \lambda.\]

Under the isomorphism constructed in \cref{thm:twistedDRDQ} the isomorphism \eqref{eq:DRqtrivialization} becomes the isomorphism
\[
P_q\otimes_{\Z/2}\C(\!(\hbar)\!)\longrightarrow (\cC^\rR_0\otimes^\bL_{\hW_W} \cC^\lambda_q)|_0
\]
given by
\[\omega\otimes 1\mapsto \omega\delta(t)\otimes \partial_t^{\lambda-1/2}\delta(t-q)\]
which is therefore also an isomorphism of differential vector spaces.
\label{ex:DRqtrivialization}
\end{example}

\begin{cor}
There is an isomorphism of differential perverse sheaves
\[(\DR_f, \hbar \partial_\hbar-\hbar^{-1} f)\cong \RH^{-1}(\phi_f).\]
\label{cor:DRphicomparison}
\end{cor}
\begin{proof}
The statement follows from a combination of \cref{thm:DQlocal} and \cref{thm:twistedDRDQ}.
\end{proof}

\begin{remark}
Such an isomorphism was established in \cite{Sabbah,SabbahSaito} on the level of cohomology sheaves and was recently upgraded in \cite{Schefers} to an isomorphism of differential perverse sheaves, by different methods than above.
\end{remark}

\subsection{Thom--Sebastiani theorem}

Throughout this section we consider complex manifolds $X_1, X_2$ and holomorphic functions $f_i\colon X_i \to \C$. Let $f_1\boxplus f_2\colon X_1\times X_2\rightarrow \C$ be the function $x_1,x_2\mapsto f_1(x_1) + f_2(x_2)$. We obviously have $f_1^{-1}(0)\times f_2^{-1}(0)\subseteq (f_1\boxplus f_2)^{-1}(0)$.

Consider the following external tensor product functors:
\begin{itemize}
    \item For a commutative ring $R$ and objects $\cM_i \in \bD(X_i; R)$, $i=1,2$, we have their external tensor product $\cM_1 \boxtimes_R \cM_2 \in \bD(X_1\times X_2; R)$.
    \item For $(\cF_i, T_i)\in \Pervaut(X_i)$ we define
    \[(\cF_1,T_1) \boxtimes (\cF_2, T_2) = (\cF_1 \boxtimes_\C \cF_2, T_1 \otimes T_2).\]
    \item For $(\cF_i, D_i)\in\Pervdiff(X_i)$ we define $(\cF_1,D_1) \boxtimes (\cF_2,D_2) = (\cF_1 \boxtimes_{\Cpp} \cF_2, D)$ with the differential structure $D(m_1 \otimes m_2) = Dm_1 \otimes m_2 + m_1 \otimes Dm_2$. 
\end{itemize}

By construction we have a commutative diagram
\[
\xymatrix{
\Pervaut(X_1)\times \Pervaut(X_2)\ar^-{\boxtimes}[r] \ar^{\RH^{-1}\times \RH^{-1}}[d] & \Pervaut(X_1\times X_2) \ar^{\RH^{-1}}[d] \\
\Pervdiff(X_1)\times \Pervdiff(X_2)\ar^-{\boxtimes}[r] & \Pervdiff(X_1\times X_2).
}
\]

Let us introduce analogs of these tensor products for microdifferential modules. Consider the following external tensor products:
\begin{itemize}
    \item Given $\cM_i \in \Mod(\cD_{X_i})$, $i=1,2$, we define 
    \[
    \cM_1 \boxD \cM_2 = \cD_{X_1 \times X_2} \bigotimes_{\cD_{X_1} \boxtimes_\C \cD_{X_2}} \cM_1 \boxtimes_\C \cM_2.
    \]
    \item Given $\cM_i \in \Mod(\hE_{X_i})$, $i=1,2$, we define 
    \[
    \cM_1 \boxE \cM_2 = \hE_{X_1 \times X_2} \bigotimes_{\hE_{X_1} \boxtimes_\C \hE_{X_2}} \cM_1 \boxtimes_\C \cM_2.
    \]
\end{itemize}

We denote by $\cD_\ast$ the transfer bimodule associated to the addition map
\[a\colon X_1 \times \C_{t_1} \times X_2 \times \C_{t_2} \longrightarrow X_1 \times X_2 \times \C_t\]
given by $(x_1,t_2,x_2,t_2) \mapsto (x_1,x_2, t_1 + t_2)$. Explicitly,
\[
\cD_\ast := \frac{\cD_{X_1 \times \C_{t_1} \times X_2 \times \C_{t_2}}}{(\partial_{t_1} - \partial_{t_2}) \cD_{X_1 \times \C_{t_1} \times X_2 \times \C_{t_2}}}
\]
This has the structure of a $(a^{-1}\cD_{X_1 \times X_2 \times \C_t}, \cD_{X_1 \times \C_{t_1} \times X_2 \times \C_{t_2}})$-bimodule, where $t$ acts on the left via $t_1 + t_2$. We have the associated functor
    \[
   \boxast\colon \Mod(\cD_{X_1\times \C_{t_1}}) \times \Mod(\cD_{X_2 \times \C_{t_2}}) \longrightarrow \Mod(\cD_{X_1 \times X_2 \times \C_t}),
    \]
    defined by
    \[
    \cM_1 \boxast \cM_2 := a_\ast \left(\cD_{\ast} \bigotimes_{\cD_{X_1 \times \C_{t_1} \times X_2 \times \C_{t_2}}} \cM_1 \boxD \cM_2 \right).
    \]

In this section, for any complex manifold $X$, we will consider $\hE_{X\times \C}$ as a sheaf on $\rJ^1 X \subseteq \P^\ast(X \times \C)$. We denote by $\alpha\colon \rJ^1 X_1 \times \rJ^1 X_2 \to \rJ^1(X_1 \times X_2)$ the addition map. Associated to this map consider the corresponding transfer $(\alpha^{-1}\hE_{X_1 \times X_2 \times \C_t}, \hE_{X_1 \times \C_{t_1} \times X_2 \times \C_{t_2}})$-bimodule (see \cite[Section 1.3]{SatoKawaiKashiwara} for the general notion of transfer bimodules)
\[
\hE_\ast := \frac{\hE_{X_1 \times \C_{t_1} \times X_2 \times \C_{t_2}}}{(\partial_{t_1} - \partial_{t_2}) \hE_{X_1 \times \C_{t_1} \times X_2 \times \C_{t_2}}}.
\]
Similarly, we have the associated functor
\[
    \boxast\colon \Mod(\hE_{X_1\times \C_{t_1}}) \times \Mod(\hE_{X_2 \times \C_{t_2}}) \longrightarrow \Mod(\hE_{X_1 \times X_2 \times \C_t}),
    \]
    defined by
    \[
    \cM_1 \boxast \cM_2 := \alpha_\ast \left(\hE_{\ast} \bigotimes_{\hE_{X_1 \times \C_{t_1} \times X_2 \times \C_{t_2}}} \cM_1 \boxE \cM_2 \right)
    \]

Again, we have a commutative diagram
        \[
\xymatrix{
\Mod(\cD_{X_1\times \C}) \times \Mod(\cD_{X_2\times \C}) \ar[d] \ar^-{\boxast}[r] & \Mod(\cD_{X_1\times X_2 \times \C})\ar[d] \\
\Mod(\hE_{X_1\times \C}) \times \Mod(\hE_{X_2\times \C}) \ar^-{\boxast}[r] & \Mod(\hE_{X_1\times X_2 \times \C}),
}
        \]
        where the vertical functors $\Mod(\cD_X) \to \Mod(\hE_X)$ are given by $\cM \mapsto \hE_X \otimes_{\pi^{-1}\cD_X} \pi^{-1}\cM$.

The functors $\boxast$ for $D$- and $E$-modules preserve regular holonomic modules. We will now use these functors to construct a Thom--Sebastiani isomorphism for the sheaf of vanishing cycles. We begin with a version of the Thom--Sebastiani isomorphism for $E$-modules.

\begin{prop}
For any $\lambda,\mu\in\C$ the morphism of $\hE_{X_1\times X_2\times\C}$-modules
\[
\TS^\cE_{f_1,f_2}\colon \cC^{\lambda+\mu-1/2}_{f_1\boxplus f_2}\longrightarrow \cC^\lambda_{f_1} \boxast \cC^\mu_{f_2}
\]
given by
\[\partial_t^{\lambda+\mu-1}\delta(t-f_1-f_2)\mapsto \partial_{t_1}^{\lambda-1/2}\delta(t_1-f_1)\otimes \partial_{t_2}^{\mu-1/2}\delta(t_2-f_2)\]
on the generators is an isomorphism.
\label{prop:EThomSebastiani}
\end{prop}
\begin{proof}
We have
\[\cC^\lambda_{f_1} \boxast \cC^\mu_{f_2} = \alpha_*\left(\hE_\ast\otimes_{\hE_{X_1\times\C\times X_2\times\C}} \cC^\lambda_{f_1} \boxE \cC^\mu_{f_2}\right).\]
The restriction of $\alpha\colon \rJ^1 X_1\times \rJ^1 X_2\rightarrow \rJ^1(X_1\times X_2)$ to the support of the sheaf $\hE_\ast\otimes_{\hE_{X_1\times\C\times X_2\times\C}} \cC^\lambda_{f_1} \boxE \cC^\mu_{f_2}$ is a closed immersion, so $\cC^\lambda_{f_1} \boxast \cC^\mu_{f_2}$ is a cyclic $\hE_{X_1\times X_2\times \C}$-module with generator \[\partial_{t_1}^{\lambda-1/2}\delta(t_1-f_1)\otimes \partial_{t_2}^{\mu-1/2}\delta(t_2-f_2).\] It is supported on a smooth Legendrian subvariety $j^1 f_1\timesast j^1 f_2\cong j^1(f_1\boxplus f_2)\subset \rJ^1(X_1\times X_2)$, so it is simple. But $\cC^{\lambda+\mu-1/2}_{f_1\boxplus f_2}$ is also a simple $\hE_{X_1\times X_2\times\C}$-module along the same Legendrian $j^1(f_1\boxplus f_2)$. Since the map between these two simple modules is nonzero (as it sends a generator to a generator), it is an isomorphism.
\end{proof}

Consider the Legendrians $\Lambda_i = j^1 0\subset \rJ^1 X_i$ and $\Lambda = j^1 0\subset \rJ^1 X$. The canonical $V$-filtrations on $\cC_{f_i}$ along $\Lambda_i$ induce by convolution a filtration (which we also denote by $V$) on $\cC_{f_1}\boxast \cC_{f_2}$. Consider also the canonical $V$-filtration on $\cC_{f_1\boxplus f_2}$ along $\Lambda$.

\begin{prop}
The morphism
\[\TS^\cE_{f_1,f_2}\colon (\cC_{f_1\boxplus f_2}, V)\longrightarrow (\cC_{f_1} \boxast \cC_{f_2}, V)
\]
is strictly compatible with the filtrations. In particular, after passing to the associated graded we obtain an isomorphism
\[\TS^\cD_{f_1, f_2}\colon \phi^\cD_{f_1\boxplus f_2}|_{f_1^{-1}(0)\times f_2^{-1}(0)}\longrightarrow \phi^\cD_{f_1}|_{f_1^{-1}(0)}\boxD \phi^\cD_{f_2}|_{f_2^{-1}(0)}
\]
in $\Mod_{\rh,\aut}(\cD_{f_1^{-1}(0)\times f_2^{-1}(0)})$.
\label{prop:DThomSebastiani}
\end{prop}
\begin{proof}
Let $\Sigma_i = \Sing f_i^{-1}(0)$. Possibly passing to open subsets of $X_1,X_2$, we may assume that $\Sing (f_1\boxplus f_2)^{-1}(0) = \Sigma_1\times \Sigma_2$.

Consider the algebraic microlocalization
\[\widetilde{\cB}_{f_i}=\C[\partial_{t_i}, \partial_{t_i}^{-1}]\otimes_{\C[\partial_{t_i}]}\cB_{f_i}.\]
The $V$-filtration on $\cB_{f_i}$ induces one on $\widetilde{\cB}_{f_i}$. Moreover, the natural maps $\pi^{-1}\cB_{f_i}\rightarrow \pi^{-1}\widetilde{\cB}_{f_i}\rightarrow \cC_{f_i}$ induce isomorphisms
\[\phi^\cD_{f_i}=\gr^0_V\cB_{f_i}\longrightarrow \gr^0_V \widetilde{\cB}_{f_i}\longrightarrow \gr^0_V \cC_{f_i}.\]
Indeed, for the first map this is shown in \cite[Section 1.1]{MaximSaitoSchurmann} and for the composite this is \eqref{eq:microlocalVcomparison}.

The $V$-filtrations on $\widetilde{\cB}_{f_i}$ induce by convolution a filtration on $\widetilde{\cB}_{f_1}\boxD \widetilde{\cB}_{f_2}$. It is shown in \cite[Theorem 1.2]{MaximSaitoSchurmann} and \cite[Lemma 11.8.2]{SabbahSchnell} that the natural map
\[
\widetilde{\cB}_{f_1} \boxD \widetilde{\cB}_{f_2} \longrightarrow \widetilde{\cB}_{f_1\boxplus f_2}
\]
defined by $\delta(t-f_1)\otimes \delta(t-f_2)\mapsto \delta(t-(f_1\boxplus f_2))$ is strictly compatible with the $V$-filtrations. Moreover, it induces an isomorphism after taking the cokernel of $\partial_{t_1} - \partial_{t_2}$ on the left-hand side. Applying $\gr^0_V$ we obtain the isomorphism $\TS^\cD_{f_1, f_2}$.
\end{proof}

Under the Riemann--Hilbert correspondence the Thom--Sebastiani isomorphism for $D$-modules goes to the Thom--Sebastiani isomorphism
\begin{equation}
\TS_{f_1, f_2}\colon \phi_{f_1}|_{f_1^{-1}(0)}\boxtimes \phi_{f_2}|_{f_2^{-1}(0)}\cong \phi_{f_1 \boxplus f_2}|_{f_1^{-1}(0)\times f_2^{-1}(0)}
\label{eq:ThomSebastiani}
\end{equation}
in $\Pervaut(f_1^{-1}(0)\times f_2^{-1}(0))$.

There is a natural morphism
\begin{equation}
(\cC_{0}^\rR \otimes^\bL_{\hW_{X_1}} \cC_{f_1}) \boxtimes (\cC_{0}^\rR \otimes^\bL_{\hW_{X_2}} \cC_{f_2}) \longrightarrow (\cC_0 \boxast \cC_0)^\rR \otimes^\bL_{\hW_{X_1 \times X_2}} (\cC_{f_1} \boxast \cC_{f_2})
\label{eq:compositemorphism}
\end{equation}
of differential perverse sheaves. The compatibility between the Thom--Sebastiani isomorphisms $\TS^\cE_{f_1, f_2}$ and $\TS_{f_1, f_2}$ is then expressed by the following commutative diagram.

\begin{prop}
The diagram
\[
\xymatrix{
(\cC_{0}^\rR \otimes^\bL_{\hW_{X_1}} \cC_{f_1}) \boxtimes (\cC_{0}^\rR \otimes^\bL_{\hW_{X_2}} \cC_{f_2}) \ar[d] \ar^{M_{f_1}\otimes M_{f_2}}[r] & \RH^{-1}(\phi_{f_1}\boxtimes \phi_{f_2}) \ar^{\RH^{-1}(\TS_{f_1, f_2})}[dd] \\
(\cC_0 \boxast \cC_0)^\rR \otimes^\bL_{\hW_{X_1 \times X_2}} (\cC_{f_1} \boxast \cC_{f_2}) \ar^-{\TS^\cE_{0,0}\otimes \TS^\cE_{f_1, f_2}}[d] & \\
\cC_0^\rR\otimes^\bL_{\hW_{X_1\times X_2}} \cC_{f_1\boxplus f_2} \ar^{M_{f_1\boxplus f_2}}[r] & \RH^{-1}(\phi_{f_1\boxplus f_2})
}
\]
commutes, where the vertical isomorphisms are given by the isomorphisms from \cref{thm:DQlocal}. In particular, \eqref{eq:compositemorphism} is an isomorphism.
\label{prop:TSagree}
\end{prop}
\begin{proof}
Unwinding the construction of $M_f$ in \cref{thm:Vanishing cycles via tensor product} via microlocalization, the top line of the diagram under consideration may be computed as the image under $\RH^{-1}$ of a sequence of morphisms of perverse sheaves with automorphism: 
\begin{align*}
(K_{X_1} \otimes_{\cD_{X_1}}^\bL \phi_{f_1}^\cD) \boxtimes (K_{X_2} \otimes_{\cD_{X_2}}^\bL \phi_{f_2}^\cD) 
&\to 
(K_{X_1} \boxD K_{X_2}) \otimes_{\cD_{X_1 \times X_2}}^\bL (\phi_{f_1} \boxD \phi_{f_2})\\
&\to 
K_{X_1\times X_2} \otimes_{\cD_{X_1\times X_2}}^\bL \phi_{f_1 \boxplus f_2}.
\end{align*}
\cref{prop:DThomSebastiani} then implies that the morphism on the right is induced by $\TS_{f_1,f_2}^\cD$ (together with the natural identification of right $\cD$-modules $K_{X_1} \boxD K_{X_2} \cong K_{X_1 \times X_2}$, which is manifestly induced by $\TS_{0,0}^\cE$ under microlocalization).
\end{proof}

\section{DT sheaf on oriented d-critical loci}\label{sec:DT}

\subsection{Exact two-forms on complex-analytic spaces}

Recall that given a chain complex $V$ in degrees $\geq -1$ one can canonically associate to it a groupoid $|V|$ by a basic version of the Dold--Kan correspondence: its objects are closed elements $x\in V^0$ of degree $0$ and morphisms from $x$ to $y$ are given by elements $h\in V^{-1}$ of degree $-1$ such that $dh = x - y$.

Let $X$ be a complex analytic space. We can define its cotangent complex $\bL_X$ together with a de Rham differential $d\colon \cO_X\rightarrow \bL_X$. The hypercohomology of this complex of sheaves $\R\Gamma(X, \cO_X\xrightarrow{d}\bL_X)$ is a complex concentrated in degrees $\geq 0$.

\begin{defn}
The groupoid of \defterm{exact two-forms of degree $0$} is
\[\Aex(X, 0) = |\R\Gamma(X, \cO_X\xrightarrow{d}\bL_X)[1]|.\]
The set of \defterm{exact two-forms of degree $-1$} is
\[\Aex(X, -1) = |\R\Gamma(X, \cO_X\xrightarrow{d}\bL_X)|.\]
\end{defn}

\begin{example}
If $X$ is a complex manifold with an open cover $\{U_i\}$ by open balls, then $\Aex(X, 0)$ has objects given by pairs $(f_{ij}, \alpha_i)$, where $f_{ij}$ is a function on $U_{ij}$ and $\alpha_i$ is a one-form on $U_i$ which together satisfy
\[\alpha_i - \alpha_j = df_{ij},\qquad f_{ij} + f_{jk} + f_{ki} = 0.\]
Morphisms from $(f_{ij}, \alpha_i)$ to $(g_{ij}, \beta_i)$ are given by functions $h_i$ such that
\[f_{ij} - g_{ij} = h_i - h_j,\qquad \alpha_i - \beta_i = dh_i.\]
Similarly, $\Aex(X, -1)$ is the set of locally constant functions on $X$.
\end{example}

These constructions have the following features:
\begin{itemize}
    \item If $Y\subset X$ is an analytic subset, there is a natural restriction map
\[\Aex(X, 0)\longrightarrow \Aex(Y, 0).\]
    \item One may identify
    \[\End_{\Aex(X, 0)}(0)\cong \Aex(X, -1).\]
    \item There is a forgetful map
    \[\Aex(X, -1)\longrightarrow \rH^0(X, \cO_X),\]
    so an exact two-form of degree $-1$ has an underlying function. Moreover, the restriction of this function to the reduced part $X^\red$ is locally constant.
\end{itemize}

\subsection{d-critical loci}
\label{sect:dCrit}

We will use the notions related to d-critical loci from \cite{BBDJS,JoycedCrit} which we now briefly recall. Let us first explain how to construct exact two-forms of degree $-1$ on critical loci.

\begin{prop}
Let $U$ be a complex manifold and $f\colon U\rightarrow \C$ is a holomorphic function. There is a natural exact two-form of degree $-1$ on the critical locus $\Crit(f)$. Its underlying function coincides with the restriction of $f$ to $\Crit(f)$.
\label{prop:Critoneform}
\end{prop}
\begin{proof}
For a one-form $\alpha$ on $U$ we denote by $\Gamma_\alpha\subset \T^* U$ its graph. The critical locus fits into a commutative diagram
\[
\xymatrix{
\Crit(f) \ar[r] \ar[d] & \Gamma_0 \ar[d] \\
\Gamma_{df} \ar[r] & \T^* U
}
\]
The Liouville one-form $\lambda$ on $\T^* U$ gives an element $\lambda\in \cA^{ex, }(\T^* U, 0)$. Its restriction to $\Gamma_{df}$ has a nullhomotopy provided by $f$ and similarly for $\Gamma_0$. Comparing the two nullhomotopies on $\Crit(f)$ we obtain an endomorphism $s\in\End_{\Aex(\Crit(f), 0)}(0)\cong \Aex(\Crit(f), -1)$.
\end{proof}

The above structure on a critical locus is globalized using the notion of a d-critical locus.

\begin{defn}
Let $X$ be a complex analytic space. A \defterm{d-critical structure on $X$} is the data of an exact two-form $s\in \Aex(X, -1)$ of degree $-1$ with the following property:
\begin{itemize}
    \item For any point $x\in X$ there is a neighborhood $R\subset X$ together with a complex manifold $U$, a holomorphic function $f\colon U\rightarrow \C$ and a closed embedding $i\colon R\hookrightarrow U$ such that $i(R) = \Crit(f)$, $s|_R$ coincides with the canonical exact two-form of degree $-1$ on $\Crit(f)$ constructed in \cref{prop:Critoneform} and $f|_{R^{\red}} = 0$. We call such a quadruple $(R, U, f, i)$ a \defterm{critical chart} of $X$.
\end{itemize}
if $X$ has a d-critical structure, we call $X$ a \defterm{d-critical locus}.
\end{defn}

We will now recall important results on ``transition maps'' between different critical charts.

\begin{defn}
Given a d-critical locus $(X, s)$, an \defterm{embedding} of critical charts $(R, U, f, i)\hookrightarrow (S, V, g, j)$, where $R\subset S$, is a locally closed embedding $\Xi\colon U\hookrightarrow V$ compatible with the rest of the data.
\end{defn}

We have the following important theorems proved in \cite{JoycedCrit}:
\begin{itemize}
\item Given an embedding $\Xi\colon (R, U, f, i)\hookrightarrow (S, V, g, j)$ of critical charts there is a canonical quadratic form $q_{\Xi}$ on the normal bundle $N_{\Xi}|_R$ of $U\subset V$. The quadratic vector bundle $N_{\Xi}|_R$ defines a $\Z/2$-torsor $P_\Xi$ on $R$ of orientations of $q_{\Xi}$. See \cite[Definition 2.26]{JoycedCrit}.

\item An embedding of critical charts $(R, U, f, i)\hookrightarrow (S, V, g, j)$ is, locally, of the form $(R, U, f, i)\hookrightarrow (R, U\times \C^n, f\boxplus (z_1^2+\dots +z_n^2), j)$, where $U\hookrightarrow U\times \C^n$ is given by $u\mapsto (u, 0)$. See \cite[Proposition 2.22]{JoycedCrit}.

\item Given critical charts $(R, U, f, i)$ and $(S, V, g, j)$ and an intersection point $x\in R\cap S$ we can find near $x$ another critical chart $(T, W, h, k)$ and local embeddings $(R, U, f, i)\hookrightarrow (T, W, h, k)$ and $(S, V, g, j)\hookrightarrow (T, W, h, k)$. See \cite[Theorem 2.20]{JoycedCrit}.
\end{itemize}

Another important structure on a d-critical locus $(X, s)$ is the \defterm{virtual canonical bundle} $K^{\vir}_X$, which is a holomorphic line bundle on $X^{\red}$ with the following properties:
\begin{itemize}
\item Given a critical chart $(R, U, f, i)$ there is an isomorphism $K^{\vir}_X|_{R^{\red}}\cong i^* K_U^{\otimes 2}|_{R^{\red}}$.
\item Given an embedding $\Xi\colon (R, U, f, i)\hookrightarrow (S, V, g, j)$ of critical charts there is an isomorphism
\[i^* K_U|_{R^{\red}}\otimes \det(N_{\Xi}|_R)\cong j^* K_V|_{R^{\red}}\]
whose square, together with a trivialization of $\det(N_{\Xi}|_R)^{\otimes 2}$ given by the quadratic form $q_{\Xi}$ on $N_{\Xi}|_R$, gives rise to an isomorphism of the presentations of the virtual canonical bundle in the two charts.
\end{itemize}

\begin{defn}
An \defterm{orientation} of a d-critical locus $(X, s)$ is the choice of a square root $(K^{\vir}_X)^{1/2}$ of the virtual canonical bundle.
\end{defn}

Suppose $(X, s)$ is an oriented d-critical locus. Then we have the following data in local charts:
\begin{itemize}
\item Given a critical chart $(R, U, f, i)$, there is a natural $\Z/2$-torsor $Q_{R, U, f, i}$ on $R$ measuring the difference between $(K^{\vir}_X)^{1/2}$ and $i^* K_U|_{R^{\red}}$.
\item Given an embedding $\Xi\colon (R, U, f, i)\hookrightarrow (S, V, g, j)$ of critical charts there is a canonical isomorphism
\begin{equation}
\Lambda_\Xi\colon Q_{S, V, g, j}|_R\xrightarrow{\sim} P_\Xi\otimes_{\Z/2} Q_{R, U, f, i}
\label{eq:Lambda}
\end{equation}
of $\Z/2$-torsors on $R$.
\end{itemize}

\subsection{Lagrangian intersections}
\label{sect:LagrangiandCrit}

In this section we explain the main example of a d-critical locus in this paper given by a Lagrangian intersection. Let $S$ be a holomorphic symplectic manifold and $L, M\subset S$ two Lagrangian submanifolds. Consider the intersection $L\cap M$ as a complex analytic space. The following is a generalization of \cref{prop:Critoneform} proved in \cite{Bussi}.

\begin{prop}
There is a natural d-critical structure on $L\cap M$. Moreover, there is a natural isomorphism of holomorphic line bundles
\[K^{\vir}_{L\cap M}\cong K_L|_{(L\cap M)^{\red}}\otimes K_M|_{(L\cap M)^{\red}}\]
on $(L\cap M)^{\red}$. In particular, choices of square roots $K_L^{1/2}$ and $K_M^{1/2}$ determine an orientation of $L\cap M$.
\label{prop:dcriticalLagrangianintersection}
\end{prop}
\begin{proof}
Consider the Lagrangian subvariety $L\cup M\subset S$. According to \cref{prop:uniquecontactification}, by shrinking $S$ to an open neighborhood of $L\cup M$, there is a canonical contactification $\rho\colon Y\rightarrow S$ together with a lift of $L, M\subset S$ to Legendrian submanifolds $\Lambda_L,\Lambda_M\subset S$ with the property that $\rho\colon \Lambda_L\cap \Lambda_M\rightarrow L\cap M$ is a homeomorphism.

The contactification $\rho\colon Y\rightarrow S$ defines an exact two-form of degree $0$ on $S$: $s\in\Aex(S, 0)$. The lift $\Lambda_L\subset Y$ of $L\subset S$ provides a nullhomotopy $h_L\colon s|_L\sim 0$ of the restriction of $s$ to $L$ and similarly for $M\subset S$. Consider a commutative diagram
\[
\xymatrix{
L\cap M \ar[r] \ar[d] & L \ar[d] \\
M \ar[r] & S
}
\]

Restricting the two nullhomotopies $h_L$ and $h_M$ to $L\cap M$ we obtain an endomorphism
\[s\in\End_{\Aex(L\cap M, 0)}(0)\cong \Aex(L\cap M, -1).\]
Let us now show that $s$ defines a d-critical structure on $L\cap M$.

Locally near a point $x\in L\cap M$ we may choose a polarization on $S$ transverse to $L$ and $M$. By \cref{prop:polarizationtwoLagrangians} we have a local identification of $Y\rightarrow S$ with $\rJ^1 L\rightarrow \T^* L$, $\Lambda_L\rightarrow L$ is identified with $j^10\rightarrow \T^*_L L$ and $\Lambda_M\rightarrow M$ is identified with $j^1 f\rightarrow \Gamma_{df}$. Thus, near $x$ we may identify $L\cap M \cong \Crit(f)$ and the d-critical structure $s$ reduces to the one on the critical locus $\Crit(f)$ constructed in \cref{prop:Critoneform}. This proves that $s$ defines a d-critical structure on $L\cap M$.

Moreover, it is clear that this d-critical structure on $L\cap M$ coincides with the one constructed in \cite[Theorem 3.1]{Bussi}, where the virtual canonical bundle is identified in the same way.
\end{proof}

\begin{remark}
As complex analytic spaces $L\cap M\cong M\cap L$. The corresponding two $d$-critical structures defined in \cref{prop:dcriticalLagrangianintersection} differ by a sign.
\end{remark}

\subsection{DT sheaf}

Let $(X, s)$ be an oriented d-critical locus. The paper \cite{BBDJS} constructs a canonical perverse sheaf $\phi_X\in\Pervaut(X)$ equipped with an automorphism on $X$, which we call the \defterm{DT sheaf}, locally modeled on the sheaf of vanishing cycles. Let us briefly recall its construction.

For a complex manifold $U$ with a holomorphic function $f\colon U\rightarrow \C$ let
\[\PV_{U, f} = (\phi_f, (-1)^{\dim U} T)\in\Pervaut(\Crit(f))\]
be the sheaf of vanishing cycles with the monodromy operator twisted by a sign.

For a vector space $W$ with a quadratic form $q\colon W\rightarrow\C$ the sheaf of vanishing cycles $\phi_q$ is supported at the origin, where it is given by a one-dimensional vector space, and has monodromy automorphism $(-1)^{\dim W}$. We consider its explicit trivialization as follows: there is a unique isomorphism
\begin{equation}
T_q\colon (\PV_{U, q})|_0\cong P_q\otimes_{\Z/2}\C
\label{eq:PVtrivialization}
\end{equation}
of vector spaces equipped with automorphisms such that the composite
\begin{equation}
P_q\otimes_{\Z/2}\C(\!(\hbar)\!)\xrightarrow{\RH^{-1}(T_q)}\RH^{-1}(\PV_{W, q})|_0\xleftarrow{M_q} (\cC^\rR_0\otimes_{\hW_W} \cC^{(\dim W+1)/2}_q)|_0
\label{eq:PVtrivializationcomposite}
\end{equation}
is given by $\omega\otimes 1\mapsto \omega\delta(t)\otimes \partial_t^{\dim W/2} \delta(t-q)$ (see \cref{ex:DRqtrivialization} for the composite isomorphism).

\begin{remark}
Any two choices of the isomorphisms $T_q$ compatible with the Thom--Sebastiani isomorphism \eqref{eq:ThomSebastiani} differ by a multiplication $\alpha^{\dim W}$ for $\alpha\in\C$.
\end{remark}

Let $\Xi\colon (R, U, f, i)\hookrightarrow (S, V, g, j)$ be an embedding of critical charts. Then \cite[Theorem 5.4]{BBDJS} shows that there is the \defterm{stabilization isomorphism}
\begin{equation}
\Theta_\Xi\colon \PV_{U, f}\longrightarrow (\PV_{V, g})|_{\Crit(f)}\otimes_{\Z/2} P_\Xi
\label{eq:Theta}
\end{equation}
of perverse sheaves equipped with automorphisms which has the following description:
\begin{itemize}
    \item Locally near a point of $R$ identify $V\cong R\times W$ and $g\mapsto f\boxplus q$ for a quadratic form $q\colon W\rightarrow\C$ using \cite[Proposition 2.22]{JoycedCrit}, so that under this isomorphism the embedding becomes $U\hookrightarrow U\times W$ given by $u\mapsto (u, 0)$. In this local model the isomorphism $\Theta_\Xi$ is equal to the composite
    \[(\phi_{f\boxplus q})|_{\Crit(f)}\cong \phi_f\otimes_\Q (\phi_q|_0)\cong \phi_f\otimes_{\Z/2} P_q,\]
    where the first isomorphism is the Thom--Sebastiani isomorphism \eqref{eq:ThomSebastiani} and the second isomorphism is $T_q$.
\end{itemize}

Combining the isomorphisms $\Lambda_\Xi$ \eqref{eq:Lambda} and $\Theta_\Xi$ \eqref{eq:Theta} we obtain an isomorphism of perverse sheaves equipped with automorphisms
\begin{equation}
\Theta_{\Xi}\otimes {\Lambda_\Xi}\colon \PV_{U, f}\otimes_{\Z/2} Q_{R, U, f, i}\longrightarrow (\PV_{S, V, g, j}\otimes_{\Z/2} Q_{S, v, g, j})|_{\Crit(f)}
\label{eq:PVstabilization}
\end{equation}
for any critical embedding $\Phi\colon (R, U, f, i)\hookrightarrow (S, V, g, j)$.

Let $(X, s)$ be an oriented $d$-critical locus. The DT sheaf $\phi_X$ is a perverse sheaf equipped with an automorphism on $X$ with the following property:
\begin{itemize}
    \item In any critical chart $(R, U, f, i)$ there is an isomorphism $\phi_X|_R\cong \PV_{U, f}\otimes_{\Z/2} Q_{R, U, f, i}$.
    \item By \cite[Theorem 2.20]{JoycedCrit} any two critical charts $(R, U, f, i)$, $(S, V, g, j)$ are locally related by a zig-zag $(R, U, f, i)\hookrightarrow (T, W, h, k)\hookleftarrow (S, V, g, j)$ of critical embeddings. The isomorphism of the local models for $\phi_X$ in charts $(R, U, f, i)$ and $(S, V, g, j)$ is given by applying the stabilization isomorphism \eqref{eq:PVstabilization} to the two critical embeddings.
\end{itemize}

\begin{remark}
We have chosen the isomorphism $T_q$ to define the sheaf $\phi_X$. Given two such isomorphisms differing by a multiplication by $\alpha^{\dim W}$ the resulting sheaves $\phi_X$ are canonically isomorphic: the isomorphism is given by a multiplication by $\alpha^{\dim U}$ in each critical chart $(R, U, f, i)$.
\end{remark}

We denote by $\phi^\cD_X\in\Mod_{\rh,\aut}(\cD_X)$ the $D$-module corresponding to $\phi_X$ under the Riemann--Hilbert correspondence.

\section{Comparison of the two perverse sheaves}

\subsection{Statement}

Throughout this section $S$ denotes a complex symplectic manifold, $\rho\colon Y\rightarrow S$ its contactification, $L, M\subset S$ Lagrangian submanifolds equipped with orientation data $K_L^{1/2},K_M^{1/2}$ and $\Lambda_L, \Lambda_M\subset Y$ their Legendrian lifts. The Lagrangian intersection $L\cap M$ has a structure of an oriented d-critical locus by \cref{prop:dcriticalLagrangianintersection}.

By \cref{thm:simpleclassification} microlocalization establishes an equivalence
\[\mu^{dR}_{\Lambda_L}\colon \Mod_{\Lambda, \rh}(\sfhE_Y)\xrightarrow{\sim} \gamma_*\Mod_{\locsys}(\cD^{\sqrt{v}}_{\tilde{\Lambda}_L}).\]
We may canonically identify $\tilde{\Lambda}_L\cong \Lambda_L\times \C^\times\cong L\times \C^\times$. Let $i_L\colon L\rightarrow \tilde{\Lambda}_L$ be the inclusion of $1\times \C^\times$ and similarly for $i_M\colon M\rightarrow \tilde{\Lambda}_M$. As explained in \cref{sect:monodromicDmodules}, for a monodromic (twisted) $D$-module $M$ on $\tilde{\Lambda}_L$ the pullback $i_L^* M$ is a (twisted) $D$-module on $L$ equipped with an automorphism.

Fix $\lambda\in\C$. Then, up to isomorphism, there is a unique $\sfhE_Y$-module $\sfM^\lambda_L$ simple along $\Lambda_L$ such that
\[i_L^*\mu_{\Lambda_L}^{dR}(\sfM^\lambda_L)\cong K_L^{1/2}\otimes \C_\lambda.\]
Similarly, for $\mu\in\C$ there is a unique $\sfhE_Y$-module $\sfM^\mu_M$ simple along $\Lambda_M$ such that
\[i_M^*\mu_{\Lambda_M}^{dR}(\sfM^\mu_M)\cong K_M^{1/2}\otimes \C_\mu.\]

The goal of this section is to prove the following result.

\begin{thm}
Let $\lambda,\mu\in\C$. There is an isomorphism of differential perverse sheaves
\[\sfM^{\lambda}_L\otimes^\bL_{\sfhW_S} \sfM^{\mu}_M\cong \RH^{-1}(\phi_{L\cap M}\otimes\C_{\mu+\lambda}),\]
where $\C_{\mu+\lambda}$ is the trivial vector space with the monodromy automorphism $T=\exp(2\pi i(\mu+\lambda))$.
\label{thm:maincomparison}
\end{thm}

\begin{cor}\label{cor:mainmicrolocalization}
Let $\mu\in\C$. There is an isomorphism
\[i_L^*\mu^{dR}_{\Lambda_L} \sfM^{\mu}_M\cong \phi^\cD_{L\cap M}\otimes \C_{\mu}\]
of $\cD_L$-modules equipped with automorphisms, where we identify $\cD^{\sqrt{v}}_L$-modules with $\cD_L$-modules using the orientation data $K_L^{1/2}$.
\end{cor}
\begin{proof}[Proof of \cref{cor:mainmicrolocalization}]
Pick any $\lambda\in\C$. By \cref{thm:maincomparison} we have
\[\sfM^{\lambda}_L\otimes^\bL_{\sfhW_S} \sfM^{\mu}_M\cong \RH^{-1}(\phi_{L\cap M}\otimes\C_{\mu+\lambda}).\]
By \cref{thm:Wtensormicrolocal} we have
\[\sfM^{\lambda}_L\otimes^\bL_{\sfhW_S} \sfM^{\mu}_M\cong \RH^{-1}\left((\mu^{dR}_{\Lambda_L} \sfM^{\lambda}_L\otimes^\bL_{\cD^{\sqrt{v}}_{\tilde{\Lambda}_L}} \mu^{dR}_{\Lambda_L} \sfM^{\mu}_M)|_L[-1]\right).\]
By definition $i^*\mu^{dR}_{\Lambda_L} \sfM^{\lambda}_L\cong K_L^{1/2}$ equipped with the automorphism $\exp(2\pi i\lambda)$. Therefore, we obtain an isomorphism
\[\C_{\mu+\lambda}\otimes \phi_{L\cap M}\cong \C_{\lambda}\otimes (K_L\otimes^\bL_{\cD_{\tilde{\Lambda}_L}} \mu^{dR}_{\Lambda_L} \sfM^{\mu}_M)|_L[-1]\]
of perverse sheaves on $L$ equipped with an automorphism. Identifying $D$-modules with perverse sheaves via the Riemann--Hilbert correspondence we get the result.
\end{proof}

Throughout this section we let $\dim L=\dim M=n$.

\subsection{d-critical charts}

Both sides of \cref{thm:maincomparison} involve gluing local models. So, we have to construct an isomorphism in local charts and show that the isomorphism is compatible with gluing data.

The manifold $S$ admits local polarizations. We use them to get a local description of $(Y\rightarrow S, \Lambda_L\rightarrow L, \Lambda_M\rightarrow M)$ following the terminology from \cite[Section 2]{Bussi}:
\begin{itemize}
\item An \defterm{$L$-chart} is an open subset $U\subset S$ equipped with a polarization, where we use \cref{prop:polarizationtwoLagrangians} to identify $\rho^{-1}(U)\rightarrow U$ with an open subset of $\rJ^1 L\rightarrow \T^* L$, $\Lambda_L\rightarrow L$ with $j^10\rightarrow \Gamma_0$ and $\Lambda_M\rightarrow M$ with $j^1 f\rightarrow \Gamma_{df}$ for a locally defined function $f\colon L\rightarrow \C$ such that $f|_{(L\cap M)^{\red}} = 0$. This defines a critical chart $(\Crit(f), L, f, i)$.
\item An \defterm{$M$-chart} is an open subset $U\subset S$ equipped with a polarization, where we use \cref{prop:polarizationtwoLagrangians} to identify $\rho^{-1}(U)\rightarrow U$ with an open subset of $\rJ^1 M\rightarrow \T^* M$, $\Lambda_M\rightarrow M$ with $j^10\rightarrow \Gamma_0$ and $\Lambda_L\rightarrow L$ with $j^1 g\rightarrow \Gamma_{dg}$ for a locally defined function $g\colon M\rightarrow \C$ such that $g|_{(L\cap M)^{\red}} = 0$. This defines a critical chart $(\Crit(g), M, -g, i)$.
\end{itemize}

To simplify the notation, throughout this section we will distinguish two such charts with polarizations $\pi_1$ and $\pi_2$, respectively; we make the first one into an $L$-chart and the second one into an $M$-chart. We assume the polarizations on the $L$-chart and the $M$-chart are transverse on the intersection. Then we may describe the intersection using an \defterm{$LM$-chart} as follows:
\begin{itemize}
    \item By \cref{prop:twopolarizationstwoLagrangians} there is a locally defined function $h\colon L\times M\rightarrow \C$ such that $\Gamma_{dh}$ gives (locally) the symplectomorphism between $\T^* L$ and $\T^* M$. The $L$-chart is embedded into the $LM$-chart via $\Xi\colon L\hookrightarrow L\times M$ under which $h$ restricts to $f$. The $M$-chart is embedded into the $LM$-chart via $\Upsilon\colon M\hookrightarrow L\times M$ under which $h$ restricts to $-g$. This defines a critical chart $(\Crit(h), L\times M, h, j)$.
\end{itemize}

We have the following description of $\phi_{L\cap M}$ in these charts:
\begin{itemize}
    \item Consider the $\Z/2$-torsor
    \[Q_L = Q_{\Crit(f), L, f, i}.\]
    The fiber of this $\Z/2$-torsor at $x\in L\cap M$ can be explicitly identified with the $\Z/2$-torsor of isomorphisms $K_{L, x}^{1/2}\rightarrow K_{M, x}^{1/2}$ squaring to the isomorphism
    \[C(\T_L, T_{\pi_1}, T_M)\colon K_{L, x}\rightarrow K_{M, x}.\]
    In other words, we have an isomorphism
    \[Q_L\cong P_{T_L, T_{\pi_1}, T_M},\]
    where the $\Z/2$-torsor on the right-hand side is introduced in \cref{def:orientationsequenceLagrangians}. The perverse sheaf $\phi_{L\cap M}$ in the $L$-chart is
    \[\phi_f\otimes_{\Z/2} Q_L.\]
    \item Consider the $\Z/2$-torsor
    \[Q_M = Q_{\Crit(g), M, -g, i}\cong P_{T_M, T_{\pi_2}, T_L}.\]
    The perverse sheaf $\phi_{L\cap M}$ in the $M$-chart is
    \[\phi_{-g}\otimes_{\Z/2} Q_M.\]
    \item Consider the $\Z/2$-torsor
    \[Q_{LM} = Q_{\Crit(h), L\times M, h, }\cong P_{T_L, T_{\pi_1}, T_{\pi_2}, T_M}.\]
    The perverse sheaf $\phi_{L\cap M}$ in the $LM$-chart is
    \[\phi_h\otimes_{\Z/2} Q_{LM}.\]
    \item Consider the embedding $\Xi\colon L\hookrightarrow L\times M$ of the $L$-chart into the $LM$-chart. Using \cref{prop:quadraticformcriticalembedding} we have
    \[P_{\Xi}\cong P_{T_M, T_{\pi_1}, T_{\pi_2}, T_M}.\]
    The isomorphism
    \[\Lambda_\Xi\colon Q_L\longrightarrow P_{\Xi}\otimes_{\Z/2} Q_{LM}\]
    is given by composition as in \cref{prop:Maslovcomposition}. We have the stabilization isomorphism
    \[\Theta_\Xi\colon \phi_f\longrightarrow \phi_h\otimes_{\Z/2} P_{\Xi},\]
    so that the isomorphism between the $L$-chart model and the $LM$-chart model is
    \[\Theta_\Xi\otimes \Lambda_\Xi\colon \phi_f\otimes_{\Z/2} Q_L\longrightarrow \phi_h\otimes_{\Z/2} Q_{LM}.\]
    \item Consider the embedding $\Upsilon\colon M\hookrightarrow L\times M$ of the $M$-chart into the $LM$-chart. We have
    \[P_{\Upsilon}\cong P_{T_L, T_{\pi_1}, T_{\pi_2}, T_L}\]
    and the isomorphism
    \[\Lambda_\Upsilon\colon Q_M\longrightarrow P_{\Upsilon}\otimes_{\Z/2} Q_{LM}\]
    is given by composition as in \cref{prop:Maslovcomposition}. We have the stabilization isomorphism
    \[\Theta_\Upsilon\colon \phi_{-g}\longrightarrow \phi_h\otimes_{\Z/2} P_{\Upsilon},\]
    so that the isomorphism between the $M$-chart model and the $LM$-chart model is
    \[\Theta_\Upsilon\otimes \Lambda_\Upsilon\colon \phi_{-g}\otimes_{\Z/2} Q_M\longrightarrow \phi_h\otimes_{\Z/2} Q_{LM}.\]
\end{itemize}

Let us also describe the orientation data on $L$ and $M$ in these charts. We use the following convention. Consider a complex manifold $U$ with an orientation data, i.e. it has a line bundle $K_U^{1/2}$ together with an isomorphism $(K_U^{1/2})^{\otimes 2}\cong K_U$. Locally $K_U^{1/2}$ admits a trivialization; once we choose a nonzero section of $K_U^{1/2}$, the isomorphism $(K_U^{1/2})^{\otimes 2}\cong K_U$ is determined by a volume form $\omega_U$. In this case we write $\sqrt{\omega_U}$ as the section of $K_U^{1/2}$. Changing the trivialization of $K_U^{1/2}$ by an invertible function $s$ the volume form changes by $\omega_U\mapsto s^2 \omega_U$.

So, we will choose the following additional data:
\begin{itemize}
    \item Choose a volume form $\omega_L$ on $L$, so that $\sqrt{\omega_L}$ defines a nonvanishing section of $K_L^{1/2}$.
    \item Choose a volume form $\omega_M$ on $M$, so that $\sqrt{\omega_M}$ defines a nonvanishing section of $K_M^{1/2}$.
\end{itemize}

Using these choices the $\Z/2$-torsors $Q_L$, $Q_M$, $Q_{LM}$ have the following explicit description:
\begin{itemize}
    \item $Q_L$ parametrizes functions $s$ on $L$ such that
    \[s^2 = \frac{(\pi_1^{S\rightarrow M})^*\omega_M}{\omega_L},\]
    where we identify $M$ and $L$ via the composite $\pi_1^{S\rightarrow M}\colon L\hookrightarrow S\xrightarrow{\pi_1} M$.
    \item $Q_M$ parametrizes functions $s$ on $M$ such that
    \[s^2 = \frac{(\pi_2^{S\rightarrow L})^*\omega_L}{\omega_M},\]
    where we identify $L$ and $M$ via the composite $\pi_1^{S\rightarrow L}\colon M\hookrightarrow S\xrightarrow{\pi_2} L$.
    \item $Q_{LM}$ parametrizes functions $u$ on $L\times M$ such that
    \[u^2 = \frac{(\pi_{12}^{S\times S\rightarrow S})^*\vol_S}{\omega_L\wedge \omega_M},\]
    where we identify $S$ with $L\times M$ via the composite $L\times M\hookrightarrow S\times S\xrightarrow{\pi_1\times \pi_2} S$.
    \item $P_\Xi$ parametrizes volume forms $\omega$ on $M$ such that $q_\Xi(\omega) = 1$. The isomorphism
    \[\Lambda_\Xi\colon Q_L\longrightarrow Q_{LM}\otimes_{\Z/2} P_\Xi\]
    is given by
    \[s\mapsto u\otimes \frac{u\omega_M}{s}.\]
    \item $P_\Upsilon$ parametrizes volume forms $\omega$ on $L$ such that $q_\Upsilon(\omega) = 1$. The isomorphism
    \[\Lambda_\Upsilon\colon Q_M\longrightarrow Q_{LM}\otimes_{\Z/2} P_\Upsilon\]
    is given by
    \[s\mapsto u\otimes \frac{u\omega_L}{s}.\]
\end{itemize}

\subsection{Local description of DQ modules}

Let us now give a local description of deformation quantization algebras and modules in the $L$- and $M$-charts from before.

\begin{itemize}
    \item In the $L$-chart we have $\sfhE_Y = \hE^{\sqrt{v}}_{L\times \C}$ and $\sfhW_S = \hW^{\sqrt{v}}_L$. The modules are
    \[\sfM^\lambda_L = \cC^{\lambda,\sqrt{L}}_0,\qquad \sfM^\mu_M = \cC^{\mu,\sqrt{M}}_f.\]
    We consider simple generators
    \[u_L^L=\sqrt{\omega_L}\partial_t^{\lambda-1/2}\delta(t)\in\cC^{\lambda, \sqrt{L}}_0,\qquad u_M^L=\sqrt{\omega_M}\partial_t^{\mu-1/2}\delta(t-f)\in\cC^{\mu, \sqrt{M}}_f,\]
    so that
    \[\sigma_{\Lambda_L}(u_L^L) = \tau^{\lambda-1/2}\sqrt{d\tau\wedge \omega_L},\qquad \sigma_{\Lambda_M}(u_M^L) = \tau^{\mu-1/2}\sqrt{d\tau\wedge \omega_M}.\]
    \item In the $M$-chart we have $\sfhE_Y = \hE^{\sqrt{v}}_{M\times \C}$ and $\sfhW_S = \hW^{\sqrt{v}}_M$. The modules are
    \[\sfM^\lambda_L = \cC^{\lambda,\sqrt{L}}_g,\qquad \sfM^\mu_M = \cC^{\mu,\sqrt{M}}_0.\]
    We consider simple generators
    \[u_L^M=\sqrt{\omega_L}\partial_t^{\lambda-1/2}\delta(t-g)\in\cC^{\lambda, \sqrt{L}}_g,\qquad u_M^M=\sqrt{\omega_M}\partial_t^{\mu-1/2}\delta(t)\in\cC^{\mu, \sqrt{M}}_0,\]
    so that
    \[\sigma_{\Lambda_L}(u_L^M) = \tau^{\lambda-1/2}\sqrt{d\tau\wedge \omega_L},\qquad \sigma_{\Lambda_M}(u_M^M) = \tau^{\mu-1/2}\sqrt{d\tau\wedge \omega_M}.\]
    \item On the intersection we have a local contactomorphism $\psi\colon \rJ^1 L\xrightarrow{\sim} \rJ^1 M$ covering the local symplectomorphism $\phi\colon \T^* L\xrightarrow{\sim} \T^* M$. We choose a $\ast$-preserving QCT $\Psi\colon \hE^{\sqrt{v}}_{M\times\C}\xrightarrow{\sim}\psi_* \hE^{\sqrt{v}}_{L\times\C}$ restricting to a $\ast$-preserving QST $\Phi\colon \hW^{\sqrt{v}}_M\xrightarrow{\sim}\phi_* \hW^{\sqrt{v}}_L$. There is an isomorphism
    \begin{equation}
    {}_{\Psi}\cC^{\lambda,\sqrt{L}}_0 \cong \cC^{\lambda,\sqrt{L}}_g
    \label{eq:LmoduleQCT}
    \end{equation}
    of $\hE^{\sqrt{v}}_{M\times\C}$-modules uniquely determined by the condition that there is an element $b^{LM}_L\in\hE^{\sqrt{v}}_{M\times\C}(0)$ such that under this isomorphism $u_L^L\mapsto b^{LM}_Lu_L^M$ and $\sigma_0(b^{LM}_L)|_{\Lambda_L} = 1$. There is also an isomorphism
    \begin{equation}
    \cC^{\mu,\sqrt{M}}_f\cong {}_{\Psi^{-1}}\cC^{\mu,\sqrt{M}}_0
    \label{eq:MmoduleQCT}
    \end{equation}
    of $\hE^{\sqrt{v}}_{L\times \C}$-modules uniquely determined by the condition that there is an element $b^{LM}_M\in\hE^{\sqrt{v}}_{L\times\C}(0)$ such that under this isomorphism $u_M^L\mapsto b^{LM}_Mu_M^M$ and $\sigma_0(b^{LM}_M)|_{\Lambda_M} = 1$.
\end{itemize}

Using \cref{prop:QCTbimodule} we encode the QCT $\Psi$ using a $\hE^{\sqrt{v}}_{L\times\C\times M\times\C}$-module $\cM$ simple along the Legendrian $\P^*_Z(L\times\C\times M\times\C)$, where $Z=\{t_L-t_M=h\}\subset L\times\C\times M\times\C$, together with a simple generator $K\in\cM$. Using \cref{prop:QSTbimodule} we encode the QST $\Phi$ using a $\hE^{\sqrt{v}}_{L\times M\times\C}$-module $\cM^\partial\subset \cM$ simple along the Legendrian $j^1 h\subset \rJ^1(L\times M)$, again with a simple generator $K\in\cM^\partial$. We have $\sigma_{j^1 h}(K) = \sqrt{\tau^n\vol_S\wedge d\tau}$, so we may locally identify
\[\cM^\partial\cong \cC^{(n+1)/2, \sqrt{S}}_h\]
as $\hE^{\sqrt{v}}_{L\times M\times \C}$-modules. Let
\[u_{LM} = \sqrt{\vol_S}\partial_t^{n/2}\delta(t-h)\in\cC^{(n+1)/2, \sqrt{S}}_h\]
be the standard simple generator. As $K$ and $u_{LM}$ are two simple generators with the same symbol, we have $K = b_{LM} u_{LM}$, where $b_{LM}\in \hE^{\sqrt{v}}_{L\times M\times \C}(0)$ with $\sigma_0(b_{LM})|_{j^1 h}$ locally constant. We rescale $K$ so that $\sigma_0(b_{LM})|_{j^1 h} = 1$.

Using $\cM$ we may rewrite isomorphisms \eqref{eq:LmoduleQCT} and \eqref{eq:MmoduleQCT} as
\[\Delta_\Upsilon\colon \cC_0^{\lambda,\sqrt{L}}\otimes_{\hW^{\sqrt{v}}_L} \cM^\partial\cong \cC_0^{\lambda,\sqrt{L}}\otimes_{\hE^{\sqrt{v}}_{L\times\C}} \cM\cong \cC^{\lambda,\sqrt{L}}_{-g},\]
an isomorphism of left $\hE^{\sqrt{v}}_{M\times\C}$-modules, uniquely determined by the fact that it sends $u_L^M\otimes b_L^{LM}K\mapsto u_L^L$ and
\[\Delta_\Xi\colon \cC_0^{\mu,\sqrt{M}}\otimes_{\hW^{\sqrt{v}}_M} \cM^\partial\cong \cC_0^{\mu,\sqrt{M}}\otimes_{\hE^{\sqrt{v}}_{M\times\C}}\cM\cong \cC_f^{\mu,\sqrt{M}},\]
an isomorphism of left $\hE^{\sqrt{v}}_{L\times\C}$-modules, uniquely determined by the fact that it sends $u_M^M\otimes b_M^{LM} K\mapsto u_M^L$.

We will use the following isomorphisms to change the orientation data for DQ modules:
\begin{itemize}
    \item A section of $Q_L$, i.e. a function $s$ on $L|_{L\cap M}$ such that $s^2 = (\pi_1^{S\rightarrow M})^*\omega_M / \omega_L$, determines an isomorphism $K_L^{1/2}\rightarrow K_M^{1/2}$ given by sending $s\sqrt{\omega_L}\mapsto \sqrt{\omega_M}$. Therefore, we obtain an isomorphism $\cC^{\mu,\sqrt{L}}_f\otimes Q_L\cong \cC^{\mu,\sqrt{M}}_f$ under which
    \[s\sqrt{\omega_L} \partial_t^{\mu-1/2}\delta(t-f)\otimes s\mapsto \sqrt{\omega_M}\partial_t^{\mu-1/2}\delta(t-f)=u^L_M.\]
    \item A section of $Q_M$, i.e. a function $s$ on $M|_{L\cap M}$ such that $s^2 = (\pi_2^{S\rightarrow L})^*\omega_L / \omega_M$, determines an isomorphism $K_M^{1/2}\rightarrow K_L^{1/2}$ given by sending $s\sqrt{\omega_M}\mapsto \sqrt{\omega_L}$. Therefore, we obtain an isomorphism $\cC^{\lambda,\sqrt{M}}_g\otimes Q_M\cong \cC^{\lambda,\sqrt{L}}_g$ under which
    \[su_L^M\otimes s= s\sqrt{\omega_L} \partial_t^{\lambda-1/2}\delta(t-g)\otimes s\mapsto \sqrt{\omega_M}\partial_t^{\lambda-1/2}\delta(t-g).\]
\end{itemize}

\subsection{Proof of \texorpdfstring{\cref{thm:maincomparison}}{Theorem \ref{thm:maincomparison}}}

We use the gluing data for the perverse sheaf $\phi_{L\cap M}$ as described in \cite[Section 2.3]{Bussi}. Cover $S$ by open sets which admit polarizations and consider them as $L$-charts. Given an intersection of two such $L$-charts $U_1, U_2$, we may cover $U_1\cap U_2$ by open sets $\{V_a\}$ with polarizations such that for every $a$ the polarization on $V_a$ is transverse both to the polarization on $U_1$ and on $U_2$. We regard $V_a$ as $M$-charts.

Thus, we have to provide an isomorphism of differential perverse sheaves as in \cref{thm:maincomparison} in each $L$- and $M$-chart and show that the isomorphisms (merely of perverse sheaves) are equal on the intersections.

\textbf{Step 1}. Consider an $L$-chart. By the explicit description of the modules given in the previous sections we have to construct an isomorphism of differential perverse sheaves
\[\cC^{\lambda,\sqrt{L}}_0\otimes^{\bL}_{\hW^{\sqrt{v}}_L} \cC^{\mu,\sqrt{M}}_f\cong \RH^{-1}(\phi_f\otimes \C_{\mu+\lambda})\otimes_{\Z/2} Q_L.\]
This isomorphism is given by the composite of the isomorphism $\cC^{\sqrt{M}}_f\cong \cC^{\sqrt{L}}_f\otimes_{\Z/2} Q_L$, the isomorphism
\[\cC^{\lambda,\sqrt{L}}_0\otimes^{\bL}_{\hW^{\sqrt{v}}_L} \cC^{\mu,\sqrt{L}}_f\cong \cC^{\lambda,\rR}_0\otimes^{\bL}_{\hW_L} \cC^\mu_f\]
and the isomorphism $M_f$ constructed in \cref{thm:DQlocal}.

\textbf{Step 2}. Consider an $M$-chart. In this case we have to construct an isomorphism of differential perverse sheaves
\[\cC^{\lambda,\sqrt{L}}_g\otimes^\bL_{\hW^{\sqrt{v}}_M}\cC^{\mu,\sqrt{M}}_0\cong \RH^{-1}(\phi_{-g})\otimes \C_{\mu+\lambda}\otimes_{\Z/2} Q_M.\]

This isomorphism is given by the composite of the isomorphism
\[\cC^{\lambda,\sqrt{L}}_g\otimes^\bL_{\hW^{\sqrt{v}}_M}\cC^{\mu,\sqrt{M}}_0\cong \cC^{\mu,\sqrt{M}}_0\otimes^\bL_{\hW^{\sqrt{v}}_M}\cC^{\lambda,\sqrt{L}}_{-g}\]
from \cref{lm:barflip} as well as the isomorphism constructed as in Step 1.

\textbf{Step 3}. We have to show that the isomorphisms defined above are equal on intersections of an $L$-chart and an $M$-chart. We describe the intersection using an $LM$-chart. In this case the isomorphism of DT sheaves in the local models is given by the composite
\[
\phi_f\otimes_{\Z/2} Q_L\xrightarrow{\Theta_\Xi\otimes \Lambda_\Xi} \phi_h\otimes_{\Z/2} Q_{LM}\xleftarrow{\Theta_\Upsilon\otimes \Lambda_\Upsilon} \phi_{-g}\otimes_{\Z/2} Q_M.
\]
Similarly, the isomorphism of the relative tensor product of DQ modules is given by the composite (we do not write the orders as we will be checking an equality of isomorphisms)
\begin{align*}
\cC^{\sqrt{L}}_0\otimes^\bL_{\hW^{\sqrt{v}}_L} \cC^{\sqrt{M}}_f&\xrightarrow{\id\otimes\Delta_\Xi} (\cC^{\sqrt{L}}_0\boxast \cC^{\sqrt{M}}_0)\otimes^\bL_{\hW^{\sqrt{v}}_L\boxtimes \hW^{\sqrt{v}}_M}\cC^{\sqrt{S}}_h\\
&\xleftarrow{\id\otimes \Delta_\Upsilon} \cC^{\sqrt{M}}_0\otimes^\bL_{\hW^{\sqrt{v}}_M} \cC^{\sqrt{L}}_{-g}\\
&\cong \cC^{\sqrt{L}}_g\otimes^\bL_{\hW^{\sqrt{v}}_M} \cC^{\sqrt{M}}_0.
\end{align*}

We will prove a stronger claim that the diagram
\begin{equation}
\xymatrix{
\cC^{\sqrt{L}}_0\otimes^\bL_{\hW^{\sqrt{v}}_L} \cC^{\sqrt{M}}_f \ar^-{\id\otimes\Delta_\Xi}[d]\ar^{M_f}[r] & \RH^{-1}(\phi_f)\otimes_{\Z/2} Q_L \ar^{\Theta_\Xi\otimes \Lambda_\Xi}[dd] \\
(\cC^{\sqrt{L}}_0\boxast \cC^{\sqrt{M}}_0)\otimes^\bL_{\hW^{\sqrt{v}}_L\boxtimes \hW^{\sqrt{v}}_M}\cC^{\sqrt{S}}_h \ar^-{\sim}[d] & \\
\cC^{\sqrt{L\times M}}_0 \otimes^\bL_{\hW^{\sqrt{v}}_{L\times M}} \cC^{\sqrt{S}}_h \ar^{M_h}[r] & \RH^{-1}(\phi_h)\otimes_{\Z/2} Q_{LM}
}
\end{equation}
commutes and similarly for $\Upsilon$. The two proofs are identical, so we will only give it for $\Xi$.

Untwisting the modules, we have to show that the diagram
\begin{equation}
\xymatrix@C=1.5cm{
\cC^{\rR}_0\otimes^\bL_{\hW_L} \cC_f\otimes_{\Z/2} Q_L \ar^-{\id\otimes \Delta'_\Xi}[r] \ar^{M_f}[d] & \cC^{\rR}_0\otimes^\bL_{\hW_{L\times M}} \cC_h\otimes_{\Z/2} Q_{LM} \ar^{M_h}[d] \\
\RH^{-1}(\phi_f)\otimes_{\Z/2} Q_L \ar^-{\Theta_\Xi\otimes\Lambda_\Xi}[r] & \RH^{-1}(\phi_h)\otimes_{\Z/2} Q_{LM}
}
\end{equation}
is commutative. Here
\[\Delta'_\Xi\colon \cC_f\otimes_{\Z/2} Q_L\longrightarrow \cC^{\rR}_0\otimes_{\hW_M} \cC_h\otimes_{\Z/2} Q_{LM}\]
is the isomorphism of left $\hE_{L\times \C}$-modules simple along $j^1 f$ under which
\[\partial_t^{\mu-1/2}\delta(t-f)\otimes s\mapsto \frac{u\omega_M}{s} \partial_t^{\mu-1/2}\delta(t)\otimes P\partial_t^{n/2}\delta(t-h)\otimes u\]
for $P = b^{LM}_Mb_{LM}\in\hE_{L\times M\times \C}(0)$ with $\sigma_0(P)|_{j^1 h} = 1$. Applying the isomorphism $\Lambda_\Xi\colon Q_L\cong Q_{LM}\otimes_{\Z/2} P_\Xi$ we are reduced to showing that the diagram
\begin{equation}
\xymatrix@C=1.5cm{
\cC^{\rR}_0\otimes^\bL_{\hW_L} \cC_f \ar^-{\id\otimes \Delta''_\Xi}[r] \ar^{M_f}[d] & \cC^{\rR}_0\otimes^\bL_{\hW_{L\times M}} \cC_h\otimes_{\Z/2} P_\Xi \ar^{M_h}[d] \\
\RH^{-1}(\phi_f) \ar^-{\Theta_\Xi}[r] & \RH^{-1}(\phi_h)\otimes_{\Z/2} P_\Xi
}
\end{equation}
commutes, where
\[\Delta''_\Xi\colon \cC_f\longrightarrow \cC^{\rR}_0\otimes_{\hW_M} \cC_h\otimes_{\Z/2} P_\Xi\]
is the isomorphism of left $\hE_{L\times \C}$-modules simple along $j^1 f$ under which
\[\partial_t^{\mu-1/2}\delta(t-f)\mapsto \omega\partial_t^{\mu-1/2}\delta(t)\otimes P\partial_t^{n/2}\delta(t-h)\otimes \omega.\]

It is enough to check commutativity of the diagram locally. Consider a point $x\in L\cap M$. We have that $\Xi\colon (L\cap M, L, f, i)\rightarrow (L\cap M, L\times M, h, i)$ is a critical embedding. Therefore, by \cite[Proposition 2.22]{JoycedCrit} in a neighborhood of $x$ there is a local isomorphism $L\times M\cong L\times W$, where $W=\T_x M$ is equipped with the Maslov quadratic form $q=q(\T_M, \T_{\pi_1}, \T_{\pi_2})\colon W\rightarrow \C$, such that under this isomorphism $h\mapsto f\boxplus q$ and the embedding $\Xi\colon L\rightarrow L\times M$ becomes $(\id\times 0)\colon L\rightarrow L\times W$. The $\Z/2$-torsor $P_\Xi$ in a neighborhood of $x$ canonically identifies with $P_q$.

Using these identification we obtain Thom--Sebastiani isomorphisms
\[\TS^\cE_{q,f}\colon \cC_q\boxast \cC_f\cong \cC_h\]
from \cref{prop:EThomSebastiani} and
\[\TS_{q,f}\colon \phi_q\boxtimes \phi_f\cong \phi_h\]
from \eqref{eq:ThomSebastiani}.

Consider the composite isomorphism
\begin{equation}
\cC_f\xrightarrow{\Delta''_\Xi}\cC^{\rR}_0\otimes_{\hW_M} \cC_h\otimes_{\Z/2} P_q\xleftarrow{\TS^\cE_{q,f}}(\cC^\rR_0\otimes_{\hW_W} \cC_q)\boxtimes \cC_f\otimes_{\Z/2} P_q
\label{eq:Deltacomposite}
\end{equation}
We claim that under this composite we have
\[\partial_t^{\mu-1/2}\delta(t-f)\mapsto \omega\delta(t)\otimes \partial_t^{n/2} \delta(t-q)\otimes \partial_t^{\mu-1/2}\delta(t-f)\otimes \omega.\]
In other words, $\Delta''_\Xi = \Delta'''_q\otimes\id$, where
\[\Delta'''_q\colon \C(\!(\hbar)\!)\longrightarrow (\cC^\rR_0\otimes_{\hW_W} \cC_q)\otimes_{\Z/2} P_q\]
is given by
\[1\mapsto \omega\delta(t)\otimes \partial_t^{n/2} \delta(t-q)\otimes \omega.\]
Indeed, both sides of \eqref{eq:Deltacomposite} are $\hE_{L\times\C}$-modules simple along $j^1f$. Any two isomorphisms of simple modules differ by a locally-constant $\C$-valued function and the claim follows from the equality $\sigma_0(P)_{j^1 h} = 1$. 

Therefore, we are reduced to showing that the diagram
\begin{equation}
\xymatrix{
\cC^{\rR}_0\otimes^\bL_{\hW_L} \cC_f \ar^-{\Delta'''_q\otimes\id}[d] \ar^{M_f}[r] & \RH^{-1}(\phi_f) \ar^-{\RH^{-1}(T_q)\otimes \id}[d] \\
P_q\otimes_{\Z/2}(\cC^\rR_0\otimes^\bL_{\hW_W} \cC_q)\boxtimes (\cC^{\rR}_0\otimes^\bL_{\hW_L} \cC_f) \ar^-{M_q\otimes M_f}[r]\ar^-{\id\otimes \TS^\cE_{q, f}}[d] & P_q\otimes_{\Z/2} \RH^{-1}(\phi_q\boxtimes\phi_f) \ar^{\id\otimes \RH^{-1}(\TS_{q, f})}[d] \\
P_q\otimes_{\Z/2} \cC^{\rR}_0\otimes^\bL_{\hW_{L\times W}} \cC_h \ar^{M_h}[r] & P_q\otimes_{\Z/2} \RH^{-1}(\phi_h)
}
\end{equation}
is commutative. The commutativity of the square on the left follows from the definition of $T_q$ (see \eqref{eq:PVtrivializationcomposite}). The commutativity of the square on the right follows from the compatibility of the Thom--Sebastiani isomorphisms $\TS^\cE$ and $\TS$ (see \cref{prop:TSagree}). This finishes the proof of \cref{thm:maincomparison}.

\subsection{Example: a clean Lagrangian intersection}
\label{sect:cleanintersection}

In this section we consider a simple application of \cref{thm:maincomparison}. Namely, we compute the perverse sheaf $\phi_{L\cap M}$ on the right-hand side when the Lagrangians $L,M$ intersect cleanly, i.e. when $L\cap M$ is smooth.

Since the d-critical locus $L\cap M$ is smooth, it defines a critical chart $(L\cap M, L\cap M, 0, \id)$ with respect to the zero potential. The orientation data on the Lagrangians induces an orientation data on the d-critical locus $L\cap M$. Since it is also a critical chart, we obtain an isomorphism
\[\phi_{L\cap M}\cong Q\otimes_{\Z/2} \C,\]
where $Q = Q_{L\cap M, L\cap M, 0, \id}$ is the $\Z/2$-torsor on the critical chart measuring the difference between the orientation of the d-critical locus and the natural orientation of the critical chart.

By definition, $Q$ parametrizes square roots of the natural isomorphism
\[\iota_{L\cap M, L\cap M, 0, \id}\colon K^{\vir}_{L\cap M}\longrightarrow (K_{L\cap M})^{\otimes 2}\]
described in \cite[Theorem 2.28 (i)]{JoycedCrit} for a general d-critical locus. By \cite[Theorem 3.1]{Bussi} we may additionally identify
\[K^{\vir}_{L\cap M}\cong K_L|_{L\cap M}\otimes K_M|_{L\cap M}.\]
The resulting isomorphism
\[\iota'\colon K_L|_{L\cap M}\otimes K_M|_{L\cap M}\cong (K_{L\cap M})^{\otimes 2}\]
has the following description. We have a long exact sequence
\[0\longrightarrow \T_{L\cap M}\longrightarrow \T_L|_{L\cap M}\oplus \T_M|_{L\cap M}\longrightarrow T_S|_{L\cap M}\longrightarrow \T^*_{L\cap M}\longrightarrow 0.\]
Taking its determinant and trivializing $\det(\T_S)$, using the symplectic volume form $\vol_S$ we get the isomorphism $\iota'$.

\section{A local Fukaya category for a holomorphic symplectic manifold}
\label{sect:categories}

In this section we will define a sheaf of dg categories associated to a holomorphic symplectic manifold which behaves as a local Fukaya category of holomorphic Lagrangians as conjectured by Behrend--Fantechi and Joyce.

\subsection{Holomorphic microlocal operators and their modules}

Let $X$ be a complex manifold. In \cref{sect:microdiffops} we have discussed the sheaf $\hE_X$ on $\P^* X$ of formal microdifferential operators. In this section we discuss several variants:
\begin{itemize}
    \item There is a sheaf $\cE_X$ of \defterm{microdifferential operators} on $\P^* X$ introduced in \cite{SatoKawaiKashiwara}. It is endowed with a filtration
    \[\cE_X = \cup_{m\in\Z} \cE_X(m)\]
    by the order of the microdifferential operator, so that
    \[\hE_X(m) = \lim_k \cE_X(m) / \cE_X(k)\]
    is the completion. In particular, there is a natural embedding $\cE_X\rightarrow \hE_X$ which is faithfully flat by \cite[Chapter II, Theorem 3.4.1]{SatoKawaiKashiwara}. There is a half-twisted version $\cE^{\sqrt{v}}_X$ of the sheaf $\cE_X$ defined analogously to $\hE^{\sqrt{v}}_X$.
    \item There is a sheaf $\cE^{\R, f}_X$ of \defterm{finite order holomorphic microlocal operators} on $\oT^* X$ introduced in \cite{Andronikof1} with a natural embedding $\gamma^{-1}\cE_X\rightarrow \cE^{\R, f}_X$ which is faithfully flat by \cite[Th\'eor\`eme 5.3.4]{Andronikof1} and which identifies $\cE_X\cong \gamma_* \cE^{\R, f}_X$. There is also a half-twisted version $\cE^{\sqrt{v},\R, f}_X$ of $\cE^{\R, f}_X$.
\end{itemize}

\begin{remark}
As we will see below, the utility of $\cE_X^\R$ (and its globalization $\sfE_Y^\R$) is that it forms a sheaf on the symplectization $\oT^\ast X$. This will allow us to see the monodromy action on the DT-sheaf in terms of the derived hom sheaf along the fibers of the symplectization map. There is no direct relationship between $\cE_X^\R$ and the ring of formal microdifferential operators $\hE_X$ introduced earlier in this paper, which is why we must also introduce the ring $\cE_X$. It is conjectured (see \cite[Conjecture 9.2]{KashiwaraVilonen}) that the induction functor from regular holonomic $\cE$-modules to $\hE_X$-modules is an equivalence of categories. 
\end{remark}

There is a notion of a (regular) holonomic $\cE_X$-module defined analogously to the notion of (regular) holonomic $\hE_X$-modules from \cref{def:Erh}. The induction functor $\cM\mapsto \hE_X\otimes_{\cE_X}\cM$ sends (regular) holonomic $\cE_X$-modules to (regular) holonomic $\hE_X$-modules. The following is \cite[Theorem 6.3]{KashiwaraKawaiMicrolocal}.

\begin{prop}
Let $X$ be a complex manifold and $\cL,\cM$ two regular holonomic $\cE_X$-modules. Then the natural morphism
\[\RcHom_{\cE_X}(\cL, \cM)\longrightarrow \RcHom_{\cE_X}(\cL, \hE_X\otimes_{\cE_X} \cM)\]
is an isomorphism.
\end{prop}

Quantized contact transformations from \cref{sect:QCT} work equally well for $\cE_X$ and $\cE^{\R, f}_X$ (see \cite{SatoKawaiKashiwara} for the former and \cite[Section 5.1]{Andronikof1} for the latter). Namely, suppose $V_{X_1}\subset \P^*X_1$ and $V_{X_2}\subset \P^* X_2$ are open subsets together with a contactomorphism $\psi\colon V_{X_1}\xrightarrow{\sim} V_{X_2}$ and denote by $\tilde{\psi}\colon \tilde{V}_{X_1}\xrightarrow{\sim} \tilde{V}_{X_2}$ the corresponding homogeneous symplectomorphism of open subsets $\tilde{V}_{X_1}\subset \oT^* X_1$ and $\tilde{V}_{X_2}\subset \oT^* X_2$. Then there is a $\ast$-preserving quantized contact transformation over $\psi$
\[\Psi\colon \hE^{\sqrt{v}}_{X_2}|_{V_{X_2}}\xrightarrow{\sim} \psi_* \hE^{\sqrt{v}}_{X_1}|_{V_{X_1}}\]
which restricts to an isomorphism
\[\Psi\colon \cE^{\sqrt{v}}_{X_2}|_{V_{X_2}}\xrightarrow{\sim} \psi_* \cE^{\sqrt{v}}_{X_1}|_{V_{X_1}}\]
and extends to an isomorphism
\[\tilde{\Psi}\colon \cE^{\sqrt{v}, \R, f}_{X_2}|_{\tilde{V}_{X_2}}\xrightarrow{\sim} \tilde{\psi}_* \cE^{\sqrt{v}, \R, f}_{X_1}|_{\tilde{V}_{X_1}}.\]

In particular, if $Y$ is a holomorphic contact manifold with $\gamma\colon\tilde{Y}\rightarrow Y$ its symplectization, then there is an algebroid $\sfE_Y$ on $Y$ and an algebroid $\sfE^{\R, f}_{\tilde{Y}}$ on $\tilde{Y}$, locally given by $\cE^{\sqrt{v}}_X$ and $\cE^{\sqrt{v}, \R, f}_X$, together with embeddings
\[\sfE_Y\subset \sfhE_Y,\qquad \gamma^{-1}\sfE_Y\subset \sfE^{\R, f}_{\tilde{Y}}.\]

Analogously to $\sfhE_Y$-modules, for a Legendrian subvariety $\Lambda\subset Y$ one may define a microlocalization functor
\[
\mu_\Lambda^{dR}\colon \Mod_{\rh}(\sfE_Y)\longrightarrow \gamma_*\Mod_{\rh}(\cD^{\sqrt{v}}_{\tilde{\Lambda}})
\]
either by precomposing the microlocalization functor $\mu_\Lambda^{dR}$ of $\sfhE_Y$-modules with the induction functor along the inclusion $\sfE_Y\subset \sfhE_Y$ or by constructing the subalgebroid $\sfE_{\Lambda/Y}\subset \sfE_Y|_\Lambda$ analogously to $\sfhE_{\Lambda/Y}\subset \sfhE_Y|_\Lambda$.

Moreover, if $\Lambda\subset Y$ is a Legendrian submanifold, \cref{thm:simpleclassification} still holds for $\sfE_Y$-modules: microlocalization defines an equivalence of stacks
\[\mu_\Lambda^{dR}\colon \Mod_{\Lambda, \rh}(\sfE_Y)\xrightarrow{\sim} \gamma_*\Mod_{\locsys}(\cD^{\sqrt{v}}_{\tilde{\Lambda}}).\]

For an $\sfE_Y$-module $\cM$ we denote
\[\cM^{\R, f}=\sfE^{\R, f}_{\tilde{Y}}\otimes_{\gamma^{-1}\sfE_Y} \gamma^{-1}\cM\]
the corresponding $\sfE^{\R, f}_{\tilde{Y}}$-module. We have the following result.

\begin{prop}\label{prop:ERHomperverse}
Let $Y$ be a contact manifold and $\cL$ and $\cM$ two regular holonomic $\sfE_Y$-modules. Then
\[\RcHom_{\sfE^{\R, f}_{\tilde{Y}}}(\cM^{\R, f}, \cL^{\R, f})[(\dim Y+1)/2]\]
is a perverse sheaf on $\tilde{Y}$, which is locally constant along the fibers of $\gamma\colon \tilde{Y}\rightarrow Y$.
\end{prop}
\begin{proof}
The statement is local on $Y$, so we may assume that $Y$ is an open subset of $\P^* X$ for some complex manifold $X$. By \cite[Theorem 7.11]{KashiwaraKawaiMicrolocal} and \cite[Corollaire 5.6.4]{Andronikof1} there is an isomorphism
\[\RcHom_{\sfE^{\R, f}_{\tilde{Y}}}(\cM^{\R, f}, \cL^{\R, f})\cong \RcHom_{\cD^{\sqrt{v}}_{\oT^*_X(X\times X)}}(\Phi(\cL, \cM^*), \cO_{\oT^*_X(X\times X)})[\dim X],\]
where
\[\Phi(\cL, \cM^*)=\mu^{dR}_{\overline{\Delta}}(\cL\boxtimes^{\cE} \cM^*)\]
is the microlocalization of the $\cE^{\sqrt{v}}_{X\times X}$-module $\cL\boxtimes^{\cE} \cM^*$ to the Legendrian $\overline{\Delta}=\P^*_X(X\times X)\subset \P^*(X\times X)$. By construction the microlocalization functor produces a regular holonomic $D$-module on $\oT^*_X(X\times X)$ which is monodromic along the fibers of $\gamma\colon \oT^*_X(X\times X)\rightarrow \P^*_X(X\times X)$. Applying the solution functor $\RcHom_{\cD^{\sqrt{v}}_{\oT^*_X(X\times X)}}(-, \cO_{\oT^*_X(X\times X)})[2\dim X]$ we therefore obtain a perverse sheaf, which is locally constant along the fibers of $\gamma$.
\end{proof}

\begin{cor}\label{cor:ERRHom}
Let $Y$ be a contact manifold and $\cL$ and $\cM$ two bounded complexes of $\sfE_Y$-modules with regular holonomic cohomology sheaves. Then $\RcHom_{\sfE^{\R, f}_{\tilde{Y}}}(\cM^{\R, f}, \cL^{\R, f})$ is a constructible complex on $\tilde{Y}$, which is locally constant along the fibers of $\gamma\colon \tilde{Y}\rightarrow Y$.
\end{cor}

When one of the modules is regular along a Legendrian submanifold, the corresponding $\RcHom$ can be computed by microlocalization.

\begin{prop}\label{prop:ERHommicrolocalization}
Let $Y$ be a contact manifold, $\cM$ a regular holonomic $\sfE_Y$-module and $\cL$ an $\sfE_Y$-module regular along a Legendrian submanifold $\Lambda\subset Y$. Then there is an isomorphism
\[\RcHom_{\sfE^{\R, f}_{\tilde{Y}}}(\cM^{\R, f}, \cL^{\R, f})|_{\tilde{\Lambda}}\cong \RcHom_{\cD^{\sqrt{v}}_{\tilde{\Lambda}}}(\mu_\Lambda^{dR}(\cM), \mu_\Lambda^{dR}(\cL)).\]
\end{prop}
\begin{proof}
By \cref{prop:ERHomperverse}
\[\cF_1=\RcHom_{\sfE^{\R, f}_{\tilde{Y}}}(\cM^{\R, f}, \cL^{\R, f})|_{\tilde{\Lambda}}[\dim \tilde{\Lambda}]\]
is a perverse sheaf on $\tilde{\Lambda}$, which is locally constant along the fibers of $\gamma\colon \tilde{\Lambda}\rightarrow \Lambda$. Similarly, since $\mu_\Lambda^{dR}(\cM)$ and $\mu_\Lambda^{dR}(\cL)$ are $\cD^{\sqrt{v}}_{\tilde{\Lambda}}$-modules, which are locally constant along the fibers of $\gamma$,
\[\cF_2=\RcHom_{\cD^{\sqrt{v}}_{\tilde{\Lambda}}}(\mu_\Lambda^{dR}(\cM), \mu_\Lambda^{dR}(\cL))[\dim \tilde{\Lambda}]\]
is a perverse sheaf on $\tilde{\Lambda}$, which is locally constant along the fibers of $\gamma$.

Let $V\in\LocSys(\tilde{\Lambda})$ be a local system on $\tilde{\Lambda}$. By the classification of $\sfE_Y$-modules regular along $\Lambda$ we may find a $\sfE_Y$-module $\cL_V$ regular along $\Lambda$ such that
\[\mu_\Lambda^{dR}(\cL_V)\cong \mu_\Lambda^{dR}(\cL)\otimes V,\qquad (\cL_V)^{\R, f}\cong \cL^{\R, f}\otimes V.\]
Therefore, by \cref{thm:simpleclassification} we obtain an isomorphism
\[\gamma_*(\cF_1\otimes V)\cong \gamma_*(\cF_2\otimes V)\]
of constructible complexes on $\Lambda$, natural in $V$. In particular, we obtain an isomorphism ${}^p\cH^{-1}(\gamma_*(\cF_1\otimes V))\cong {}^p\cH^{-1}(\gamma_*(\cF_2\otimes V))$ of perverse sheaves on $\Lambda$, natural in $V$. By \cref{lm:automorphismYoneda} it lifts to an isomorphism $\cF_1\cong \cF_2$ of perverse sheaves on $\tilde{\Lambda}$.
\end{proof}

\subsection{Construction of the category}
\label{sec:construction of the category}

Recall that for a monoidal $\infty$-category $\cV$, one has the notion of a $\cV$-enriched $\infty$-category $\cC$, see \cite{GepnerHaugseng}. Given two objects $x,y\in\cC$ in a $\cV$-enriched $\infty$-category we have the Hom object $\Hom_\cC(x, y)\in\cV$. When $R$ is a commutative ring and $\cV=\bD(R)$ is the derived $\infty$-category of chain complexes of $R$-modules, an $R$-linear $\infty$-category is, by definition, a $\cV$-enriched $\infty$-category. If $\cC$ is a sheaf of $R$-enriched $\infty$-categories on a complex manifold $X$, for two global objects $x,y\in\cC$ we may define both the usual Hom object $\Hom_\cC(x, y)\in\bD(R)$ as well as the sheaf Hom object $\cHom_{\cC}(x, y)\in\bD(X; R)$.

In this section for a holomorphic symplectic manifold $S$ we define a sheaf of $\C$-linear $\infty$-categories $\sfDRH_S$ on $S$ by slightly generalizing the construction of \cite{DAgnoloKashiwara}. We begin by observing that the set $\Lag_S$ of Lagrangian subvarieties in $S$ is directed: given two Lagrangian subvarieties $L_1, L_2\subset S$, their union $L_1\cup L_2\subset S$ is again a Lagrangian subvariety and it contains both $L_1$ and $L_2$. Let $\Lag^{\contact}_S$ be the category of triples $(L, Y, \Lambda)$ of a Lagrangian subvariety $L\subset S$ a contactification $\rho\colon Y\rightarrow U$ of a neighborhood $U$ of $L$ and a lift of $L$ to a Legendrian subvariety $\Lambda\subset Y$. By \cref{prop:uniquecontactification} the forgetful functor $\Lag^{\contact}_S\rightarrow \Lag_S$ is an equivalence.

First, let us fix an object $(L, Y, \Lambda)\in\Lag_S^{\contact}$. Consider the $\C$-linear algebroid $\sfE^{\R, f}_{\tilde{Y}}$ on $\tilde{Y}$. Then the sheaf of presentable $\infty$-categories $\bD(\sfE^{\R, f}_{\tilde{Y}})$ on $\tilde{Y}$ is tensored over $\bD(\C)$, so by \cite[Corollary 7.4.13]{GepnerHaugseng} it is $\C$-linear. Since $\gamma_*\sfE^{\R, f}_{\tilde{Y}}\cong \sfE_Y$, we have a fully faithful functor of $\C$-linear
$\infty$-categories
\[\bD(\sfE_Y)\longrightarrow \gamma_* \bD(\sfE^{\R, f}_{\tilde{Y}})\]
given by $\cM\mapsto \cM^{\R, f}$. Thus, $\bD(\sfE_Y)$ canonically enhances to a $\bD(\C^\times; \C)$-enriched $\infty$-category, so we have two versions of sheaf Hom objects:
\[\cHom(\cM, \cL) = \RcHom_{\sfE_Y}(\cM, \cL),\qquad \underline{\cHom}(\cM, \cL) = \RcHom_{\sfE^{\R, f}_Y}(\cM^{\R, f}, \cL^{\R, f}).\]

Let $\bD_{\rh}(\sfE_Y)\subset \bD(\sfE_Y)$ be the full subcategory of objects of bounded cohomological amplitude with cohomology sheaves being regular holonomic $\sfE_Y$-modules. Let $\bD_{\Lambda, \rh}(\sfE_Y)\subset \bD_{\rh}(\sfE_Y)$ be the subcategory of objects supported on a Lagrangian subvariety $\Lambda\subset Y$. Set
\[\sfDRH_{S, L} = \rho_* \bD_{\Lambda, \rh}(\sfE_Y),\]
which, as before, is a sheaf of $\bD(\C^\times; \C)$-enriched $\infty$-categories. 

Using \cref{cor:ERRHom} we see that, in fact, $\sfDRH_{S, L}$ is enriched over the full subcategory $\bD_{\locsys}(\C^\times; \C)$ of $\C$-linear local systems on $\C^\times$. 
Taking monodromy at the basepoint $1\in \C^\times$ provides an equivalence of symmetric monoidal $\infty$-categories
\[\bD_{\locsys}(\C^\times; \C)\cong \bD(\C[T^{\pm 1}])^{\conv}\]
Here, $T$ denotes the generator of $\pi_1(\C^\times, 1)$ and so $\C[T^{\pm 1}]$ is identified with the group algebra of $\pi_1(\C^\times,1)$. The right hand side is equipped with the convolution monoidal structure, that is, the monoidal structure coming from the standard Hopf algebra structure on the group algebra. 
Since the forgetful functor
\[\bD(\C[T^{\pm 1}])^{\conv}\longrightarrow \bD(\C)\]
is monoidal, by forgetting the monodromy action, we may also regard $\sfDRH_{S, L}$ as a sheaf of $\C$-linear $\infty$-categories. 

So far, for a fixed object $(L, Y, \Lambda)\in\Lag^{\contact}_S$, we have defined a sheaf of $\bD(\C[T^{\pm 1}])^{\conv}$-enriched $\infty$-categories $\sfDRH_{S, L}$ on $L$. The construction is obviously functorial with respect to inclusions of Lagrangian/Legendrian subvarieties and we set
\[\sfDRH_S = \colim_{(L, Y, \Lambda)\in\Lag^{\contact}_S} i_*\sfDRH_{S, L},\]
where $i\colon L\rightarrow S$ is the natural inclusion. We may summarize the constructions of this section as follows.

\begin{thm}\label{thm:RHproperties}
Let $S$ be a symplectic manifold. There is a sheaf $\sfDRH_S$ of $\bD(\C[T^{\pm 1}])^{\conv}$-enriched $\infty$-categories on $S$ with the following properties:
\begin{enumerate}
    \item It is equipped with a $t$-structure; we denote by $\sfRH_S\subset \sfDRH_S$ the heart.
    \item For any two objects $\cM, \cL\in\sfDRH_S$ the internal Hom sheaf \[\underline{\cHom}_{\sfDRH_S}(\cM, \cL)\in\bD(S; \C[T^{\pm 1}])\] is constructible as a sheaf of complexes of $\C$-vector spaces.
    \item For any two objects $\cM, \cL\in\sfRH_S$ the sheaf $\underline{\cHom}_{\sfDRH_S}(\cM, \cL)[\dim S/2]$ is perverse.
    \item If $L\subset S$ is a Lagrangian submanifold, there is a full subcategory $\gamma_*\Mod_{\locsys}(\cD^{\sqrt{v}}_{L\times \C^\times})\subset \sfRH_S|_L$ of half-twisted $D$-modules on $L\times \C^\times$, where $\gamma\colon L\times \C^\times\rightarrow L$ is the projection. In particular, if $L$ is equipped with an orientation data $K_L^{1/2}$ and $\lambda\in\C/\Z$, there is a canonical object $\sfM^\lambda_L\in\sfRH_S$ supported on $L$ corresponding to the trivial local system on $L$ with monodromy $\exp(2\pi i\lambda)$ around $\C^\times$.
    \item Let $L, M\subset S$ be Lagrangian submanifolds equipped with orientation data $K_L^{1/2}, K_M^{1/2}$ and $\lambda,\mu\in\C/\Z$. Then there is an isomorphism
    \[\underline{\cHom}_{\sfDRH_S}(\sfM^{\mu}_M, \sfM^{\lambda}_L)[\dim S/2]\cong \phi_{L\cap M}\otimes \C_{\lambda-\mu}\]
    of perverse sheaves equipped with a monodromy operator, where $\C_{\lambda-\mu}$ is the trivial one-dimensional local system equipped with the monodromy operator $\exp(2\pi i(\lambda-\mu))$.
\end{enumerate}
\end{thm}
\begin{proof}
Let us explain why the sheaf $\sfDRH_S$ constructed in this section satisfies the asserted properties:
\begin{enumerate}
    \item The natural $t$-structure on the derived $\infty$-category $\bD(\sfE_Y)$ restricts to a $t$-structure on $\sfDRH_{S, L} = \rho_*\bD_{\Lambda, \rh}(\sfE_Y)$. Inclusion $i\colon L_1\rightarrow L_2$ of Lagrangian subvarieties of $S$ induces a $t$-exact functor $i_*\sfDRH_{S, L_1}\rightarrow \sfDRH_{S, L_2}$, so the $t$-structure descends to the colimit $\sfDRH_S$. The heart $\RH_{S, L}$ of the $t$-structure on $\sfDRH_{S, L}$ is given by the sheaf $\rho_*\Mod_{\Lambda, \rh}(\sfE_Y)$ of regular holonomic $\sfE_Y$-modules supported on $\Lambda$.
    \item By definition, we may find an object $(L, Y, \Lambda)\in\Lag^{\contact}_S$, such that $\cM, \cL\in\bD_{\Lambda, \rh}(\sfE_Y)$. Moreover, we have
    \[\underline{\cHom}_{\sfDRH_S}(\cM, \cL) = \RcHom_{\sfE^{\R, f}_{\tilde{Y}}}(\cM^{\R, f}, \cL^{\R, f})|_{\Lambda\times\{1\}}[-1].\]
    By \cref{cor:ERRHom} $\RcHom_{\sfE^{\R, f}_{\tilde{Y}}}(\cM^{\R, f}, \cL^{\R, f})$ as a sheaf of complexes of $\C$-vector spaces is constructible, which implies the result.
    \item Unpacking the notation as in part (2) the relevant claim follows from \cref{prop:ERHomperverse}.
    \item Consider the full subcategory $\RH_{S, L}=\rho_*\Mod_{\Lambda, \rh}(\sfE_Y)\subset \RH_S|_L$. Then by \cref{thm:simpleclassification} we have an equivalence
    \[\mu_\Lambda^{dR}\colon \sfRH_{S, L}\xrightarrow{\sim} \gamma_*\Mod_{\locsys}(\cD^{\sqrt{v}}_{L\times \C^\times}).\]
    \item Consider a contactification $\rho\colon Y\rightarrow S$ of a neighborhood of $L\cup M$ and the corresponding Legendrian submanifolds $\Lambda_L,\Lambda_M\subset S$ lifting $L, M\subset S$. We will identify $\cD_L^{\sqrt{v}}$-modules with $\cD_L$-modules using the orientation data $K_L^{1/2}$. Then by construction we have
    \begin{equation}\label{eq:Lmicrolocal}
    i_L^*\mu_{\Lambda_L}^{dR}(\sfM^\lambda_L)\cong \cO_L\otimes \C_\lambda.
    \end{equation}
    Moreover, by \cref{cor:mainmicrolocalization} we have
    \begin{equation}\label{eq:Mmicrolocal}
    i_L^*\mu_{\Lambda_L}^{dR}(\sfM^\mu_M)\cong \phi^\cD_{L\cap M}\otimes \C_\mu.
    \end{equation}
    We have
    \[\underline{\cHom}_{\sfDRH_S}(\sfM^{\mu}_M, \sfM^{\lambda}_L) = \RcHom_{\sfE^{\R, f}_{\tilde{Y}}}(\sfM^{\mu, \R, f}_M, \sfM^{\lambda, \R, f}_L)|_{L\times \{1\}}[-1].\]
    Plugging in \eqref{eq:Lmicrolocal} and \eqref{eq:Mmicrolocal} into \cref{prop:ERHommicrolocalization} we have
    \[\RcHom_{\sfE^{\R, f}_{\tilde{Y}}}(\sfM^{\mu, \R, f}_M, \sfM^{\lambda, \R, f}_L)|_{L\times\{1\}}[-1]\cong \RcHom_{\cD_L}(\phi^\cD_{L\cap M}\otimes \C_\mu, \cO_L\otimes \C_\lambda).\]
    Using the Verdier self-duality of $\phi^\cD_{L\cap M}$ (see \cite[Equation (6.5)]{BBDJS}) we get the result.
\end{enumerate}
\end{proof}

\begin{remark}
For $\cM, \cL\in\sfDRH_S$ we have
\[\cHom_{\sfDRH_S}(\cM, \cL)\cong \Hom_{\C[T^{\pm 1}]}(\C_0, \underline{\cHom}_{\sfDRH_S}(\cM, \cL))[1],\]
where $\C_0$ is the one-dimensional $\C[T^{\pm 1}]$-module, $T$ acts by the identity, and the $\Hom$ on the right is implicitly derived. The authors of \cite{DAgnoloKashiwara} consider the underlying homotopy category of $\sfDRH_S$ (as well as $\RH_S$) considered as a $\C$-linear category via $\cHom$.
\end{remark}

\subsection{Microlocal Riemann--Hilbert correspondence}

For a contact manifold $Y$ we denote by $\mush_{\tilde{Y}}$ the sheaf of $\C$-linear $\infty$-categories of microsheaves on $\tilde{Y}$ constructed in \cite{NadlerShende,CKNS1}. We denote by $\muPerv_Y\subset \gamma_*\mush_{\tilde{Y}}$ the subcategory of perverse microsheaves constructed in \cite{Waschkies1,CKNS1}. Note that we view $\muPerv_Y$ as a sheaf of abelian categories on $Y$. The microlocal Riemann--Hilbert equivalence established in \cite{Andronikof3,Waschkies2,CKNS2} produces an equivalence
\[\muRH\colon \muPerv^{\op}_Y\longrightarrow \Mod_{\rh}(\sfE_Y).\]

\begin{prop}\label{prop:muPervEcomparison}
Let $\cF, \cG$ be two perverse microsheaves on $Y$. Then there is an isomorphism
\[\cHom_{\mush_{\tilde{Y}}}(\cF, \cG)\cong \RcHom_{\sfE^{\R, f}_{\tilde{Y}}}(\muRH(\cG)^{\R, f}, \muRH(\cF)^{\R, f})\]
of complexes of sheaves on $\tilde{Y}$.
\end{prop}
\begin{proof}
Consider a cover of $Y$ by open subsets $V_i\subset \P^* X_i$, where $X_i$ is a complex manifold. The right-hand side of the required isomorphism is a shift of a perverse sheaf by \cref{prop:ERHomperverse}. Therefore, it is enough to construct the isomorphism on each $V_i$ and check that these isomorphisms are compatible on each double intersection $V_{ij} = V_i\cap V_j$. As before, denote $\tilde{V}_i = \gamma^{-1}(V_i)\subset \tilde{Y}$ and $\tilde{V}_{ij} = \gamma^{-1}(V_{ij})\subset \tilde{Y}$. To simplify the notation, we assume that each $X_i$ is equipped with a choice of the square root line bundle $K^{1/2}_{X_i}$.

Consider the algebra $\cE^{\sqrt{v},\R}_{X_i}\supset \cE^{\sqrt{v},\R, f}_{X_i}$ of holomorphic microlocal operators constructed in \cite{SatoKawaiKashiwara}. For a $\cE^{\sqrt{v}}_{X_i}$-module $\cM$ let
\[\cM^\R = \cE^{\sqrt{v},\R}_{X_i}\otimes_{\gamma^{-1}\cE^{\sqrt{v}}_{X_i}} \gamma^{-1}\cM\]
be the corresponding $\cE^{\sqrt{v},\R}_{X_i}$-module. Consider also the microsheaf $\mu K^{1/2}_{X_i}\in\mush_{\oT^* X_i}$ obtained by microlocalizing the sheaf of holomorphic sections of $K_{X_i}^{1/2}$. It is naturally a $\cE^{\sqrt{v},\R}_{X_i}$-module. To simplify the notation, we will drop certain obvious restrictions to open subsets.

By \cite[Example 5.3]{CKNS1} on $V_{ij}$ we may obtain a quantized contact transformation given by a microsheaf kernel
\[\cK_{ij}\in\mush_{\overline{\tilde{V_i}}\times \tilde{V_j}}\]
and by \cite[Theorem 3.4.2]{Waschkies2}, possibly refining the cover $\{V_i\}$, we may choose a nondegenerate section of the $(\cE^{\sqrt{v},\R}_{X_j}, \cE^{\sqrt{v},\R}_{X_i})$-bimodule
\[\cH_{ij} = \cHom_{\mush_{\overline{\tilde{V_i}}\times \tilde{V_j}}}(\cK_{ij}[-\dim X_i], K^{1/2}_{X_i}\boxtimes K^{1/2}_{X_j})\]
which provides isomorphisms
\[\cE^{\sqrt{v},\R}_{X_j}\xrightarrow{\sim} \cE^{\sqrt{v},\R}_{X_i}\]
of algebras on $\tilde{V}_{ij}$ and
\[s_{ij}\colon \cK_{ij}\circ \mu K^{1/2}_{X_i}\xrightarrow{\sim} \mu K^{1/2}_{X_j}\]
of $\cE^{\sqrt{v},\R}_{X_j}$-modules on $\tilde{V}_{ij}$.

Using quantized contact transformations we may glue the sheaves of algebras $\cE^{\sqrt{v}, \R}_{X_i}$ into an algebroid $\sfE^\R_{\tilde{Y}}$ on $\tilde{Y}$. Moreover, for any two regular holonomic $\sfE_Y$-modules $\cM_1, \cM_2$, the natural morphism
\[\RcHom_{\sfE^{\R, f}_{\tilde{Y}}}(\cM_1^{\R, f}, \cM_2^{\R, f})\longrightarrow \RcHom_{\sfE^\R_{\tilde{Y}}}(\cM_1^\R, \cM_2^\R)\]
is an isomorphism: the claim is local, so it reduces to the same claim on $V_i$, where it is shown in \cite[Corollaire 5.6.4]{Andronikof1}. Thus, we have to construct an isomorphism
\[\cHom_{\mush_{\tilde{Y}}}(\cF, \cG)\cong \RcHom_{\sfE^\R_{\tilde{Y}}}(\muRH(\cG)^\R, \muRH(\cF)^\R).\]

The perverse microsheaves $\cF,\cG\in\muPerv_Y$ are given in terms of the cover $\{V_i\}$ as follows:
\begin{itemize}
    \item On $V_i$ we have objects $\cF_i,\cG_i\in\mush_{\oT^* X_i}$.
    \item On $V_{ij}$ we are given isomorphisms $f_{ij}\colon \cK_{ij}\circ \cF_i\xrightarrow{\sim} \cF_j$ and $g_{ij}\colon \cK_{ij}\circ \cG_i\xrightarrow{\sim} \cG_j$.
\end{itemize}

By definition (see \cite[Theorem 7.7]{CKNS2}) we may describe the $\sfE^\R_{\tilde{Y}}$-module $\muRH(\cF)^\R$ (and similarly for $\muRH(\cG)^\R$) in terms of the cover $\{V_i\}$ as follows:
\begin{itemize}
    \item On $V_i$ we have the $\cE^{\sqrt{v},\R}_{X_i}$-module $\cHom_{\mush_{\oT^* X_i}}(\cF_i, \mu K^{1/2}_{X_i})$.
    \item On $V_{ij}$ we are given the isomorphism
    \begin{align*}
    \cHom_{\mush_{\oT^* X_i}}(\cF_i, \mu K^{1/2}_{X_i})&\xrightarrow{\sim}\cHom_{\mush_{\oT^* X_j}}(\cK_{ij}\circ \cF_i, \cK_{ij}\circ \mu K^{1/2}_{X_i})\\
    &\xrightarrow{(f_{ij}, s_{ij})} \cHom_{\mush_{\oT^* X_j}}(\cF_j, \mu K^{1/2}_{X_j}).
    \end{align*}
\end{itemize}

We are ready to construct the required isomorphism of (shifted) perverse sheaves. On $V_i$ we have a natural morphism of complexes of sheaves
\[
\xymatrix{
\cHom_{\mush_{\oT^* X_i}}(\cF_i, \cG_i) \ar[d] \\
\RcHom_{\cE^\R_{X_i}}(\cHom_{\mush_{\oT^* X_i}}(\cG_i, \mu K^{1/2}_{X_i}), \cHom_{\mush_{\oT^* X_i}}(\cF_i, \mu K^{1/2}_{X_i}))
}
\]
given by the functoriality of $\cHom_{\mush_{\oT^* X_i}}(-, \mu K^{1/2}_{X_i})$. The claim that this morphism is an isomorphism is shown in the proof of \cite[Theorem 3.6.5]{Waschkies2} by reducing to the case of ordinary perverse sheaves on $X_i$ using a generic position argument. The compatibility of these isomorphisms on $V_{ij}$ follows from the commutative diagram
\[
\adjustbox{scale=0.8,center}{
\begin{tikzcd}
\cHom(\cF_i, \cG_i) \ar{r} \ar{d} & \RcHom_{\cE^\R_{X_i}}(\cHom(\cG_i, \mu K^{1/2}_{X_i}), \cHom(\cF_i, \mu K^{1/2}_{X_i})) \ar{d} \\
\cHom(\cK_{ij}\circ \cF_i, \cK_{ij}\circ\cG_i) \ar{r} \ar{d}{(f_{ij}, g_{ij})} & \RcHom_{\cE^\R_{X_i}}(\cHom(\cK_{ij}\circ\cG_i, \cK_{ij}\circ\mu K^{1/2}_{X_i}), \cHom(\cK_{ij}\circ\cF_i, \cK_{ij}\circ \mu K^{1/2}_{X_i})) \ar{d}{(f_{ij}, g_{ij})} \\
\cHom(\cF_j, \cG_j) \ar{r} \ar{dr} & \RcHom_{\cE^\R_{X_j}}(\cHom(\cG_j, \cK_{ij}\circ\mu K^{1/2}_{X_i}), \cHom(\cF_j, \cK_{ij}\circ\mu K^{1/2}_{X_j})) \ar{d}{s_{ij}} \\
& \RcHom_{\cE^\R_{X_j}}(\cHom(\cG_j, \mu K^{1/2}_{X_j}), \cHom(\cF_j, \mu K^{1/2}_{X_j})).
\end{tikzcd}}
\]
\end{proof}

Now let $\Lambda\subset Y$ be a Legendrian submanifold equipped with a square root line bundle $K_\Lambda^{1/2}$. In particular, the corresponding homogeneous Lagrangian $\tilde{\Lambda}\subset \tilde{Y}$ in the symplectization is equipped with a square root line bundle $K_{\tilde{\Lambda}}^{1/2}$ and so, by the classification of simple $\sfE_Y$-modules given in \cref{thm:simpleclassification}, there is a simple $\sfE_Y$-module supported $\sfM_\Lambda$ on $\Lambda$ with
\[\mu_\Lambda^{dR}(\sfM_\Lambda) \cong K_{\tilde{\Lambda}}^{1/2}\]
the trivial local system. Let $\cM_\Lambda\in\muPerv_Y$ be the perverse microsheaf such that
\[\muRH(\cM_\Lambda) \cong \sfM_\Lambda.\]

\begin{thm}\label{thm:microsheafHom}
Let $Y$ be a contact manifold and $\Lambda_1, \Lambda_2\subset Y$ two Legendrian submanifolds equipped with square root line bundles $K_{\Lambda_i}^{1/2}$. There is an isomorphism
\[\cHom_{\mush_{\tilde{Y}}}(\cM_{\Lambda_1}, \cM_{\Lambda_2})[\dim \tilde{Y}/2]\cong \phi_{\tilde{\Lambda}_2\cap \tilde{\Lambda}_1}.\]
\end{thm}
\begin{proof}
We use a dummy variable trick as in \cite[Appendix A]{KashiwaraKawaiHolonomic} to reduce the case of an arbitrary contact manifold to the case of a contactification of a symplectic manifold.

The symplectic manifold $\tilde{Y}$ is exact and so it admits a contactification $\tilde{\rho}\colon \tilde{Y}\times \C\rightarrow \tilde{Y}$ whose symplectization is $\tilde{Y}\times \T^* \C^\times$. The Lagrangian submanifolds $\tilde{\Lambda}_1, \tilde{\Lambda}_2\subset \tilde{Y}$ lift to Legendrians $\tilde{\Lambda}_1\times\{0\}, \tilde{\Lambda}_2\times\{0\}\subset \tilde{Y}\times \C$ and hence lift to homogeneous Lagrangian submanifolds $\tilde{\Lambda}_1\times \C^\times, \tilde{\Lambda}_2\times \C^\times\subset \tilde{Y}\times \T^* \C^\times$.

The shifted constant local system $\C_{\C^\times}[1]$ on $\C^\times$ defines a perverse microsheaf on $\T^* \C^\times$ and
\[\cM_{\tilde{\Lambda}_i\times \C} \cong \cM_{\Lambda_i}\boxtimes \C_{\C^\times}[1].\]
Moreover, we have
\[\cHom_{\mush_{\tilde{Y}}}(\cM_{\Lambda_1}, \cM_{\Lambda_2})\cong \cHom_{\mush_{\tilde{Y}\times \T^* \C^\times}}(\cM_{\tilde{\Lambda}_1\times \C}, \cM_{\tilde{\Lambda}_2\times \C})|_{\tilde{Y}}[-1].\]
By \cref{prop:muPervEcomparison}
\[\cHom_{\mush_{\tilde{Y}\times\T^*\C^\times}}(\cM_{\tilde{\Lambda}_1\times\C}, \cM_{\tilde{\Lambda}_2\times\C})\cong \RcHom_{\sfE^{\R, f}_{\tilde{Y}\times\T^*\C^\times}}(\sfM_{\tilde{\Lambda}_2\times\C^\times}, \sfM_{\tilde{\Lambda}_1\times\C^\times}).\]

By \cref{thm:RHproperties}(5) we obtain an isomorphism
\[\RcHom_{\sfE^{\R, f}_{\tilde{Y}\times\T^*\C^\times}}(\sfM_{\tilde{\Lambda}_2\times\C^\times}, \sfM_{\tilde{\Lambda}_1\times\C^\times})[\dim\tilde{Y}/2+1]\cong \phi_{\tilde{\Lambda}_2\cap \tilde{\Lambda}_1}\]
which concludes the proof.
\end{proof}


\printbibliography

@article{KontsevichSoibelman,
	archiveprefix = {arXiv},
	author = {Kontsevich, Maxim and Soibelman, Yan},
	eprint = {2402.07343},
	primaryclass = {math.SG},
	title = {Holomorphic Floer theory I: exponential integrals in finite and infinite dimensions},
	year = {2024}}

@article{BehrendFantechi,
	author = {Behrend, Kai and Fantechi, Barbara},
    pages = {1--47},
    publisher = {Birkh\"auser Boston, Boston, MA},
    year = {2009},
    series = {Progr. Math.},
    volume = {269},
    title = {Gerstenhaber and {B}atalin-{V}ilkovisky structures on
              {L}agrangian intersections},
    booktitle = {Algebra, arithmetic, and geometry: in honor of {Y}u. {I}.
              {M}anin. {V}ol. {I}}}

@article{Mladenov,
	author = {Mladenov, Borislav},
    archiveprefix = {arXiv},
    eprint = {2007.05498},
    primaryclass = {math.AG},
	fjournal = {Selecta Mathematica. New Series},
	journal = {Selecta Math. (N.S.)},
	number = {1},
	pages = {Paper No. 8, 42},
	title = {Formality of differential graded algebras and complex {L}agrangian submanifolds},
	volume = {30},
	year = {2024}}

@article{Andronikof3,
	author = {Andronikof, Emmanuel},
	doi = {10.12775/TMNA.1994.036},
	fjournal = {Topological Methods in Nonlinear Analysis},
	issn = {1230-3429},
	journal = {Topol. Methods Nonlinear Anal.},
	mrclass = {32C38 (58G07)},
	mrnumber = {1350980},
	mrreviewer = {Andrea\ D'Agnolo},
	number = {2},
	pages = {417--425},
	title = {A microlocal version of the {R}iemann-{H}ilbert correspondence},
	volume = {4},
	year = {1994}}

@article{GepnerHaugseng,
	archiveprefix = {arXiv},
	author = {Gepner, David and Haugseng, Rune},
	doi = {10.1016/j.aim.2015.02.007},
	eprint = {1312.3178},
	fjournal = {Advances in Mathematics},
	issn = {0001-8708,1090-2082},
	journal = {Adv. Math.},
	mrclass = {18D20 (18D10 18D50)},
	mrnumber = {3345192},
	mrreviewer = {Christopher\ L.\ Rogers},
	pages = {575--716},
	primaryclass = {math.AT},
	title = {Enriched {$\infty$}-categories via non-symmetric {$\infty$}-operads},
	volume = {279},
	year = {2015}}

@article{Andronikof1,
	author = {Andronikof, Emmanuel},
	fjournal = {M\'emoires de la Soci\'et\'e{} Math\'ematique de France. Nouvelle S\'erie},
	issn = {0037-9484},
	journal = {M\'em. Soc. Math. France (N.S.)},
	mrclass = {58G07},
	mrnumber = {1273991},
	number = {57},
	pages = {176},
	title = {Microlocalisation temp\'er\'ee},
	year = {1994}}

@article{KashiwaraVilonen,
	archiveprefix = {arXiv},
	author = {Kashiwara, Masaki and Vilonen, Kari},
	doi = {10.4007/annals.2014.180.2.4},
	eprint = {1209.5124},
	fjournal = {Annals of Mathematics. Second Series},
	issn = {0003-486X,1939-8980},
	journal = {Ann. of Math. (2)},
	mrclass = {32C38 (32G34)},
	mrnumber = {3224719},
	mrreviewer = {Andrea\ D'Agnolo},
	number = {2},
	pages = {573--620},
	primaryclass = {math.AG},
	title = {Microdifferential systems and the codimension-three conjecture},
	volume = {180},
	year = {2014}}

@article{KuwagakiPetrShende,
	archiveprefix = {arXiv},
	author = {Kuwagaki, Tatsuki and Petr, Adrian and Shende, Vivek},
	eprint = {2406.08852},
	primaryclass = {math.SG},
	title = {On Fukaya categories and prequantization bundles},
	year = {2024}}

@article{CKNS2,
	archiveprefix = {arXiv},
	author = {C\^ot\'e, Laurent and Kuo, Christopher and Nadler, David and Shende, Vivek},
	eprint = {2406.16222},
	primaryclass = {math.SG},
	title = {The microlocal Riemann-Hilbert correspondence for complex contact manifolds},
	year = {2024}}

@article{GPS,
	archiveprefix = {arXiv},
	author = {Ganatra, Sheel and Pardon, John and Shende, Vivek},
	doi = {10.4007/annals.2024.199.3.1},
	eprint = {1809.08807},
	fjournal = {Annals of Mathematics. Second Series},
	issn = {0003-486X,1939-8980},
	journal = {Ann. of Math. (2)},
	mrclass = {53D37 (14J33 53D40)},
	mrnumber = {4740209},
	number = {3},
	pages = {943--1042},
	primaryclass = {math.SG},
	title = {Microlocal {M}orse theory of wrapped {F}ukaya categories},
	volume = {199},
	year = {2024}}

@article{SolomonVerbitsky,
	archiveprefix = {arXiv},
	author = {Solomon, Jake P. and Verbitsky, Misha},
	doi = {10.1112/s0010437x1900753x},
	eprint = {1805.00102},
	fjournal = {Compositio Mathematica},
	issn = {0010-437X,1570-5846},
	journal = {Compos. Math.},
	mrclass = {53C26 (32B20 53C38 53D12 53D37)},
	mrnumber = {4010429},
	mrreviewer = {Ljudmila\ Kamenova},
	number = {10},
	pages = {1924--1958},
	primaryclass = {math.SG},
	title = {Locality in the {F}ukaya category of a hyperk\"{a}hler manifold},
	volume = {155},
	year = {2019}}

@article{DoanRezchikov,
	archiveprefix = {arXiv},
	author = {Doan, Aleksander and Rezchikov, Semon},
	eprint = {2210.12047},
	primaryclass = {math.SG},
	title = {Holomorphic Floer Theory and the Fueter Equation},
	year = {2022}}

@article{JoyceSafronov,
	archiveprefix = {arXiv},
	author = {Joyce, Dominic and Safronov, Pavel},
	doi = {10.5802/afst.1616},
	eprint = {1506.04024},
	fjournal = {Annales de la Facult\'{e} des Sciences de Toulouse. Math\'{e}matiques. S\'{e}rie 6},
	issn = {0240-2963,2258-7519},
	journal = {Ann. Fac. Sci. Toulouse Math. (6)},
	mrclass = {14A30 (14F08 14J33 53D12)},
	mrnumber = {4093980},
	mrreviewer = {Richard\ P.\ Thomas},
	number = {5},
	pages = {831--908},
	primaryclass = {math.AG},
	title = {A {L}agrangian neighbourhood theorem for shifted symplectic derived schemes},
	volume = {28},
	year = {2019}}

@article{DKS,
	archiveprefix = {arXiv},
	author = {Detcherry, Renaud and Kalfagianni, Efstratia and Sikora, Adam},
	eprint = {2305.16188},
	primaryclass = {math.GT},
	title = {Kauffman bracket skein modules of small 3-manifolds},
	year = {2023}}

@article{AbouzaidManolescu,
	archiveprefix = {arXiv},
	author = {Abouzaid, Mohammed and Manolescu, Ciprian},
	doi = {10.4171/jems/994},
	eprint = {1708.00289},
	fjournal = {Journal of the European Mathematical Society (JEMS)},
	issn = {1435-9855,1435-9863},
	journal = {J. Eur. Math. Soc. (JEMS)},
	mrclass = {57K31 (53D40 57K18 57R58)},
	mrnumber = {4167016},
	number = {11},
	pages = {3641--3695},
	primaryclass = {math.GT},
	title = {A sheaf-theoretic model for {${\rm SL}(2,\mathbb{C})$} {F}loer homology},
	volume = {22},
	year = {2020}}

@article{GJS,
	archiveprefix = {arXiv},
	author = {Gunningham, Sam and Jordan, David and Safronov, Pavel},
	doi = {10.1007/s00222-022-01167-0},
	eprint = {1908.05233},
	fjournal = {Inventiones Mathematicae},
	issn = {0020-9910,1432-1297},
	journal = {Invent. Math.},
	mrclass = {57K31},
	mrnumber = {4557403},
	number = {1},
	pages = {301--363},
	primaryclass = {math.QA},
	title = {The finiteness conjecture for skein modules},
	volume = {232},
	year = {2023}}

@incollection{Andronikof2,
	author = {Andronikov, E.},
	doi = {10.1007/BF02362631},
	fjournal = {Journal of Mathematical Sciences},
	issn = {1072-3374},
	journal = {J. Math. Sci.},
	mrclass = {32C38 (58G07)},
	mrnumber = {1431535},
	mrreviewer = {Andrea\ D'Agnolo},
	note = {Algebra, 3},
	number = {6},
	pages = {3754--3758},
	title = {Microlocalization of perverse sheaves},
	volume = {82},
	year = {1996}}

@article{Waschkies2,
	author = {Waschkies, Ingo},
	fjournal = {Kyoto University. Research Institute for Mathematical Sciences. Publications},
	issn = {0034-5318,1663-4926},
	journal = {Publ. Res. Inst. Math. Sci.},
	mrclass = {32C38 (32S60)},
	mrnumber = {2115967},
	mrreviewer = {St\'{e}phane\ Guillermou},
	number = {1},
	pages = {37--72},
	title = {Microlocal {R}iemann-{H}ilbert correspondence},
	volume = {41},
	year = {2005}}

@article{Waschkies1,
	author = {Waschkies, Ingo},
	doi = {10.24033/bsmf.2469},
	fjournal = {Bulletin de la Soci\'{e}t\'{e} Math\'{e}matique de France},
	issn = {0037-9484,2102-622X},
	journal = {Bull. Soc. Math. France},
	mrclass = {32S60 (18D05)},
	mrnumber = {2081221},
	mrreviewer = {Corrado\ Marastoni},
	number = {3},
	pages = {397--462},
	title = {The stack of microlocal perverse sheaves},
	volume = {132},
	year = {2004}}

@article{CKNS1,
	archiveprefix = {arXiv},
	author = {C\^ot\'e, Laurent and Kuo, Christopher and Nadler, David and Shende, Vivek},
	eprint = {2209.12998},
	primaryclass = {math.SG},
	title = {Perverse Microsheaves},
	year = {2022}}

@article{NadlerShende,
	archiveprefix = {arXiv},
	author = {Nadler, David and Shende, Vivek},
	eprint = {2007.10154},
	primaryclass = {math.SG},
	title = {Sheaf quantization in Weinstein symplectic manifolds},
	year = {2020}}

@book{BanicaStanasila,
	author = {B\u{a}nic\u{a}, Constantin and St\u{a}n\u{a}\c{s}il\u{a}, Octavian},
	mrclass = {32-02 (32C35 32D15)},
	mrnumber = {463470},
	mrreviewer = {J.\ S.\ Joel},
	note = {Translated from the Romanian},
	pages = {296},
	publisher = {Editura Academiei, Bucharest; John Wiley \& Sons, London-New York-Sydney},
	title = {Algebraic methods in the global theory of complex spaces},
	year = {1976}}

@unpublished{SabbahSchnell,
	author = {Sabbah, Claude and Schnell, Christian},
	title = {The MHM Project},
	url = {https://perso.pages.math.cnrs.fr/users/claude.sabbah/MHMProject/mhm.pdf}}

@article{Schefers,
	archiveprefix = {arXiv},
	author = {Schefers, Kendric},
	eprint = {2310.09979},
	primaryclass = {math.AG},
	title = {Derived $V$-filtrations and the Kontsevich-Sabbah-Saito theorem},
	year = {2023}}

@article{KashiwaraKawaiMicrolocal,
	author = {Kashiwara, Masaki and Kawai, Takahiro},
	doi = {10.2977/prims/1195182018},
	fjournal = {Kyoto University. Research Institute for Mathematical Sciences. Publications},
	issn = {0034-5318},
	journal = {Publ. Res. Inst. Math. Sci.},
	mrclass = {58G07 (32C38 35A27)},
	mrnumber = {723458},
	mrreviewer = {Zoghman Mebkhout},
	number = {3},
	pages = {1003--1032},
	title = {Microlocal analysis},
	volume = {19},
	year = {1983}}

@incollection{MonteiroFernandes,
	author = {Monteiro Fernandes, Teresa},
	booktitle = {Singularities and differential equations ({W}arsaw, 1993)},
	mrclass = {58G07 (32C38 35A27)},
	mrnumber = {1449160},
	mrreviewer = {Andrea D'Agnolo},
	pages = {223--233},
	publisher = {Polish Acad. Sci. Inst. Math., Warsaw},
	series = {Banach Center Publ.},
	title = {Some functorial properties of microlocalization for {$\mathscr{D}$}-modules},
	volume = {33},
	year = {1996}}

@incollection{Malgrange,
	author = {Malgrange, B.},
	booktitle = {Analysis and topology on singular spaces, {II}, {III} ({L}uminy, 1981)},
	mrclass = {58G07 (32B10 32C40)},
	mrnumber = {737934},
	mrreviewer = {M. Sebastiani},
	pages = {243--267},
	publisher = {Soc. Math. France, Paris},
	series = {Ast\'{e}risque},
	title = {Polyn\^{o}mes de {B}ernstein-{S}ato et cohomologie \'{e}vanescente},
	volume = {101},
	year = {1983}}

@inproceedings{KashiwaraKawaiMicrolocalization,
	author = {Kashiwara, Masaki and Kawai, Takahiro},
	booktitle = {Complex analysis, microlocal calculus and relativistic quantum theory ({P}roc. {I}nternat. {C}olloq., {C}entre {P}hys., {L}es {H}ouches, 1979)},
	mrclass = {58G07 (46F15)},
	mrnumber = {579740},
	mrreviewer = {P. Schapira},
	pages = {21--76},
	publisher = {Springer, Berlin-New York},
	series = {Lecture Notes in Phys.},
	title = {Second-microlocalization and asymptotic expansions},
	volume = {126},
	year = {1980}}

@article{KashiwaraKawaiHolonomic,
	author = {Kashiwara, Masaki and Kawai, Takahiro},
	doi = {10.2977/prims/1195184396},
	fjournal = {Kyoto University. Research Institute for Mathematical Sciences. Publications},
	issn = {0034-5318},
	journal = {Publ. Res. Inst. Math. Sci.},
	mrclass = {58G15 (14D05 32C40)},
	mrnumber = {650216},
	mrreviewer = {P. Schapira},
	number = {3},
	pages = {813--979},
	title = {On holonomic systems of microdifferential equations. {III}. {S}ystems with regular singularities},
	volume = {17},
	year = {1981}}

@article{KashiwaraIntroduction,
	author = {Kashiwara, Masaki},
	fjournal = {L'Enseignement Math\'{e}matique. Revue Internationale. 2e S\'{e}rie},
	issn = {0013-8584},
	journal = {Enseign. Math. (2)},
	mrclass = {58G07 (32C38)},
	mrnumber = {879516},
	number = {3-4},
	pages = {227--259},
	title = {Introduction to microlocal analysis},
	volume = {32},
	year = {1986}}

@article{Giesecke,
	author = {Giesecke, Burghart},
	doi = {10.1007/BF01111199},
	fjournal = {Mathematische Zeitschrift},
	issn = {0025-5874},
	journal = {Math. Z.},
	mrclass = {32.44 (57.60)},
	mrnumber = {159346},
	mrreviewer = {A. H. Wallace},
	pages = {177--213},
	title = {Simpliziale {Z}erlegung abz\"{a}hlbarer analytischer {R}\"{a}ume},
	volume = {83},
	year = {1964}}

@article{DAgnoloKashiwara,
	archiveprefix = {arXiv},
	author = {D'Agnolo, Andrea and Kashiwara, Masaki},
	doi = {10.1007/s00220-011-1325-7},
	eprint = {1008.5273},
	fjournal = {Communications in Mathematical Physics},
	issn = {0010-3616},
	journal = {Comm. Math. Phys.},
	mrclass = {53D55 (18E30 32C38)},
	mrnumber = {2842971},
	mrreviewer = {Ana Rita Martins},
	number = {1},
	pages = {81--113},
	primaryclass = {math.AG},
	title = {On quantization of complex symplectic manifolds},
	volume = {308},
	year = {2011}}

@article{WeinsteinLagrangian,
	author = {Weinstein, Alan},
	doi = {10.1016/0001-8708(71)90020-X},
	fjournal = {Advances in Mathematics},
	issn = {0001-8708},
	journal = {Advances in Math.},
	mrclass = {57.50},
	mrnumber = {286137},
	mrreviewer = {D. G. Ebin},
	pages = {329--346 (1971)},
	title = {Symplectic manifolds and their {L}agrangian submanifolds},
	volume = {6},
	year = {1971}}

@book{Dimca,
	author = {Dimca, Alexandru},
	doi = {10.1007/978-3-642-18868-8},
	isbn = {3-540-20665-5},
	mrclass = {55N30 (14F43 32S60 55N25)},
	mrnumber = {2050072},
	mrreviewer = {Ana Jerem\'{\i}as L\'{o}pez},
	pages = {xvi+236},
	publisher = {Springer-Verlag, Berlin},
	series = {Universitext},
	title = {Sheaves in topology},
	year = {2004}}

@article{PSY,
	archiveprefix = {arXiv},
	author = {Porta, Marco and Shaul, Liran and Yekutieli, Amnon},
	doi = {10.1007/s10468-012-9385-8},
	eprint = {1010.4386},
	fjournal = {Algebras and Representation Theory},
	issn = {1386-923X},
	journal = {Algebr. Represent. Theory},
	mrclass = {13D07 (13B35 13C12 13D09)},
	mrnumber = {3160712},
	mrreviewer = {Toma Albu},
	number = {1},
	pages = {31--67},
	primaryclass = {math.AC},
	title = {On the homology of completion and torsion},
	volume = {17},
	year = {2014}}

@article{DAgnoloGuillermouSchapira,
	archiveprefix = {arXiv},
	author = {D'Agnolo, Andrea and Guillermou, St\'{e}phane and Schapira, Pierre},
	doi = {10.2977/PRIMS/35},
	eprint = {0906.2175},
	fjournal = {Publications of the Research Institute for Mathematical Sciences},
	issn = {0034-5318},
	journal = {Publ. Res. Inst. Math. Sci.},
	mrclass = {32C38 (16D90 53D55)},
	mrnumber = {2827727},
	mrreviewer = {Ana Rita Martins},
	number = {1},
	pages = {221--255},
	primaryclass = {math.AG},
	title = {Regular holonomic {$\mathcal{D}[[\hslash]]$}-modules},
	volume = {47},
	year = {2011}}

@book{KashiwaraSchapiraSheaves,
	author = {Kashiwara, Masaki and Schapira, Pierre},
	isbn = {3-540-51861-4},
	mrclass = {58G07 (18F20 32C38 35A27)},
	mrnumber = {1299726},
	note = {With a chapter in French by Christian Houzel, Corrected reprint of the 1990 original},
	pages = {x+512},
	publisher = {Springer-Verlag, Berlin},
	series = {Grundlehren der mathematischen Wissenschaften [Fundamental Principles of Mathematical Sciences]},
	title = {Sheaves on manifolds},
	volume = {292},
	year = {1994}}

@article{MaximSaitoSchurmann,
	archiveprefix = {arXiv},
	author = {Maxim, Laurentiu and Saito, Morihiko and Sch\"{u}rmann, J\"{o}rg},
	doi = {10.1093/imrn/rny032},
	eprint = {1610.07295},
	fjournal = {International Mathematics Research Notices. IMRN},
	issn = {1073-7928},
	journal = {Int. Math. Res. Not. IMRN},
	mrclass = {32C38 (14F18 32S35)},
	mrnumber = {4050564},
	mrreviewer = {Andrea D'Agnolo},
	number = {1},
	pages = {91--111},
	primaryclass = {math.AG},
	title = {Thom-{S}ebastiani theorems for filtered {$\mathcal{D}$}-modules and for multiplier ideals},
	year = {2020}}

@unpublished{SaitoTS,
	author = {Saito, Morihiko},
	note = {preprint},
	title = {Thom-Sebastiani theorem for Hodge modules},
	year = {2010}}

@article{Sabbah,
	archiveprefix = {arXiv},
	author = {Sabbah, Claude},
	eprint = {1012.3818},
	primaryclass = {math.AG},
	title = {On a twisted de Rham complex, II},
	year = {2010}}

@article{SabbahSaito,
	archiveprefix = {arXiv},
	author = {Sabbah, Claude and Saito, Morihiko},
	doi = {10.14231/AG-2014-006},
	eprint = {1212.0436},
	fjournal = {Algebraic Geometry},
	issn = {2313-1691},
	journal = {Algebr. Geom.},
	mrclass = {14F10 (14F17)},
	mrnumber = {3234116},
	mrreviewer = {Uli Walther},
	number = {1},
	pages = {107--130},
	primaryclass = {math.AG},
	title = {Kontsevich's conjecture on an algebraic formula for vanishing cycles of local systems},
	volume = {1},
	year = {2014}}

@incollection{KashiwaraVanishing,
	author = {Kashiwara, M.},
	booktitle = {Algebraic geometry ({T}okyo/{K}yoto, 1982)},
	doi = {10.1007/BFb0099962},
	mrclass = {58G05 (14D05 32C38)},
	mrnumber = {726425},
	mrreviewer = {P. Schapira},
	pages = {134--142},
	publisher = {Springer, Berlin},
	series = {Lecture Notes in Math.},
	title = {Vanishing cycle sheaves and holonomic systems of differential equations},
	volume = {1016},
	year = {1983}}

@article{BBDJS,
	archiveprefix = {arXiv},
	author = {Brav, C. and Bussi, V. and Dupont, D. and Joyce, D. and Szendr\H{o}i, B.},
	doi = {10.5427/jsing.2015.11e},
	eprint = {1211.3259},
	fjournal = {Journal of Singularities},
	journal = {J. Singul.},
	mrclass = {14F05 (14C25)},
	mrnumber = {3353002},
	mrreviewer = {Shintarou Yanagida},
	note = {With an appendix by J\"{o}rg Sch\"{u}rmann},
	pages = {85--151},
	primaryclass = {math.AG},
	title = {Symmetries and stabilization for sheaves of vanishing cycles},
	volume = {11},
	year = {2015}}

@article{JoycedCrit,
	archiveprefix = {arXiv},
	author = {Joyce, Dominic},
	eprint = {1304.4508},
	fjournal = {Journal of Differential Geometry},
	issn = {0022-040X},
	journal = {J. Differential Geom.},
	mrclass = {14F05 (14A15 32C99 32Q99)},
	mrnumber = {3399099},
	mrreviewer = {Antony Maciocia},
	number = {2},
	pages = {289--367},
	primaryclass = {math.AG},
	title = {A classical model for derived critical loci},
	volume = {101},
	year = {2015}}

@book{KashiwaraDmodules,
	author = {Kashiwara, Masaki},
	doi = {10.1090/mmono/217},
	isbn = {0-8218-2766-9},
	mrclass = {32C38},
	mrnumber = {1943036},
	mrreviewer = {Corrado Marastoni},
	note = {Translated from the 2000 Japanese original by Mutsumi Saito, Iwanami Series in Modern Mathematics},
	pages = {xvi+254},
	publisher = {American Mathematical Society, Providence, RI},
	series = {Translations of Mathematical Monographs},
	title = {{$D$}-modules and microlocal calculus},
	volume = {217},
	year = {2003}}

@inproceedings{SatoKawaiKashiwara,
	author = {Sato, Mikio and Kawai, Takahiro and Kashiwara, Masaki},
	booktitle = {Hyperfunctions and pseudo-differential equations ({P}roc. {C}onf., {K}atata, 1971; dedicated to the memory of {A}ndr\'{e} {M}artineau)},
	mrclass = {58G15 (32C35 35N10)},
	mrnumber = {0420735},
	mrreviewer = {P. Schapira},
	pages = {265--529. Lecture Notes in Math., Vol. 287},
	title = {Microfunctions and pseudo-differential equations},
	year = {1973}}

@article{DAgnoloSchapira,
	archiveprefix = {arXiv},
	author = {D'Agnolo, Andrea and Schapira, Pierre},
	doi = {10.1016/j.aim.2006.12.009},
	eprint = {math/0506064},
	fjournal = {Advances in Mathematics},
	issn = {0001-8708},
	journal = {Adv. Math.},
	mrclass = {32C38 (53D55)},
	mrnumber = {2331247},
	number = {1},
	pages = {358--379},
	title = {Quantization of complex {L}agrangian submanifolds},
	volume = {213},
	year = {2007}}

@article{KashiwaraContact,
	author = {Kashiwara, Masaki},
	doi = {10.2977/prims/1195163179},
	fjournal = {Kyoto University. Research Institute for Mathematical Sciences. Publications},
	issn = {0034-5318},
	journal = {Publ. Res. Inst. Math. Sci.},
	mrclass = {58G07 (53C56)},
	mrnumber = {1384750},
	mrreviewer = {Andrea D'Agnolo},
	number = {1},
	pages = {1--7},
	title = {Quantization of contact manifolds},
	volume = {32},
	year = {1996}}

@article{PoleselloSchapira,
	archiveprefix = {arXiv},
	author = {Polesello, Pietro and Schapira, Pierre},
	doi = {10.1155/S1073792804132819},
	eprint = {math/0305171},
	fjournal = {International Mathematics Research Notices},
	issn = {1073-7928},
	journal = {Int. Math. Res. Not.},
	mrclass = {32C38 (53D55)},
	mrnumber = {2077680},
	mrreviewer = {Corrado Marastoni},
	number = {49},
	pages = {2637--2664},
	title = {Stacks of quantization-deformation modules on complex symplectic manifolds},
	year = {2004}}

@article{Bussi,
	archiveprefix = {arXiv},
	author = {Bussi, Vittoria},
	eprint = {1404.1329},
	primaryclass = {math.AG},
	title = {Categorification of Lagrangian intersections on complex symplectic manifolds using perverse sheaves of vanishing cycles},
	year = {2014}}

@article{KashiwaraSchapira,
	archiveprefix = {arXiv},
	author = {Kashiwara, Masaki and Schapira, Pierre},
	eprint = {1003.3304},
	fjournal = {Ast\'{e}risque},
	isbn = {978-2-85629-345-4},
	issn = {0303-1179},
	journal = {Ast\'{e}risque},
	mrclass = {53D55 (14F43 17B63 53D17)},
	mrnumber = {3012169},
	mrreviewer = {Andrea D'Agnolo},
	number = {345},
	pages = {xii+147},
	primaryclass = {math.AG},
	title = {Deformation quantization modules},
	year = {2012}}

\end{document}